\documentclass[letterpaper, 11pt,  reqno]{amsart}

\usepackage{amsmath,amssymb,amscd,amsthm,amsxtra, esint}

\usepackage{hyperref} 
\allowdisplaybreaks[2]

\newtheorem{theorem}{Theorem} [section]

\newtheorem{lemma}[theorem]{Lemma}
\newtheorem{proposition}[theorem]{Proposition}
\newtheorem{remark}[theorem]{Remark}

\newtheorem{definition}[theorem]{Definition}

\newtheorem*{acknowledgment}{Acknowledgments}

\DeclareMathOperator*{\supp}{supp}

\newcommand{\noi}{\noindent}
\newcommand{\Z}{\mathbb{Z}}
\newcommand{\R}{\mathbb{R}}
\newcommand{\C}{\mathbb{C}}
\newcommand{\T}{\mathbb{T}}

\let\Im=\undefined\DeclareMathOperator*{\Im}{Im}

\let\P= \undefined
\newcommand{\P}{\mathbf{P}}

\newcommand{\N}{\mathcal{N}}

\newcommand{\F}{\mathcal{F}}

\newcommand{\al}{\alpha}
\newcommand{\be}{\beta}
\newcommand{\dl}{\delta}

\newcommand{\nb}{\nabla}

\newcommand{\Dl}{\Delta}
\newcommand{\eps}{\varepsilon}

\newcommand{\g}{\gamma}
\newcommand{\G}{\Gamma}
\newcommand{\ld}{\lambda}

\newcommand{\ft}{\widehat}

\newcommand{\wt}{\widetilde}
\newcommand{\cj}{\overline}

\newcommand{\dt}{\partial_t}

\newcommand{\ta}{\theta}

\renewcommand{\l}{\ell}
\renewcommand{\o}{\omega}
\renewcommand{\O}{\Omega}

\newcommand{\les}{\lesssim}
\newcommand{\ges}{\gtrsim}

\newcommand{\jb}[1]
{\langle #1 \rangle}

\numberwithin{equation}{section}
\numberwithin{theorem}{section}

\begin{document}
\baselineskip = 15pt

\title[probabilistic Cauchy theory for the cubic NLS on $\R^d$]
{On the probabilistic Cauchy theory of the cubic
nonlinear Schr\"odinger equation on $\R^d$, $d \geq 3$}

\author[\'A.~B\'enyi, T.~Oh, and O.~Pocovnicu]
{\'Arp\'ad  B\'enyi, Tadahiro Oh, and Oana Pocovnicu}

\address{
\'Arp\'ad  B\'enyi\\
Department of Mathematics\\
 Western Washington University\\
 516 High Street, Bellingham\\
  WA 98225\\ USA}
\email{arpad.benyi@wwu.edu}

\address{
Tadahiro Oh\\
School of Mathematics\\
The University of Edinburgh\\
and The Maxwell Institute for the Mathematical Sciences\\
James Clerk Maxwell Building\\
The King's Buildings\\
 Peter Guthrie Tait Road\\
Edinburgh\\ 
EH9 3FD\\United Kingdom} 

\email{hiro.oh@ed.ac.uk}

\address{
Oana Pocovnicu\\
School of Mathematics\\
Institute for Advanced Study\\
Einstein Drive, Princeton\\ NJ 08540\\ USA
and
Department of Mathematics\\
Princeton University\\
Fine Hall\\ Washington Road\\
Princeton\\ NJ 08544 USA 
}

\email{opocovnicu@math.princeton.edu}




\begin{abstract}

We consider the Cauchy problem
of the cubic nonlinear Schr\"odinger equation (NLS) $:
i \partial_t u + \Delta u = \pm |u|^{2}u$ on $\R^d$, $d \geq 3$,
with random initial data and prove almost sure well-posedness results below the scaling critical regularity
$s_\textup{crit} = \frac{d-2}{2}$.
More precisely, given a function on $\R^d$, we introduce a randomization adapted to the Wiener decomposition, and, intrinsically, to the so-called modulation spaces.
Our goal in this paper is three-fold.
(i) We prove  almost sure  local well-posedness of the cubic NLS below the scaling critical regularity
along with small data global existence and scattering.
(ii)
We implement
a probabilistic perturbation argument
and prove
 `conditional' almost sure global well-posedness for $d = 4$ in the defocusing case, assuming an a priori energy bound on the critical Sobolev norm of the nonlinear part of a solution; when $d \ne 4$, we  show that
conditional almost sure global well-posedness in the defocusing case also holds under an additional assumption of global well-posedness of solutions to the defocusing cubic NLS with deterministic initial data in the critical
Sobolev regularity.
(iii) Lastly, we prove global well-posedness and scattering with a large probability for initial data randomized on dilated cubes.
\end{abstract}

%
\maketitle
\tableofcontents

\baselineskip = 15pt

\section{Introduction}
\subsection{Background}
In this paper, we consider
the Cauchy problem of the cubic nonlinear Schr\"odinger equation (NLS)
on $\R^d$, $d\geq 3$:
\begin{equation}
\begin{cases}\label{NLS1}
i \partial_t u + \Delta u = \pm \N(u)\\
u\big|_{t = 0} = u_0 \in H^s(\R^d),
\end{cases}
\qquad ( t, x) \in \R \times \R^d,
\end{equation}

\noi
where $\N(u) := |u|^2 u$.
The cubic NLS \eqref{NLS1} has been studied
extensively from both the theoretical and applied points of view.
Our main focus is to study well-posedness of \eqref{NLS1}
with {\it random} and {\it   rough} initial data.

It is well known that the cubic NLS
\eqref{NLS1}
 enjoys the dilation symmetry.
 More precisely,  if $u(t, x)$ is a solution to \eqref{NLS1}
on $\R^d$ with an initial condition $u_0$, then
\begin{align}
u_\mu(t, x) := \mu^{-1} u (\mu^{-2}t, \mu^{-1}x)
\label{scaling}
\end{align}

\noi
is also a solution to \eqref{NLS1} with the $\mu$-scaled initial condition
$u_{0, \mu}(x) := \mu^{-1} u_0 (\mu^{-1}x)$.
Associated to this dilation symmetry,
there is the so-called  scaling-critical Sobolev index $s_\textup{crit} := \frac{d-2}{2}$
such that the homogeneous $\dot{H}^{s_\textup{crit}}$-norm is invariant
under this dilation symmetry.
In general, we have
\begin{align}
\|u_{0, \mu}\|_{\dot H^s(\R^d)}
= \mu^{\frac{d-2}{2} - s}\|u_0\|_{\dot H^s(\R^d)} .
\label{scaling2}
\end{align}

\noi
If an initial condition $u_0$ is in $H^s(\R^d)$,
we say that the Cauchy problem \eqref{NLS1} is
subcritical, critical, or supercritical,
depending on $s > s_\textup{crit}$, $s = s_\textup{crit}$, or $s < s_\textup{crit}$,
respectively.

Let us first discuss
 the (sub-)critical regime.
In this case,  \eqref{NLS1} is known to be locally well-posed.
See Cazenave-Weissler \cite{CW}
for local well-posedness of \eqref{NLS1} in the critical
Sobolev spaces.
As is well known,
the conservation laws play an important role
in discussing long time behavior of solutions.
There are three known conservation laws for the cubic NLS \eqref{NLS1}:
\begin{align*}
\text{Mass: }&  M[u](t) := \int_{\R^d} |u(t, x)|^2 dx,\\
\text{Momentum: }&  P[u](t) := \Im \int_{\R^d}  u (t, x)\cj{\nb  u(t, x)} dx,\\
\text{Hamiltonian: } &  H[u](t) := \frac 12 \int_{\R^d} |\nb u(t, x)|^2 dx
\pm \frac 14 \int_{\R^d} |u(t, x)|^4 dx.
\end{align*}

\noi
The Hamiltonian is also referred to as the energy.
In view of the conservation of the energy,
the cubic NLS is called  energy-subcritical when $d \leq 3$
($s_\textup{crit} < 1$),
 energy-critical when $d = 4$ ($s_\textup{crit} = 1$),
and energy-supercritical when $d \geq 5$ ($s_\textup{crit} > 1$), respectively.

In the following, let us discuss
the known results on
 the global-in-time behavior of solutions to
the defocusing NLS,
corresponding to the $+$ sign in \eqref{NLS1},
in high dimensions $d\geq 3$.
When $d = 4$, the Hamiltonian is invariant under the scaling \eqref{scaling}
and plays a crucial role in the global well-posedness theory.
Indeed, Ryckman-Vi\c{s}an \cite{RV}
proved global well-posedness and scattering
for the defocusing cubic NLS on $\R^4$.
See also Vi\c{s}an \cite{Visan}.
When $d \ne 4$,
there is no known
 scaling invariant positive conservation law
for  \eqref{NLS1}
in high dimensions $d\geq 3$.
This makes it difficult to
study
 the global-in-time behavior of solutions,
in particular, in the scaling-critical regularity.
There are, however,
`conditional' global well-posedness and scattering results
as we describe below.
When $d = 3$ ($s_\textup{crit} = \frac 12$),
Kenig-Merle \cite{KM2}
applied the concentration compactness and rigidity method
developed in their previous paper \cite{KM} and proved that
if $u \in L^\infty_t \dot H_x^\frac{1}{2}(I\times \R^3)$,
where $I$ is a maximal interval of existence, then
$u$ exists globally in time and scatters.
For $ d\geq 5$, the cubic NLS is supercritical
with respect to any known conservation law.
Nonetheless,  motivated by a similar result of  Kenig-Merle \cite{KM3} on radial solutions
to the energy-supercritical nonlinear wave equation (NLW) on $\R^3$,
Killip-Vi\c{s}an \cite{KV} proved
that if $u \in L^\infty_t \dot H_x^{s_\textup{crit}}(I\times \R^d)$,
where $I$ is a maximal interval of existence, then
$u$ exists globally in time and scatters.
Note that the results in \cite{KM2} and \cite{KV}
are {\it conditional} in the sense that
they
assume an {\it a priori} control on the critical Sobolev norm.
The question of global well-posedness and scattering
without any a priori assumption remains a challenging open problem
for $d = 3$ and $d \geq 5$.

So far,  we have discussed
well-posedness in the (sub-)critical regularity.
In particular, the cubic NLS \eqref{NLS1} is locally well-posed
 in the (sub-)critical regularity, i.e.~$s \geq s_\textup{crit}$.
In the supercritical regime, i.e.~$s < s_\textup{crit}$,
on the contrary,
\eqref{NLS1} is known to be ill-posed.
See  \cite{CCT, BGT2, Carles, AC}.
In the following, however,
we consider
the Cauchy problem \eqref{NLS1}
with initial data in $H^s(\R^d)$, $s < s_\textup{crit}$
in a  probabilistic manner.
More precisely, given a function $\phi \in H^s(\R^d)$
with $s < s_\textup{crit}$,
we introduce a randomization $\phi^\o$
and prove almost sure well-posedness of \eqref{NLS1}.

In studying the Gibbs measure
for the defocusing (Wick ordered) cubic NLS on $\T^2$,
Bourgain \cite{BO7} considered random initial data of the form:
\begin{equation}
 u_0^\omega(x) = \sum_{n \in \Z^2} \frac{g_n(\o)}{\sqrt{1+|n|^2}}e^{i n \cdot x},
\label{I1}
 \end{equation}

\noi
where $\{g_n\}_{n \in \Z^2}$ is a sequence of independent
standard complex-valued  Gaussian random variables.
The function \eqref{I1} represents a typical
element in the support of the Gibbs measure,
more precisely,  in the support of the Gaussian free field on $\T^2$
associated to this Gibbs measure,
and is critical with respect to the scaling.
With a combination of deterministic PDE techniques
and probabilistic arguments,
Bourgain showed that
 the (Wick ordered) cubic NLS on $\T^2$ is well-posed
almost surely with respect to the random initial data \eqref{I1}.
In the context of
the cubic NLW
on a three dimensional compact Riemannian manifold $M$,
Burq-Tzvetkov \cite{BT2}
considered the Cauchy problem with a more general class of random initial data.
Given an eigenfunction expansion $u_0(x) = \sum_{n = 1}^\infty  c_n e_n(x) \in H^s(M)$
of an initial condition\footnote{For NLW, one needs to specify
$(u, \dt u)|_{t = 0}$ as an initial condition. For simplicity of presentation, we
only discuss $u|_{t = 0}$.},
where $\{e_n\}_{n = 1}^\infty$ is an orthonormal basis of $L^2(M)$
consisting of the eigenfunctions of the Laplace-Beltrami operator,
they introduced a randomization $u_0^\o$ by
\begin{equation}
u_0^\omega (x) = \sum_{n = 1}^\infty g_n (\omega) c_n e_n(x).
\label{I2}
\end{equation}

\noi
Here,  $\{g_n\}_{n = 1}^\infty$ is a sequence of independent mean-zero
random variables with a uniform bound on the fourth moments.
Then,
they proved almost sure local well-posedness
with random initial data of the form \eqref{I2}
for $s \geq \frac 14$.
Since  the scaling-critical Sobolev index
for this problem is
$ s_\textup{crit}=\frac 12 $, this result allows us to take initial data below the critical regularity
and still construct solutions upon randomization of the initial data.
 We point out that
the randomized function  $u_0^\o$ in \eqref{I2}
has the same Sobolev regularity as the original function $u_0$
and is not smoother, almost surely.
However, it enjoys a better integrability,
which allows one to prove improvements of Strichartz estimates.
(See Lemmata \ref{PROP:Str1} and \ref{PROP:Str2} below.)
Such an improvement on integrability  for random Fourier series
is known as  Paley-Zygmund's theorem \cite{PZ}.
See also Kahane \cite{Kahane}
and Ayache-Tzvetkov \cite{AT}.
 There are several works on Cauchy problems of evolution equations with random data
that followed these results,
including some on almost sure global well-posedness:
\cite{Bo97, Thomann, CO, Oh11, BTT,  Deng, DS1, BT3, NPS, DS2, R, BTT2,  BB1, BB2, NS, PRT, LM}.

\subsection{Randomization adapted to the Wiener decomposition and modulation spaces}

Many of the results mentioned above are on compact domains,
where there is  a countable basis of eigenfunctions of the Laplacian
and thus there is a natural way to introduce a randomization.
On $\R^d$,
there is no  countable basis of $L^2(\R^d)$ consisting of eigenfunctions of the Laplacian.
Randomizations have been introduced
with respect to some other countable bases of $L^2(\R^d)$,
for example,
a countable basis of eigenfunctions of
the Laplacian with a confining potential such as the harmonic
oscillator $\Dl - |x|^2$,
leading to a careful study of properties of eigenfunctions.
In this paper, our goal is to introduce a simple and natural randomization for functions on $\R^d$. For this purpose, we first review some basic notions related to the so-called modulation spaces of time-frequency analysis \cite{Gr}.

The modulation spaces were introduced by Feichtinger \cite{Fei} in early eighties. The groundwork theory regarding these spaces of time-frequency analysis
was then established in joint collaboration with Gr\"ochenig \cite{FG1, FG2}.
The modulation spaces arise from a uniform partition of the frequency space, commonly known as the {\it Wiener decomposition} \cite{W}: $\R^d = \bigcup_{n \in \Z^d} Q_n$,
where $Q_n$ is the unit cube centered at $n \in \Z^d$.
The Wiener decomposition of $\R^d$ induces a natural uniform decomposition of the frequency space of a signal via the (non-smooth)
frequency-uniform decomposition
operators $\mathcal F^{-1}\chi_{Q_n}\mathcal F$. Here, $\mathcal Fu=\widehat u$ denotes the Fourier transform of
a distribution $u$.
The drawback of this approach is the roughness of the characteristic functions $\chi_{Q_n}$, but this issue can easily be fixed by smoothing them out appropriately.
We have the following definition of the (weighted) modulation spaces $M^{p, q}_s$.
Let $\psi \in \mathcal{S}(\R^d)$ such that
\begin{equation}
\supp \psi \subset [-1,1]^d
\qquad \text{and} \qquad\sum_{n \in \Z^d} \psi(\xi -n) \equiv 1
\ \text{ for any }\xi \in \R^d.
\label{mod1a}
\end{equation}
Let $0<p,q\leq \infty$ and $s\in\mathbb R$; $M^{p, q}_s$ consists of all tempered distributions $u\in\mathcal S'(\R^d)$ for which the (quasi) norm
\begin{equation}
\|u\|_{M_s^{p, q}(\R^d)} := \big\| \jb{n}^s \|\psi(D-n) u
\|_{L_x^p(\R^d)} \big\|_{\l^q_n(\mathbb{Z}^d)}
\label{mod2}
\end{equation}

\noi
is finite. Note that $\psi(D-n)u(x)=\int_{\R^d} \psi (\xi-n)\widehat u (\xi)e^{2\pi ix\cdot \xi}\, d\xi$ is just a Fourier multiplier operator with symbol $\chi_{Q_n}$ conveniently smoothed.

It is worthwhile to compare the definition \eqref{mod2}
with that of the Besov spaces. Let $\varphi_0, \varphi \in \mathcal{S}(\R^d)$ such that
$\supp \varphi_0 \subset \{ |\xi| \leq 2\}$, $\supp \varphi \subset
\{ \frac{1}{2}\leq |\xi| \leq 2\}$, and $ \varphi_0(\xi) + \sum_{j =
1}^\infty \varphi(2^{-j}\xi) \equiv 1.$ With $\varphi_j(\xi) =
\varphi(2^{-j}\xi)$,
we define  the (inhomogeneous) Besov spaces $B_s^{p, q}$
via the norm
\begin{equation} \label{besov1}
\|u\|_{B^s_{p, q}(\R^d) } = \big\| 2^{js} \|\varphi_j(D) u
\|_{L^p(\R^d)} \big\|_{\l^q_j(\mathbb{Z}_{\geq 0})}.
\end{equation}

\noi
There are several known embeddings between Besov, Sobolev, and
modulation spaces. See,  for example,  Okoudjou \cite{Ok},
Toft \cite{To}, Sugimoto-Tomita \cite{suto2},  and Kobayashi-Sugimoto \cite{kosu}.

Now, given a function $\phi$ on $\R^d$,
we have
\[ \phi = \sum_{n \in \Z^d} \psi(D-n) \phi,\]

\noi
where $\psi(D-n)$ is defined above. This decomposition leads to a randomization of $\phi$ that is  
very natural from the perspective of time-frequency analysis associated to modulation spaces. 
Let $\{g_n\}_{n \in \Z^d}$ be a sequence of independent mean zero complex-valued random variables
on a probability space $(\O, \F, P)$,
where the real and imaginary parts of $g_n$ are independent
and endowed
with probability distributions $\mu_n^{(1)}$ and $\mu_n^{(2)}$.
Then, we can define the {\it Wiener randomization of $\phi$} by
\begin{equation}
\phi^\omega : = \sum_{n \in \Z^d} g_n (\omega) \psi(D-n) \phi.
\label{R1}
\end{equation}

 Almost simultaneously with  our first paper \cite{BOP1}, L\"uhrmann-Mendelson \cite{LM} also 
considered a randomization of the form \eqref{R1} (with cubes $Q_n$ being substituted by appropriately localized balls) in the study of NLW
on $\R^3$. See Remark \ref{REM:LM} below. For a similar randomization used in the study of the Navier-Stokes equations, see the work of Zhang and Fang \cite{ZF}.
We would like to stress again, however, that our reason for considering
the randomization of the form \eqref{R1} comes from its connection to time-frequency analysis. See  also our previous papers \cite{BP} and \cite{BOP1}.

In the sequel, we make the following assumption on the distributions $\mu_n^{(j)}$: there exists $c>0$ such that
\begin{equation}
\bigg| \int_{\R} e^{\g x } d \mu_n^{(j)}(x) \bigg| \leq e^{c\g^2}
\label{R2}
\end{equation}
	
\noi
for all $\g \in \R$, $n \in \Z^d$, $j = 1, 2$.
Note that \eqref{R2} is satisfied by
standard complex-valued Gaussian random variables,
standard Bernoulli random variables,
and any random variables with compactly supported distributions.

It is easy to see that, if $\phi \in H^s(\R^d)$,
then  the randomized function $\phi^\o$ is
almost surely in $H^s(\R^d)$. 
While  there is no smoothing upon randomization in terms of differentiability in general, this randomization \emph{behaves better under integrability};
if $\phi \in L^2(\R^d)$,
then  the randomized function $\phi^\o$ is
almost surely  in $L^p(\R^d)$ for any finite $p \geq 2$.
As a result of this enhanced integrability,
we have improvements of the Strichartz estimates.
See Lemmata \ref{PROP:Str1} and  \ref{PROP:Str2}.
These improved Strichartz estimates play an essential role in proving
probabilistic well-posedness results,
which we describe below.

\subsection{Main results}

Recall that
the scaling critical Sobolev index for the cubic NLS on $\R^d$
is $s_\textup{crit} = \frac{d-2}{2}$.
In the following,
we take  $\phi \in H^s(\R^d) \setminus H^{s_\textup{crit}} (\R^d)$
for some range of $s < s_\textup{crit}$,
that is,
below the critical regularity.
Then, we consider the well-posedness problem
of \eqref{NLS1} with respect to
the  randomized initial data $\phi^\o$ defined in \eqref{R1}.

For $d \geq 3$, define $s_d$ by
\begin{align}
 s_d := \frac{d-1}{d+1}\cdot s_\textup{crit} = \frac{d-1}{d+1} \cdot \frac{d-2}{2}
\label{Sd1}
\end{align}

\noi
Note that $s_d <s_\textup{crit}$ and $\frac{s_d}{s_\textup{crit}} \to 1$ as $d \to \infty$.
Throughout the paper,
we use  $S(t)= e^{it\Dl}$ to denote the linear propagator of the Schr\"odinger group.

We are now ready to state our main results.

\begin{theorem}[Almost sure local well-posedness]\label{THM:1}
Let $d \geq 3$
and   $s > s_d$.
Given $\phi \in H^s(\R^d)$, let $\phi^\o$ be its Wiener randomization defined in \eqref{R1},
satisfying \eqref{R2}.
Then, the cubic NLS \eqref{NLS1} on $\R^d$
is almost surely locally well-posed
with respect to the randomization $\phi^\omega$ as initial data.
More precisely,
there exist $ C, c, \g>0$ such that
for each $0< T\ll 1$,
there exists a set $\O_T \subset \O$ with the following properties:

\smallskip
\begin{itemize}
\item[\textup{(i)}]
$P(\O_T^c) < C \exp\big(-\frac{c}{T^\g \|\phi\|_{H^s}^2 }\big)$,

\item[\textup{(ii)}]
For each $\o \in \O_T$, there exists a (unique) solution $u$
to \eqref{NLS1}
with $u|_{t = 0} = \phi^\o$
in the class
\[ S(t) \phi^\o + C([-T, T]: H^{\frac{d-2}{2}} (\R^d)) \subset C([-T, T]:H^s(\R^d)).\]

\end{itemize}
\end{theorem}

\noi
We prove Theorem \ref{THM:1}
by considering the equation satisfied by the nonlinear part of a solution $u$.
Namely,
let
$z (t) = z^\o(t) : = S(t) \phi^\o$
and $v(t) := u(t) - S(t) \phi^\o$
be  the linear and nonlinear parts of $u$, respectively.
Then,
\eqref{NLS1} is equivalent to the following perturbed NLS:
\begin{equation}
\begin{cases}
	 i \dt v + \Dl v = \pm |v + z|^2(v+z)\\
v|_{t = 0} = 0.
 \end{cases}
\label{NLS2}
\end{equation}

\noi
We reduce our analysis to
the Cauchy problem \eqref{NLS2}	
for $v$,
viewing $z$ as a random forcing term.
Note that such a point of view is common in the study
of stochastic PDEs.
As a result, the uniqueness in Theorem \ref{THM:1}
refers to uniqueness of the nonlinear part
$v(t) = u(t) - S(t) \phi^\o$ of a solution $u$.

The proof of Theorem \ref{THM:1} is based on the fixed point argument
involving
 the variants of the $X^{s, b}$-spaces
adapted to the $U^p$- and $V^p$-spaces
introduced by Koch, Tataru, and their collaborators \cite{KochT, HHK, HTT1}.
See Section \ref{SEC:Up} for the basic definitions and properties of these function spaces.
The main ingredient is the local-in-time improvement of the Strichartz estimates
(Lemma \ref{PROP:Str1})
and the refinement of the bilinear Strichartz estimate (Lemma \ref{LEM:Ys1} (ii)).
We point out that,
although $\phi$ and its randomization $\phi^\o$
have  a
supercritical Sobolev regularity,
 the randomization essentially makes the  problem  subcritical, at least locally in time,
 and therefore,
one can also prove
Theorem \ref{THM:1}
only  with the
classical subcritical $X^{s, b}$-spaces, $b > \frac 12$.
See \cite{BOP1} for the result when
$d = 4$.

Next, we turn our attention to the global-in-time behavior
of the solutions constructed in Theorem \ref{THM:1}.
The key nonlinear estimate
in the proof of
Theorem \ref{THM:1}
combined with
the global-in-time improvement of the Strichartz estimates
(Lemma \ref{PROP:Str2}) yields
the following result
on small data global well-posedness and scattering.

\begin{theorem}[Probabilistic
small data global well-posedness and scattering]\label{THM:2}
Let $d\geq 3$ and
 $s \in ( s_d, s_\textup{crit}]$,
 where $s_d$ is as in \eqref{Sd1}.
Given $\phi \in H^s(\R^d)$,
let $\phi^\o$ be its Wiener randomization defined in \eqref{R1},
satisfying \eqref{R2}.
Then,
there exist $ C, c >0$ such that
for each $0<\eps \ll1$,
there exists a set $\O_\eps \subset \O$ with the following properties:

\smallskip
\begin{itemize}
\item[\textup{(i)}]
$P(\O_\eps^c) \leq C \exp\big(-\frac{c}{\eps^2 \|\phi\|_{H^s}^2 }\big)
\to 0$
as $\eps \to 0$,

\item[\textup{(ii)}]
For each $\o \in \O_\eps$, there exists a (unique)
global-in-time solution $u$ to \eqref{NLS1}
with
\[u|_{t = 0} = \eps \phi^\o\]
in the class
\[ \eps S(t) \phi^\o + C(\R : H^{\frac{d-2}{2}} (\R^d)) \subset C(\R:H^s(\R^d)),\]

\item[\textup{(iii)}]
We have scattering for each $\o\in \Omega_\eps$.
More precisely, for each $\o \in \O_\eps$,
there exists $v_+^\o \in H^\frac{d-2}{2}(\R^d)$ such that
\[ \| u(t) - S(t)(\eps \phi^\o + v_+^\o) \|_{H^\frac{d-2}{2}(\R^d)} \to 0\]

\noi
as $t \to \infty$.
A similar statement holds for $t \to -\infty$.
\end{itemize}
\end{theorem}

\noi
In general, a local well-posedness result in a critical space
is often accompanied by small data global well-posedness and scattering.
In this sense,
Theorem \ref{THM:2}
is an expected consequence of Theorem \ref{THM:1},
since, in our construction,
 the nonlinear part $v$ lies in the critical space $H^\frac{d-2}{2}(\R^d)$.
The next natural question is
probabilistic global well-posedness
for large data.
In order to state our result,
we need to make several hypotheses.
The first hypothesis is on a
probabilistic  a priori energy bound on the nonlinear part $v$.

\medskip
\noi
{\bf Hypothesis (A):}
Given any $T, \eps > 0$, there exists $R = R(T, \eps)$ and $\O_{T, \eps} \subset \O$ such that
\begin{itemize}
\item[(i)]
$P(\O_{T, \eps}^c) < \eps$, and
\item[(ii)]
If $v = v^\o$ is the solution to \eqref{NLS2}
for $\o \in \O_{T, \eps}$, then
the following {\it a priori} energy estimate holds:
\begin{equation}
 \|v(t) \|_{L^\infty([0, T]; H^\frac{d-2}{2}(\R^d))} \leq R(T, \eps).
\label{HypA}
 \end{equation}

\end{itemize}

\medskip

\noi
Note that Hypothesis (A)
does {\it not} refer to  existence
of a solution $v = v^\o$ on $[0, T]$ for given $\o \in \O_{T, \eps}$.
It only hypothesizes the {\it a priori} energy bound \eqref{HypA},
just like the usual conservation laws. It may be possible to prove \eqref{HypA} independently from the argument presented in this paper.
Such a probabilistic  a priori energy
estimate is known, for example, for the cubic NLW. See Burq-Tzvetkov \cite{BT3}.
We point out that the upper bound $R(T, \eps)$ in \cite{BT3} tends to $\infty$
as $T\to \infty$. See also \cite{POC}.

The next hypothesis is on global existence and space-time bounds
of solutions to the cubic NLS  \eqref{NLS1}
with deterministic initial data belonging to the critical space $H^\frac{d-2}{2}(\R^d)$.

\medskip

\noi
{\bf Hypothesis (B):}
Given  any $w_0 \in H^\frac{d-2}{2}(\R^d)$,
there exists a global solution $w$ to the defocusing cubic NLS \eqref{NLS1}
with $w|_{t = 0} = w_0$.
Moreover, there exists a function $C:[0, \infty)\times [0, \infty)\to [0, \infty)$
which is non-decreasing in each argument
such that
\begin{equation}
\| w\|_{L^{d+2}_{t, x}([0, T]\times \R^d)} \leq C\big(\|w_0\|_{H^\frac{d-2}{2}(\R^d)}, T\big)
\label{HypB}
\end{equation}

\noi
for any $T > 0$.
\medskip

\noi
Note that when $d = 4$, Hypothesis (B) is known to be true for any $T> 0$
thanks to the global well-posedness result by Ryckman-Vi\c{s}an \cite{RV} and Vi\c{s}an \cite{Visan}.
For other dimensions $d \geq 3$ with $d \ne 4$,
it is not known whether Hypothesis (B) holds.
Let us compare \eqref{HypB}
and the results in \cite{KM2} and \cite{KV}.
Assuming that  $w \in L^\infty_t \dot H_x^{s_\textup{crit}}(I_*\times \R^d)$,
where $I_*$ is a maximal interval of existence,
it was shown in \cite{KM2} and \cite{KV} that $I_* = \R$ and
\begin{equation}
\| w\|_{L^{d+2}_{t, x}(\R \times \R^d)} \leq C\Big(\|w\|_{L^\infty_t \dot H_x^\frac{d-2}{2}(\R\times \R^d)}\Big).
\label{HypC}
\end{equation}

\noi
We point out that  Hypothesis (B) is not directly comparable
to the results in \cite{KM2, KV} in the following sense.
On the one hand, by assuming that  $w \in L^\infty_t \dot H_x^{s_\textup{crit}}(I_*\times \R^d)$,
the results in \cite{KM2, KV} yield the global-in-time bound \eqref{HypC},
while
Hypothesis (B) assumes the bound \eqref{HypB}  only for each {\it finite} time $T>0$
and does not assume a global-in-time bound.
On the other hand,
\eqref{HypB} is much stronger than \eqref{HypC}
in the sense that the right-hand side of \eqref{HypB} depends only on the size of an initial condition $w_0$,
while the right-hand side of \eqref{HypC} depends
on the global-in-time $L^\infty_t \dot H^\frac{d-2}{2}_x$-bound of the solution $w$.
Hypothesis (B), just like Hypothesis (A),
 is of independent interest
from  Theorem \ref{THM:3} below
and is closely related to the fundamental open problem
of global well-posedness and scattering for the defocusing cubic NLS \eqref{NLS1}
for $d = 3$ and $d \geq 5$.

We now state our third theorem on almost sure global well-posedness
of the cubic NLS
under Hypotheses (A) and (B).
We restrict ourselves to the defocusing NLS in the next theorem.

\begin{theorem}[Conditional almost sure global well-posedness] \label{THM:3}
Let $d \geq 3$
and
 $s \in ( s_d, s_\textup{crit}]$,
 where $s_d$ is as in \eqref{Sd1}.
Assume Hypothesis \textup{(A)}.
Furthermore, assume Hypothesis \textup{(B)} if $d \ne 4$.
Given $\phi \in H^s(\R^d)$, let $\phi^\o$ be its Wiener randomization defined in \eqref{R1},
satisfying \eqref{R2}.
Then, the defocusing cubic NLS \eqref{NLS1} on $\R^d$
is almost surely globally  well-posed
with respect to the randomization $\phi^\omega$ as initial data.
More precisely,
there exists a set $\Sigma \subset \O$
with $P(\Sigma) = 1$
such that,
for each $\o \in \Sigma$, there exists a (unique) global-in-time  solution $u$
to \eqref{NLS1}
with $u|_{t = 0} = \phi^\o$
in the class
\[ S(t) \phi^\o + C(\R: H^{\frac{d-2}{2}} (\R^d)) \subset C(\R:H^s(\R^d)).\]

\end{theorem}

\noi
The main tool in the proof of Theorem \ref{THM:3}
is a perturbation lemma
for the cubic NLS (Lemma \ref{LEM:perturb}).
Assuming a control on the critical norm (Hypothesis (A)),
we iteratively apply the perturbation lemma in the {\it probabilistic} setting
to show that a solution can be extended to a time depending
only on the critical norm.
Such a perturbative approach  was previously used by
Tao-Vi\c{s}an-Zhang \cite{TVZ} and Killip-Vi\c{s}an with the second and third authors \cite{KOPV}.
The novelty of Theorem \ref{THM:3}
is an application of such a technique in the probabilistic setting.
While there is no invariant measure for the nonlinear evolution in our setting,
we exploit the quasi-invariance property of
the distribution of the linear solution $S(t) \phi^\o$.
See Remark \ref{REM:asGWP}.
Our implementation of the proof of Theorem \ref{THM:3}
is sufficiently general that it can be easily applied to other equations.
See \cite{POC} in the context of the energy-critical  NLW on $\R^d$, $d = 4, 5$, 
where both Hypotheses (A) and (B) are satisfied.

When $d \ne 4$,
the conditional almost sure global well-posedness
in Theorem \ref{THM:3}
has a flavor analogous to
the deterministic conditional global well-posedness
in the critical Sobolev spaces by Kenig-Merle \cite{KM2}
and Killip-Vi\c{s}an \cite{KV}.
In the following, let us discuss the situation when $d = 4$.
In this case,
we only assume Hypothesis (A)
for Theorem \ref{THM:3}.
While it would be interesting to remove this assumption,
 we do not know how to prove the validity of Hypothesis (A)  at this point.
This is mainly due to the lack of conservation of $H[v](t)$,
i.e. the Hamiltonian evaluated at the nonlinear part $v$ of a solution.
In the context of the energy-critical defocusing cubic NLW on $\R^4$,
however,
one can
prove an analogue of Hypothesis (A)
by establishing
 a probabilistic  a priori bound on the energy $\mathcal{E}[v]$ of the nonlinear part
$v$ of a solution, where the energy $\mathcal{E}[v]$
is defined by
\[ \mathcal{E}[v](t)
 = \frac 12 \int_{\R^4}  |\dt v(t, x)|^2 dx + \frac 12 \int_{\R^4}|\nb v(t, x)|^2 dx
+ \frac 14 \int_{\R^4} |v(t, x)|^4 dx.
\]

\noi
As a consequence, the third author \cite{POC}
successfully implemented a probabilistic perturbation argument
and
proved
almost sure global well-posedness of the energy-critical defocusing  cubic NLW on $\R^4$ with randomized initial data below the scaling critical regularity.\footnote{
In \cite{POC}, 
the third author also proved 
almost sure global well-posedness of the energy-critical defocusing  NLW on $\R^5$.
This result was recently extended to 
the dimension 3 by the second and third authors \cite{OP}.} 
We point out that
 the first term in the energy $\mathcal{E}[v]$
involving the time derivative
plays an essential role in establishing
 a probabilistic  a priori bound on the energy for NLW.
It  seems substantially harder
to verify
Hypothesis (A) for NLS, even when $d = 4$.

\smallskip

While Theorem \ref{THM:3}
provides only conditional almost sure global existence,
our last theorem (Theorem \ref{THM:4}) below presents a
 way to  construct global-in-time solutions below the scaling critical regularity
with a large probability.
The main idea is to use the scaling \eqref{scaling} of the equation
for random initial data below the scaling criticality.
For example, suppose that we have a solution $u$ to \eqref{NLS1}
on a short time interval
with a deterministic initial condition $u_0 \in H^s(\R^d)$, $s < s_\textup{crit}$.
In view of \eqref{scaling} and \eqref{scaling2},
by taking $\mu \to 0$,
we see that
the $H^s$-norm of the scaled initial condition
goes to 0.
Thus,
one might think that the problem can be reduced to small data theory.
This, of course, does not work
  in the usual deterministic setting,  since we do not know
how to construct solutions depending only on
 the $H^s$-norm of the initial data, $s < s_\text{crit}$.
Even in the probabilistic setting, this naive idea does not work
if we simply apply the scaling to the randomized function $\phi^\o$
defined in \eqref{R1}.
This is due to the fact that we need to use (sub-)critical space-time norms
controlling the random linear term $z^\o(t) = S(t) \phi^\o$,
which do not become small even if we take $\mu \ll1$.

To resolve this issue,
we consider
 a randomization based on a partition of the frequency space by {\it dilated} cubes.
Given $\mu > 0$,
define $\psi^\mu$ by
\begin{equation}
\psi^\mu(\xi) = \psi(\mu^{-1} \xi).
\label{psi}
\end{equation}

\noi
Then, we can write
a function $\phi$ on $\R^d$
as
\[ \phi = \sum_{n \in  \Z^d} \psi^\mu (D- \mu n) \phi.\]

Now, we introduce the randomization $\phi^{\omega, \mu} $ of $\phi$ on dilated cubes of scale $\mu$ by
\begin{equation}
\phi^{\omega, \mu} : = \sum_{n \in \Z^d} g_n (\omega) \psi^\mu(D- \mu n) \phi,
\label{R3}
\end{equation}

\noi
where  $\{g_n\}_{n \in \Z^d}$ is a sequence of independent mean zero complex-valued random variables,
satisfying \eqref{R2} as before.
Then, we have the following global well-posedness of \eqref{NLS1}
with a large probability.

\begin{theorem}\label{THM:4}
Let $ d\geq 3$ and
$\phi \in H^s(\R^d)$,
for some
 $s \in ( s_d, s_\textup{crit})$,
where $s_d$ is as in \eqref{Sd1}.
Then,
given
the randomization $\phi^{\o, \mu}$  on dilated cubes
of scale $ \mu \ll1 $ defined in \eqref{R3},
satisfying \eqref{R2},
the cubic NLS \eqref{NLS1} on $\R^d$
is globally well-posed with a large probability.
More precisely, for each $0< \eps \ll 1$,
there exists a small dilation scale $\mu_0 = \mu_0(\eps, \|\phi\|_{H^s})> 0$
such that
for each $\mu \in (0, \mu_0)$,
there exists a set $\Omega_\mu \subset \Omega$
with the following properties:
\begin{itemize}
\item[\textup{(i)}]
$P(\Omega_\mu^c) < \eps$,

\item[\textup{(ii)}]
If $\phi^{\o, \mu}$ is the randomization on dilated cubes defined in \eqref{R3},
satisfying \eqref{R2},
then, for each $\o \in \O_\mu$, there exists a (unique) global-in-time  solution $u$
to \eqref{NLS1}
with $u|_{t = 0} = \phi^{\o, \mu}$
in the class
\[ S(t) \phi^\o + C(\R: H^{\frac{d-2}{2}} (\R^d)) \subset C(\R:H^s(\R^d)).\]

\noi
Moreover,
for each $\o \in \O_\mu$,
scattering holds in the sense
that
there exists $v_+^\o \in H^\frac{d-2}{2}(\R^d)$ such that
\[ \| u(t) - S(t)( \phi^{\o, \mu} + v_+^\o) \|_{H^\frac{d-2}{2}(\R^d)} \to 0\]

\noi
as $t \to \infty$.
A similar statement holds for $t \to -\infty$.

\end{itemize}

\end{theorem}

We conclude this introduction with several remarks.

\begin{remark}\rm
In probabilistic well-posedness results \cite{BO7, Bo97, CO, NS} for NLS on $\T^d$,
random initial data are assumed to be of  the following specific form:
\begin{equation}
\label{I3}
 u_0^\omega(x) = \sum_{n \in \Z^d} \frac{g_n(\o)}{(1+|n|^2)^\frac{\al}{2}}e^{i n \cdot x},
 \end{equation}

\noi
where $\{g_n\}_{n \in \Z^d}$ is a sequence of independent
complex-valued standard Gaussian random variables.
The expression \eqref{I3}
has a close connection to the study of invariant measures
and hence it is of importance.
At the same time, due to the
lack of a full range of Strichartz estimates on $\T^d$,
one could not handle a general randomization of a given function
as in \eqref{I2}.
In this paper, we consider NLS on $\R^d$
and thus we do not encounter this issue
thanks to a full range of the Strichartz estimates.
For NLW, finite speed of propagation allows us to use a full range of Strichartz estimates
even on compact domains, at least locally in time.
Thus, one does not encounter such an issue.

\end{remark}

\begin{remark}\label{REM:LM} \rm
In a recent preprint,
L\"uhrmann-Mendelson \cite{LM}
considered the defocusing NLW on $\R^3$
with randomized initial data, essentially given by  \eqref{R1},
below the critical regularity
and proved almost sure global well-posedness in the energy-subcritical case,
following the method developed in \cite{CO},
namely an adaptation of Bourgain's high-low method \cite{Bo98}
in the probabilistic setting.
As Bourgain's high-low method is a subcritical tool, their global result misses the energy-critical case.\footnote{
In \cite{OP}, the second and third authors recently 
proved almost sure global well-posedness of the energy-critical defocusing  quintic NLW on $\R^3$.} 

The third author \cite{POC}
recently proved
almost sure global well-posedness
of the  energy-critical  defocusing NLW on $\R^d$, $d = 4, 5$, 
with randomized initial data
below the  critical regularity.
The argument is based on an application of a perturbation lemma as in Theorem \ref{THM:3}
along with a probabilistic  a priori control on the energy,
which is not available for the cubic NLS \eqref{NLS1}.

\end{remark}

This paper is organized as follows.
In Section \ref{SEC:2}, we state some probabilistic lemmata.
In Section \ref{SEC:Up},
we go over the basic definitions and properties
of function spaces involving the $U^p$- and $V^p$-spaces.
We prove the key nonlinear estimates in Section \ref{SEC:4}
and then use them to prove Theorems \ref{THM:1} and \ref{THM:2}
in Section \ref{SEC:THM12}.
We divide the proof of Theorem \ref{THM:3} into three sections.
In Sections \ref{SEC:6} and  \ref{SEC:7},
we discuss the Cauchy theory for the defocusing cubic NLS with a deterministic perturbation.
We implement these results in the probabilistic setting and prove Theorem \ref{THM:3}
in Section \ref{SEC:8}.
In Section \ref{SEC:9}, we show how Theorem \ref{THM:4}
follows from the arguments in Sections \ref{SEC:4} and \ref{SEC:THM12},
once we consider a randomization on dilated cubes.
In Appendix \ref{SEC:A}, we state and prove some additional properties
of the function spaces defined in Section \ref{SEC:Up}.

Lastly, note that
we present the proofs of  these results only for positive times
in view of the time reversibility of \eqref{NLS1}.

\section{Probabilistic lemmata}
\label{SEC:2}

In this section, we summarize the probabilistic lemmata used in this paper.
In particular, the probabilistic Strichartz estimates (Lemmata \ref{PROP:Str1}
and \ref{PROP:Str2}) play an essential role.
First, we recall the usual Strichartz estimates on $\R^d$ for readers' convenience.
We say that
a pair  $(q, r)$ is Schr\"odinger admissible
if it satisfies
\begin{equation}
\frac{2}{q} + \frac{d}{r} = \frac{d}{2}
\label{Admin}
\end{equation}

\noi
 with $2\leq q, r \leq \infty$
and $(q, r, d) \ne (2, \infty, 2)$.
Then, the following Strichartz estimates
are known to hold.

\begin{lemma}[\cite{Strichartz, Yajima, GV, KeelTao}]\label{LEM:Str0}
Let $(q, r)$ be  Schr\"odinger admissible.
Then, we have
\begin{equation}
\| S(t) \phi\|_{L^q_t L^r_x (\R\times \R^d)} \lesssim \|\phi\|_{L^2_x(\R^d)}.
\label{Str0}
\end{equation}

\end{lemma}

\noi
In particular, when $q = r$, we have $q = r = \frac{2(d+2)}{d}$.
By applying Sobolev inequality and \eqref{Str0}, we also have
\begin{equation}
\| S(t) \phi\|_{L^p_{t, x} (\R\times \R^d)}
\lesssim \big\||\nb|^{\frac d2 - \frac{d+2}{p}}\phi\big\|_{L^2_x(\R^d)}.
\label{Str0a}
\end{equation}

\noi
for $p \geq \frac{2(d+2)}{d}$.
Recall that the derivative loss in \eqref{Str0a} depends
only on the size of the frequency support and not its location.
Namely,  if $\ft{\phi}$ is supported on a cube $Q$ of side length $N$,
then we have
\begin{equation}
\| S(t) \phi\|_{L^p_{t, x} (\R\times \R^d)}
\lesssim N^{\frac d2 - \frac{d+2}{p}}\|\phi\|_{L^2_x(\R^d)},
\label{Str0b}
\end{equation}
regardless of the center of the cube $Q$.

Next, we present improvements
of the Strichartz estimates under the Wiener randomization \eqref{R1} and where, throughout, we assume \eqref{R2}.
See \cite{BOP1} for the proofs.

\begin{lemma}[Improved local-in-time  Strichartz estimate]\label{PROP:Str1}
Given $\phi \in L^2(\R^d)$,
let $\phi^\o$ be its Wiener randomization defined in \eqref{R1}, satisfying  \eqref{R2}.
Then,
given finite  $q, r \geq 2$,
there exist $C, c>0$ such that
\begin{align*}
P\Big( \|S(t) \phi^\omega\|_{L^q_t L^r_x([0, T]\times \R^d)}> \ld\Big)
\leq C\exp \bigg(-c \frac{\ld^2}{ T^\frac{2}{q}\|\phi\|_{L^2}^{2}}\bigg)
\end{align*}
	
\noi
for all  $ T > 0$ and $\lambda>0$.
In particular,
with $\ld = T^\theta $, we have
\begin{equation}
\|S(t) \phi^\o\|_{L^q_tL^r_x([0, T]\times \R^d)}
\les T^\theta
\label{Str1a}
\end{equation}

\noi
outside a set of probability
\[
\leq C\exp \bigg(-c \frac{1}{ T^{2(\frac{1}{q}-\theta)}\|\phi\|_{L^2}^{2}}\bigg).
\]
Note that
this probability
can be made arbitrarily small by letting $T\to 0$
as long as $\theta < \frac{1}{q}$.

\end{lemma}
\noi

The next lemma
states
an improvement of the Strichartz estimates in the global-in-time setting.

\begin{lemma}[Improved global-in-time Strichartz estimate]\label{PROP:Str2}
Given $\phi \in L^2(\R^d)$,
let $\phi^\o$ be its Wiener randomization defined in \eqref{R1},
satisfying  \eqref{R2}.
Given a Schr\"odinger admissible pair $(q, r)$ with $q, r < \infty$,
let $\wt {r} \geq r$.
Then, there exist $C, c>0$ such that
\begin{align*}
P\Big( \|S(t) \phi^\omega\|_{L^q_t L^{\wt{r}}_x ( \R \times \R^d)} > \ld\Big)
\leq Ce^{-c \ld^2 \|\phi\|_{L^2}^{-2}}.
\end{align*}

\noi
In particular, given any small $\eps > 0$, we have
\[ \|S(t) \phi^\omega\|_{L^q_t L^{\wt{r}}_x ( \R \times \R^d)}
\les \Big( \log \frac{1}{\eps}\Big)^\frac{1}{2} \|\phi\|_{L^2(\R^d)} \]

\noi
outside a set of probability $< \eps$.

\end{lemma}

\noi
Recall that the diagonal Strichartz admissible index is given by
$p= \frac{2(d+2)}{d}$.
In the diagonal case $q = \wt r$,
it is easy to see that the condition of Lemma \ref{PROP:Str2} is satisfied if $q = \wt r \geq p = \frac{2(d+2)}{d}$.
In the following, we apply Lemma \ref{PROP:Str2} in this setting.

We also need the following lemma on the control of the size of $H^s$-norm of $\phi^\o$.

\begin{lemma} \label{LEM:Hs}
Given  $\phi \in H^s(\R^d)$, let $\phi^\o$ be its Wiener randomization
defined in \eqref{R1}, satisfying  \eqref{R2}.
Then, we have
\begin{align}
P\Big( \| \phi^\omega \|_{ H^s(  \R^d)} > \ld\Big)
\leq C e^{-c \ld^2  \|\phi\|_{ H^s}^{-2}}.
\label{Hs1}
\end{align}

\end{lemma}

We conclude this section by introducing some notations
involving Strichartz and space-time Lebesgue spaces.
In the sequel, given an interval $I\subset\R$, we
often use
$L^q_tL^r_x(I)$
to
denote $L^q_tL^r_x(I\times \R^d)$.
We also define the $\dot S^{s_\textup{crit}}(I)$-norm in the usual manner by setting
\[ \|u \|_{\dot S^{s_\textup{crit}}(I)}
:= \sup\Big\{ \big\| |\nb|^\frac{d-2}{2} u \big\|_{L^q_tL^r_x(I\times \R^d)
}\Big\},\]

\noi
where the supremum is taken over all Schr\"odinger admissible pairs $(q, r)$.

\section{Function spaces and their properties}
\label{SEC:Up}

In this section, we go over the basic definitions and properties of
the $U^p$- and $V^p$-spaces, developed by
Tataru, Koch, and their collaborators \cite{KochT, HHK, HTT1}.
These spaces have been very effective in
establishing well-posedness  of various dispersive PDEs in critical regularities.
See Hadac-Herr-Koch \cite{HHK} and Herr-Tataru-Tzvetkov \cite{HTT1} for detailed proofs.

Let $H$ be a separable Hilbert space over $\C$.
 In particular, it will be either $H^s(\R^d)$ or $\C$.
 Let $\mathcal{Z}$ be the collection of finite partitions $\{t_k\}_{k = 0}^K$ of $\R$:
 $-\infty < t_0 < \cdots < t_K \leq \infty$.
 If $t_K = \infty$,
 we use the convention $u(t_K) :=0$
 for all functions $u:\R\to H$.
We use $\chi_I$ to denote the sharp characteristic function
of a set $I \subset \R$.

 \begin{definition} \label{DEF:X1}\rm
Let $1\leq p < \infty$.

\smallskip
\noi
\textup{(i)}
A $U^p$-atom is defined by a step function $a:\R\to H$
of the form
\[ a = \sum_{k = 1}^K \phi_{k - 1} \chi_{[t_{k-1}, t_k)}, \]

\noi
where $\{t_k\}_{k = 0}^K \in \mathcal{Z}$
and $\{\phi_k\}_{k = 0}^{K-1} \subset H$
with $\sum_{k = 0}^{K-1} \|\phi_k\|_H^p = 1$.
Then, we define the atomic space $U^p(\R; H)$
to be the collection of functions $u:\R\to H$
of the form
\begin{equation} u = \sum_{j = 1}^\infty \ld_j a_j,
\quad \text{ where $a_j$'s are $U^p$-atoms and $\{\ld_j\}_{j \in \mathbb{N}}\in \l^1(\mathbb{N}; \C)$},
\label{X1}
\end{equation}

\noi
with the norm
\[ \|u\|_{U^p(\R; H)} : = \inf \Big\{ \|{ \bf \ld} \|_{\l^1}
: \eqref{X1} \text{ holds with } \ld = \{\ld_j \}_{j \in \mathbb{N}}
\text{ and some $U^p$-atoms } a_j\Big\}.\]

\smallskip
\noi
\textup{(ii)}
We define the space $V^p(\R; H)$ of  functions of bounded $p$-variation
to be the collection of functions $u : \R \to H$
with $\|u\|_{V^p(\R; H)} < \infty$,
where the $V^p$-norm is defined by
\[
\|u\|_{V^p(\R; H)}
:= \sup_{\{t_k\}_{k = 0}^K \in \mathcal{Z}}
\bigg(\sum_{k = 1}^K\|u(t_k) - u(t_{k-1})\|_H^p\bigg)^\frac{1}{p}.
\]

\noi
We also define $V^p_\text{rc}(\R; H)$
to be the closed subspace of all right-continuous functions
in $V^p(\R; H)$ such that
$\lim_{t \to -\infty} u(t) = 0$.

\smallskip
\noi
\textup{(iii)}
Let $s \in \R$.
We define $U^p_\Dl H^s$ (and $V^p_\Dl H^s$, respectively)
to be the spaces of all functions $u: \R \to H^s(\T^d)$
such that the following
$U^p_\Dl H^s$-norm (and $V^p_\Dl H^s$-norm, respectively)
is finite:
\[ \|u \|_{U^p_\Dl H^s} := \|S(-t) u\|_{U^p(\R; H^s)}
\quad \text{and} \quad
\|u \|_{V^p_\Dl H^s} := \|S(-t) u\|_{V^p(\R; H^s)}, \]

\noi
where $S(t) = e^{it\Dl}$ denotes the linear propagator for \eqref{NLS1}.
We use $V^p_{\text{rc}, \Dl} H^s$
to denote the subspace of right-continuous functions in
$V^p_\Dl H^s$.

 \end{definition}

\begin{remark}\label{REM:UpVp} \rm
Note that
the spaces $U^p(\R; H)$,
$V^p(\R; H)$,  and $V^p_\text{rc}(\R; H)$
are Banach spaces.
The closed subspace
of continuous functions in $U^p(\R; H)$ is also a Banach space.
Moreover, we have the following embeddings:
\begin{equation*}
U^p(\R; H) \hookrightarrow V^p_\text{rc}(\R; H)
\hookrightarrow U^q(\R; H)   \hookrightarrow L^\infty(\R; H)
\end{equation*}

\noi
for $ 1\leq p < q < \infty$.
Similar embeddings hold for $U^p_\Dl H^s$ and $V^p_\Dl H^s$.
\end{remark}

Next, we state a transference principle and an interpolation result.

\begin{lemma}\label{LEM:Xinterpolate}
\textup{(i)}
{\rm (Transference principle)}
Suppose that we have
\[ \big\| T(S(t) \phi_1, \dots, S(t) \phi_k)\big\|_{L^p_t L^q_x(\R\times \R^d)}
\les \prod_{j = 1}^k \|\phi_j\|_{L^2_x}\]

\noi
for some $1\leq p, q \leq \infty$.
Then, we have
\[ \big\| T(u_1,  \dots, u_k)\big\|_{L^p_t L^q_x(\R\times \R^d)}
\les \prod_{j = 1}^k \|u_j\|_{U^p_\Dl L^2_x}.\]

\noi
\textup{(ii)}
{\rm (Interpolation)}
 Let $E$ be a Banach space.
Suppose that $T: U^{p_1}\times \cdots \times U^{p_k} \to E$
is a bounded $k$-linear operator such that
\[ \|T(u_1, \dots, u_k)\|_{E} \leq C_1 \prod_{j = 1}^k \|u_j\|_{U^{p_j}}\]

\noi
for some $p_1, \dots, p_k > 2$.
Moreover, assume that there exists $C_2 \in (0, C_1]$
such that
\[ \|T(u_1, \dots, u_k)\|_{E} \leq C_2 \prod_{j = 1}^k \|u_j\|_{U^{2}}.\]

\noi
Then, we have
\[ \|T(u_1, \dots, u_k)\|_{E} \leq C_2 \bigg(\ln \frac{C_1}{C_2}+ 1\bigg)^k
 \prod_{j = 1}^k \|u_j\|_{V^2}\]

\noi
for $u_j \in V^2_\textup{rc}$, $j = 1, \dots, k$.
\end{lemma}
	
\noi
A transference principle as above has been commonly used
in the Fourier restriction norm method.
 See \cite[Proposition 2.19]{HHK}
for the proof of  Lemma \ref{LEM:Xinterpolate} (i).
The proof of the interpolation result follows from extending the trilinear result in \cite{HTT1}
to a general $k$-linear case.
See also  \cite[Proposition 2.20]{HHK}.

\smallskip

Let $\eta: \R \to [0, 1]$ be an even, smooth cutoff function
supported on $[-\frac{8}{5}, \frac{8}{5}]$
such that $\eta \equiv 1$ on $[-\frac{5}{4}, \frac{5}{4}]$.
Given a dyadic number $N \geq 1$, we
set
$\eta_1(\xi) = \eta(|\xi|)$
and
\[\eta_N(\xi) = \eta\bigg(\frac{|\xi|}{N}\bigg) - \eta\bigg(\frac{2|\xi|}{N}\bigg)\]

\noi
for $N \geq 2$.
Then, we define the Littlewood-Paley projection operator
$\P_N$ as the Fourier multiplier operator with symbol $\eta_N$.
Moreover, we define $\P_{\leq N}$ and $\P_{\geq N}$
by $\P_{\leq N} = \sum_{1 \leq M \leq N} \P_M$
and  $\P_{\geq N} = \sum_{ M \geq N} \P_M$.

\begin{definition} \label{DEF:X3}
\rm
\textup{(i)}
Let $s\in \R$.
We define $X^s(\R)$ to be the space of all tempered distributions
$u : \R \to H^s(\R^d)$
such that $\| u\|_{X^s(\R)} < \infty$,
where the $X^s$-norm is defined by
\[ \|u \|_{X^s (\R)} : = \bigg(
\sum_{\substack{N\geq 1\\\textup{dyadic}}} N^{2s}
\| \P_N u \|_{U^2_{\Dl}L^2}^2
\bigg)^\frac{1}{2}.
\]

\smallskip
\noi
\textup{(ii)}
Let $s\in \R$.
We define $Y^s(\R)$ to be the space of all tempered distributions
$u : \R \to H^s(\R^d)$
such that
for every $N\in \mathbb{N}$, the map $t \mapsto \P_N u(t)$
is in
$V^2_{\textup{rc}, \Dl} H^s$
and
$\| u\|_{Y^s(\R)} < \infty$,
where the $Y^s$-norm is defined by
\[ \|u \|_{Y^s(\R)} : = \bigg( \sum_{\substack{N\geq 1\\\textup{dyadic}}} N^{2s}
\| \P_N u\|_{V^2_\Dl L^2}^2\bigg)^\frac{1}{2}.
\]

\end{definition}

\noi
Recall the following embeddings:
\begin{equation}\label{inclusions}
U^2_\Dl H^s \hookrightarrow X^s \hookrightarrow
Y^s \hookrightarrow V^2_\Dl H^s \hookrightarrow U^p_\Dl H^s,
\end{equation}
for $p>2$.

Given an  interval $I \subset \R$,
we define the local-in-time versions $X^s(I)$
and $Y^s(I)$ of these spaces
as restriction norms.
For example, we define the $X^s(I)$-norm by
\[ \|u \|_{X^s(I)} = \inf\big\{ \|v\|_{X^s(\R)}: \, v|_I = u\big\}.\]

\noi
We also define the norm for the nonhomogeneous term:
\begin{align}
\| F\|_{N^s(I)} = \bigg\|\int_{t_0}^t S(t - t') F(t') dt'\bigg\|_{X^s(I)}.
\end{align}

\noi
In the following,  we will perform our analysis in
$X^s(I) \cap C(I; H^s)$, that is, in
a Banach subspace
of continuous functions in $X^s(I)$.
See Appendix \ref{SEC:A} for additional  properties of the $X^s(I)$-spaces.

We conclude this section by presenting
some basic
estimates involving these function spaces.	

\begin{lemma}\label{LEM:Ys1}
\textup{(i) (Linear estimates)}
Let $s \geq 0$ and $0 < T \leq \infty$.
Then, we have
\begin{align*}
\|   S(t) \phi \|_{X^s([0, T))}
& \leq \|\phi\|_{H^s}, \\
\|F\|_{N^s([0, T))}
 & \leq \sup_{\substack{v \in Y^{-s}([0, T))\\\|v\|_{Y^{-s} }= 1}}
 \bigg| \int_0^T \int_{\R^d} F(t, x) \cj{v(t, x)} dx dt\bigg|
\end{align*}

\noi
for all $\phi \in H^s(\R^d)$ and
$F \in L^1([0, T); H^s(\R^d))$.

\smallskip

\noi
\textup{(ii) (Strichartz estimates)}
Let $(q, r)$ be Schr\"odinger admissible
with $ q > 2$
and $p \geq \frac{2(d+2)}{d}$.
Then, for $0< T\leq \infty$ and  $N_1 \leq N_2$, we have
\begin{align}
\|  u \|_{L^q_t L^r_x([0, T)\times \R^d)} & \les \|u\|_{Y^0([0, T))}, \label{Ys2}\\
\|  u \|_{L^p_{t,x}([0, T)\times \R^d)}
& \les  \big\||\nb|^{\frac d2 - \frac {d+2}p} u\big\|_{Y^0([0, T))},
\label{Ys3}\\
\| \P_{N_1} u_1 \P_{N_2}u_2\|_{L^2_{t, x}([0, T)\times \R^d)}
& \les N_1^\frac{d-2}{2} \bigg(\frac{N_1}{N_2}\bigg)^{\frac 12-}
\|\P_{N_1} u_1\|_{Y^0([0, T))}\|\P_{N_2} u_2\|_{Y^0([0, T))}.
\label{Ys4}
\end{align}

\end{lemma}

\noi
Note that there is a slight loss of regularity in \eqref{Ys4}
since we use the $Y^0$-norm
on the right-hand side instead of the $X^0$-norm.
In view of \eqref{inclusions},
we may replace the $Y^0$-norms on the right-hand sides
of \eqref{Ys2}, \eqref{Ys3}, and \eqref{Ys4}
by the $X^0$-norm in the following.

\begin{proof}
In the following, we briefly discuss the proof of (ii).
See \cite{HHK, HTT1} for the proof of (i).
The first estimate \eqref{Ys2} follows from the Strichartz estimate \eqref{Str0},
Lemma \ref{LEM:Xinterpolate} (i),
and \eqref{inclusions}:
\[\|u\|_{L^q_tL^r_x}\lesssim \|u\|_{U^q_{\Dl}L^2}\lesssim \|u\|_{Y^0},\]

\noi
for $q>2$.
The second estimate
\eqref{Ys3} follows  from \eqref{Str0a} in a similar manner.
It remains to prove \eqref{Ys4}.
On the one hand,
the following bilinear refinement of the Strichartz estimate
by Bourgain \cite{Bo98} and Ozawa-Tsutsumi \cite{OT}:
\[\|\P_{N_1}S(t) \phi_1 \P_{N_2}S(t) \phi_2\|_{L^2_{t,x}}
\lesssim  N_1^{\frac{d-2}{2}}\bigg(\frac{N_1}{N_2}\bigg)^{\frac 12}
\|\P_{N_1}\phi_1\|_{L^2}\|\P_{N_2}\phi_2\|_{L^2}\]

\noi
and Lemma \ref{LEM:Xinterpolate} (i)
yield
\begin{equation}\label{biU2}
\|\P_{N_1}u_1\P_{N_2}u_2\|_{L^2_{t,x}}\lesssim  N_1^{\frac{d-2}{2}}
\bigg(\frac{N_1}{N_2}\bigg)^{\frac 12}\|\P_{N_1}u_1\|_{U^2_{\Dl}L^2}\|\P_{N_2}u_2\|_{U^2_{\Dl}L^2}.
\end{equation}

\noi
On the other hand,
by
Bernstein's inequality
and noting that
 $(4, \frac{2d}{d-1})$ is Strichartz admissible, we have
\begin{align*}
\|\P_{N_j}S(t) \phi_j\|_{L^4_{t,x}}
&\lesssim
N_j^{\frac{d-2}{4}}
\|\P_{N_j}S(t) \phi_j\|_{L^4_tL^{\frac{2d}{d-1}}_x}
\lesssim N_j^{\frac{d-2}{4}} \|\P_{N_j}\phi_j\|_{L^2}.
\end{align*}

\noi
Then, by Cauchy-Schwarz inequality
and
Lemma \ref{LEM:Xinterpolate} (i), we obtain
\begin{equation}\label{biU4}
\|\P_{N_1}u_1\P_{N_2}u_2\|_{L^2_{t,x}}\lesssim  N_1^{\frac{d-2}{4}}N_2^{\frac{d-2}{4}}\|\P_{N_1}u_1\|_{U^4_{\Dl}L^2}\|\P_{N_2}u_2\|_{U^4_{\Dl}L^2}.
\end{equation}

\noi
Hence, by Lemma \ref{LEM:Xinterpolate} (ii), with
\eqref{biU2} and  \eqref{biU4}, we have
\begin{align}
\|\P_{N_1}u_1 \P_{N_2}u_2\|_{L^2_{t,x}}
\lesssim  N_1^{\frac{d-2}{2}}
\bigg(\frac{N_1}{N_2}\bigg)^{\frac 12}
\bigg(\ln \Big(\frac{N_2}{N_1}\Big)+1\bigg)^2
\|\P_{N_1}u_1\|_{V^2_{\Dl}L^2}\|\P_{N_2}u_2\|_{V^2_{\Dl}L^2}.
\label{biU5}
\end{align}

\noi
Finally, \eqref{Ys4} follows from \eqref{inclusions} and \eqref{biU5}.
\end{proof}

Similarly to the usual Strichartz estimate  \eqref{Str0a},
 the derivative loss in \eqref{Ys3} depends
only on the size of the spatial frequency support and not its location.
Namely,  if the spatial frequency support of $\ft{u}(t, \xi)$
is contained in a cube of side length $N$ for all $t \in \R$,
then we have
\begin{equation}
\|  u \|_{L^p_{t,x}([0, T)\times \R^d)}
 \les  N^{\frac d2 - \frac{d+2}{p}}
\| u\|_{Y^0([0, T))}.
\label{Ys5}
\end{equation}

\noi
This is a direct consequence of    \eqref{Str0b}.

Lastly, we  recall Schur's test for readers' convenience.

\begin{lemma}[Schur's test]\label{LEM:Schur}
Suppose that we have
\[ \sup_m \sum_{n} |K_{m, n}| + \sup_n \sum_{m} |K_{m, n}| < \infty\]

\noi
for some $K_{m, n} \in \C$, $m, n \in \Z$.
Then, we have
\[\sum_{m, n } K_{m, n} a_m  b_n \les \|a_m\|_{\l^2_m}\|b_n\|_{\l^2_n}.\]
for any $\ell^2$-sequences $\{a_m\}_{m\in\Z}$ and  $\{b_n\}_{n\in\Z}$.
\end{lemma}

\section{Probabilistic nonlinear estimates}

\label{SEC:4}

In this section, we prove the key nonlinear estimates
in the critical regularity $s_\textup{crit} = \frac{d-2}{2}$.
In the next section, we use them to prove Theorems \ref{THM:1} and \ref{THM:2}.
Given $z(t) = S(t) \phi^\o$, define $\G$ by
\begin{equation}
\G v(t) =\mp i  \int_0^t S(t-t') \N (v+z)(t') dt',
\label{NLS5}
\end{equation}
where $\N(v+z)=|v+z|^2(v+z)$.
\noi
Then, we have the following nonlinear estimates.

\begin{proposition}\label{PROP:NL1}
Given $d\geq 3$, let $s \in (s_d, s_\textup{crit}]$, where $s_d$ is defined in \eqref{Sd1}. Given $\phi \in H^s(\R^d)$,
let $\phi^\o$ be its Wiener randomization defined in \eqref{R1},
satisfying \eqref{R2}.

\smallskip

\noi
\textup{(i)}
Let
 $0<T \leq 1$. Then,  there exists $0<\theta\ll1$ such that we have
\begin{align}
\|\G v\|_{X^\frac{d-2}{2}([0, T))}
& \leq C_1
\big(\|v\|_{X^\frac{d-2}{2}([0, T))} ^3 + T^\theta R^3\big),  \label{nl1a}\\
\|\G v_1 - \G v_2  \|_{X^\frac{d-2}{2}([0, T))}
& \leq C_2
\Big(\sum_{j = 1}^2 \|v_j\|_{X^\frac{d-2}{2}([0, T))} ^2
+ T^\theta R^2\Big)
\|v_1 -v_2 \|_{X^\frac{d-2}{2}([0, T))},
\label{nl1b}
\end{align}

\noi
for all $v, v_1, v_2 \in X^{\frac{d-2}{2}}([0, T))$
and  $R>0$,
outside a set of probability $\leq  C \exp\big(-c \frac{R^2}{\|\phi\|_{H^s}^2}\big)$.

\noi
\textup{(ii)}
Given $0 < \eps \ll1$, define $\wt \G$ by
\begin{equation}
\wt \G v(t) =  \mp i \int_0^t S(t-t') \N (v+\eps z)(t') dt'.
\label{NLS6}
\end{equation}

\noi
Then,  we have
\begin{align}
\|\wt \G v\|_{X^\frac{d-2}{2}(\R )}
& \leq C_3
\big(\|v\|_{X^\frac{d-2}{2}(\R)} ^3 + R^3\big),  \label{nl1c}\\
\|\wt \G v_1 - \wt \G v_2  \|_{X^\frac{d-2}{2}(\R)}
& \leq C_4
\Big(\sum_{j = 1}^2 \|v_j\|_{X^\frac{d-2}{2}(\R)} ^2
+ R^2\Big)
\|v_1 -v_2 \|_{X^\frac{d-2}{2}(\R)},
\label{nl1d}
\end{align}

\noi
for all $v, v_1, v_2 \in X^{\frac{d-2}{2}}(\R)$
and  $R>0$,
outside a set of probability $\leq C \exp\big(-c \frac{R^2}{\eps ^2 \|\phi\|_{H^s}^2}\big)$.

\end{proposition}

\begin{proof}
(i)
Let $0<T \leq 1$.
We only prove \eqref{nl1a} since  \eqref{nl1b} follows in a  similar manner.
Given $N \geq 1$, define $\G_N$ by
\begin{equation}
\G_N v(t) =\mp i  \int_0^t S(t-t')\P_{\leq N} \N (v+z)(t') dt'.
\label{NNLS1}
\end{equation}

\noi
By Bernstein's and H\"older's inequalities, we have
\begin{align}
\|\P_{\leq N} \N (v+z)\|_{L^1_t ([0, T); H^{\frac{d-2}{2}}_x)}
& \les N^\frac{d-2}{2}
\| \N (v+z)\|_{L^1_t L^2_x}\notag \\
& \les N^\frac{d-2}{2}
\| v\|_{L^3_t([0, T);  L^6_x)}^3
+ N^\frac{d-2}{2}
\| z\|_{L^3_t([0, T);  L^6_x)}^3.
\label{NNLS2}
\end{align}

\noi
On the one hand,
it follows from Lemma \ref{PROP:Str1}
that the second term on the right-hand side of \eqref{NNLS2}
is finite almost surely.
On the other hand,
noting that $(3, \frac{6d}{3d-4})$ is Strichartz admissible,
it follows from   Sobolev's inequality and
\eqref{Ys2} in Lemma \ref{LEM:Ys1} that
\begin{align}
\| v\|_{L^3_t([0, T);  L^6_x)}
\les \big\| \jb{\nb}^\frac{d-2}{3}v\big\|_{L^3_t([0, T);  L^\frac{6d}{3d-4}_x)}
\les \big\| v \|_{X^\frac{d-2}{3}([0, T))} < \infty.
\label{NNLS3}
\end{align}

\noi
Therefore,
by Lemma \ref{LEM:Ys1} (i),  we have
\begin{align}
\|\G_N v(t)\|_{X^\frac{d-2}{2}}
  \lesssim \sup_{\substack{v_4 \in Y^{0}([0, T))\\\|v_4\|_{Y^{0}} = 1}}
 \bigg| \int_0^T \int_{\R^d} \jb{\nb}^\frac{d-2}{2} \N(v+z)(t, x) \cj{v_4(t, x)} dx dt\bigg|,
\label{nl1}
\end{align}

\noi
almost surely, where $v_4 = \P_{\leq N} v_4$.
In the following,
we estimate the right-hand side of \eqref{nl1},
independently of the cutoff size $N \geq 1$,
by performing a case-by-case analysis of expressions of the form:
\begin{align}
\bigg| \int_0^T \int_{\R^d}
\jb{\nb}^\frac{d-2}{2} ( w_1 w_2 w_3 )v_4 dx dt\bigg|
\label{nl2}
\end{align}

\noi
where $\|v_4\|_{Y^{0}([0, T))} \leq 1$
and $w_j=  v$ or $z$,  $j = 1, 2, 3$.
As a result,
by taking $N \to \infty$,
the same estimates hold for $\G v$ without any cutoff,
thus yielding \eqref{nl1a}.

Before proceeding further, let us simplify some of the notations.
In the following, we drop the complex conjugate sign.
We  also denote $X^s([0, T))$ and $Y^s([0, T))$ by $X^s$ and $Y^s$ since $T$ is fixed.
Similarly, it is understood that the time integration in $L^p_{t, x}$
is over $[0, T)$.
Lastly, in most of the cases,
we dyadically decompose
 $w_j = v_j$ or $z_j$, $j = 1, 2, 3$,
and $v_4$ such that their spatial frequency supports are $\{ |\xi_j|\sim N_j\}$
for some dyadic $N_j \geq 1$
but still denote them as $w_j= v_j \text{ or } z_j$, $j = 1, 2, 3$, and $v_4$.
Note that, if we can afford a small derivative loss in the largest frequency,
there is no difficulty in summing over the dyadic blocks $N_j$, $j = 1, \dots, 4$.

\medskip
\noi
{\bf  Case (1):} $v v v$ case.

In this case, we do not need to  perform dyadic decompositions
and we divide the frequency spaces into
$\{|\xi_1| \geq |\xi_2|, |\xi_3|\}$,
$\{|\xi_2| \geq |\xi_1|, |\xi_3|\}$,
and
$\{|\xi_3| \geq |\xi_1|, |\xi_2|\}$.
Without loss of generality, assume that $|\xi_1| \geq |\xi_2|, |\xi_3|$.
By $L^\frac{2(d+2)}{d}_{t,x}L^{d+2}_{t,x}L^{d+2}_{t,x}L^{\frac{2(d+2)}{d}}_{t,x}$-H\"older's inequality,
\eqref{Ys3} in Lemma \ref{LEM:Ys1}, and \eqref{inclusions},
we have
\begin{align*}
\bigg|\int_0^T\int_{ \R^d} \jb{\nb}^\frac{d-2}{2} v_1 v_2 v_3 v_4 dx dt \bigg|
& \leq \| \jb{\nb}^\frac{d-2}{2} v_1\|_{L^\frac{2(d+2)}{d}_{t, x}}
\|v_2\|_{L^{d+2}_{t, x}}\|v_3\|_{L^{d+2}_{t, x}}
\|v_4\|_{L^{\frac{2(d+2)}{d}}_{t, x}}\\
& \les
\prod_{j = 1}^3 \|v_j\|_{Y^\frac{d-2}{2}} \|v_4\|_{Y^0}
\les \prod_{j = 1}^3 \|v_j\|_{X^\frac{d-2}{2}}.
\end{align*}

\medskip

\noi
{\bf Case (2):} $zz z$ case. \quad

Without loss of generality, assume $N_3 \geq N_2 \geq N_1$.

\smallskip

\noi
{\bf $\bullet$  Subcase (2.a):} $N_2 \sim N_3$.

By $L^{d+2}_{t,x}L^{4}_{t,x}L^{4}_{t,x}L^{\frac{2(d+2)}{d}}_{t,x}$-H\"older's inequality,
we have
\begin{align*}
\bigg|\int_0^T\int_{ \R^d}  z_1 z_2 \jb{\nb}^\frac{d-2}{2} z_3 v_4 dx dt \bigg|
& \les \|z_1\|_{L^{d+2}_{t, x}} \|\jb{\nb}^\frac{d-2}{4}z_2\|_{L^{4}_{t, x}}
\|\jb{\nb}^\frac{d-2}{4} z_3 \|_{L^{4}_{t, x}}\| v_4 \|_{L^\frac{2(d+2)}{d}_{t, x}}.
\end{align*}

\noi
Hence,
by Lemmata \ref{PROP:Str1} and \ref{LEM:Ys1},
the contribution to \eqref{nl1} in this case is at most
$\les T^{0+} R^3$
outside a set of probability
\begin{equation*}
\leq
C\exp\bigg(- c\frac{R^2}{T^{\frac{2}{d+2}-}\|\phi\|_{L^2}^2}\bigg)
+ C\exp\bigg(- c\frac{R^2}{T^{\frac{1}{2}-}\|\phi\|_{H^{\frac{d-2}{4}+}}^2}\bigg)
\end{equation*}

\noi
as long as $s > \frac{d-2}{4}$.
Note that $s$ needs to be strictly greater than $\frac {d-2}{4}$
due to the summations over dyadic blocks.
See \cite{BOP1} for more details.
Similar comments apply in the following.

\medskip

\noi
{\bf $\bullet$  Subcase (2.b):} $N_3 \sim N_4 \gg N_1, N_2$.

\smallskip

\noi
$\circ$ \underline{Subsubcase (2.b.i):} $N_1, N_2 \ll N_3^\frac{1}{d-1}$.

For small $\al > 0$,
it follows from  Cauchy-Schwarz inequality and Lemma \ref{LEM:Ys1} that
\begin{align}
\|z_2  \jb{\nb}^\frac{d-2}{2}z_3\|_{L^2_{t, x}}
& \les N_3^\frac{d-2}{2} \|z_2\|_{L^4_{t, x} }^\al \|z_3\|_{L^4_{t, x} }^\al \|z_2  z_3\|_{L^2_{t, x}}^{1-\al}\notag\\
& \les
N_2^{\frac{d-1}{2}-s-\frac{d-1}{2}\al-}
N_3^{\frac{d-3}{2} - s+\frac{1}{2}\al +}
\prod_{j = 2}^3\big(\|\jb{\nb}^s z_j\|_{L^4_{t, x} }^\al
\|\P_{N_j} \phi^\o\|_{H^{s}}^{1-\al}\big).
\label{T1a}
\end{align}

\noi
Then,
by \eqref{T1a} and the bilinear estimate \eqref{Ys4} in Lemma \ref{LEM:Ys1},
 we have
\begin{align*}
 \bigg|\int_0^T  \int_{ \R^d}    z_1 z_2 &   \jb{\nb}^\frac{d-2}{2}  z_3 v_4 dx dt \bigg|
   \les \|z_2 \jb{\nb}^\frac{d-2}{2} z_3\|_{L^2_{t, x}} \|z_1 v_4\|_{L^2_{t, x}}\\
& \les N_1^{\frac{d-1}{2} - s-}N_2^{\frac{d-1}{2}  - s-\frac{d-1}{2}\al-}
N_3^{\frac{d-4}{2}-s +\frac{1}{2}\al+}
\\
& \hphantom{XXXXXX}
\times
\|\P_{N_1} \phi^\o\|_{H^{s}}
\prod_{j = 2}^3\big(\|\jb{\nb}^s z_j\|_{L^4_{t, x} }^\al
\|\P_{N_j} \phi^\o\|_{H^{s}}^{1-\al}\big)
\|v_4\|_{Y^0}\\
& \les N_3^{\frac{d-2}{2} - \frac{d+1}{d-1} s + }
\|\P_{N_1} \phi^\o\|_{H^{s}}
\prod_{j = 2}^3\big(\|\jb{\nb}^s z_j\|_{L^4_{t, x} }^\al
\|\P_{N_j} \phi^\o\|_{H^{s}}^{1-\al}\big)
\|v_4\|_{Y^0}.
\end{align*}

\noi
Hence,
by Lemmata \ref{PROP:Str1} and  \ref{LEM:Hs},
the contribution to \eqref{nl1} in this case is at most
$\les T^{0+} R^3$
outside a set of probability
\begin{equation*}
\leq
 C\exp\bigg(-c \frac {R^2}{  T^{\frac{1}{2}-} \|\phi\|_{H^{s}}^{2}}\bigg)
+ C\exp\bigg(-c \frac {R^2}{  \|\phi\|_{H^{s}}^{2}}\bigg)
\end{equation*}

\noi
as long as
\begin{equation}
 s> \frac{d-1}{d+1}\cdot
  \frac{d-2}{2}=s_d
\label{nl3}
 \end{equation}

\noi
and $\al < 1-  \frac{2}{d-1} s$.

\smallskip

\noi
$\circ$ \underline{Subsubcase (2.b.ii):} $N_2\ges N_3^\frac{1}{d-1} \gg N_1$.

By H\"older's inequality
and the bilinear estimate \eqref{Ys4} in Lemma \ref{LEM:Ys1},
 we have
\begin{align*}
\bigg|\int_0^T\int_{ \R^d}   z_1 z_2 & \jb{\nb}^\frac{d-2}{2} z_3 v_4 dx dt \bigg|
 \les \|z_2\|_{L^4_{t, x}}
\|\jb{\nb}^\frac{d-2}{2} z_3 \|_{L^4_{t, x}}\|z_1 v_4 \|_{L^2_{t, x}}\\
& \les
N_1^{\frac{d-1}{2}-s-} N_2^{-s}
N_3^{\frac{d-3}{2}-s+}
\|\P_{N_1}\phi^\o\|_{H^s}
\prod_{j = 2}^3 \|\jb{\nb}^s z_j\|_{L^4_{t, x}}
\|v_4\|_{Y^0}\\
& \les N_3^{\frac{d-2}{2} - \frac{d+1}{d-1} s + }
\|\P_{N_1}\phi^\o\|_{H^s}
\prod_{j = 2}^3 \|\jb{\nb}^s z_j\|_{L^4_{t, x}}
\|v_4\|_{Y^0}.
\end{align*}

\noi
Hence,
by  Lemmata \ref{PROP:Str1}
and  \ref{LEM:Hs},
the contribution to \eqref{nl1} in this case is at most
$\les T^{0+} R^3$
outside a set of probability
\begin{equation*}
\leq
C\exp\bigg(-c \frac {R^2}{  T^{\frac{1}{2}-} \|\phi\|_{H^{s}}^{2}}\bigg)
+ C\exp\bigg(-c \frac {R^2}{  \|\phi\|_{H^{s}}^{2}}\bigg)
\end{equation*}

\noi
as long as
\eqref{nl3} is satisfied.

\smallskip

\noi
$\circ$ \underline{Subsubcase (2.b.iii):} $N_1, N_2\ges N_3^\frac{1}{d-1} $.

By $L^{\frac{6(d+2)}{d+4}}_{t,x}L^{\frac{6(d+2)}{d+4}}_{t,x}L^{\frac{6(d+2)}{d+4}}_{t,x}L^{\frac{2(d+2)}{d}}_{t,x}$-H\"older's inequality
and  \eqref{Ys3} in
Lemma \ref{LEM:Ys1},
  we have
\begin{align*}
\bigg|\int_0^T \int_{\R^d}  z_1 z_2 \jb{\nb}^\frac{d-2}{2} z_3 v_4 dx dt \bigg|
& \les N_3^{\frac{d-2}{2} -  \frac{d+1}{d-1} s }
\prod_{j =1}^3 \| \jb{\nb}^s z_j\|_{L^\frac{6(d+2)}{d+4}_{t, x}}
\|v_4\|_{Y^0}.
\end{align*}

\noi
Hence,
by Lemma \ref{PROP:Str1},
the contribution to \eqref{nl1} in this case is at most
$\les T^{0+} R^3$
outside a set of probability
\begin{equation*}
\leq  C\exp\bigg(-c \frac {R^2}{  T^{\frac{d+4}{3(d+2)}-} \|\phi\|_{H^{s}}^{2}}\bigg)
\end{equation*}

\noi
as long as
\eqref{nl3} is satisfied.

\medskip

\noi
{\bf Case (3):} $v v z$ case.

Without loss of generality,   assume $N_1 \geq N_2$.

\medskip

\noi
{\bf $\bullet$  Subcase (3.a):} $N_1 \ges N_3$.

In the following, we apply dyadic decompositions only to $v_1$,  $v_2$,
and $z_3$.
In this case, we have $N_1 \sim \max (N_2, N_3, |\xi_4|)$,
where $\xi_4$ is the spatial frequency of $v_4$.
Then, by   H\"older's inequality,
\eqref{Ys4}, and \eqref{Ys3},
 we have
\begin{align*}
\bigg|
\int_0^T \int_{\R^d} \jb{\nb}^\frac{d-2}{2}  v_1   v_2 z_3 v_4 dx dt \bigg|
& \les
\sum_{N_1 \ges N_2, N_3}
  \|\jb{\nabla}^{\frac{d-2}{2}}\P_{N_1}v_1 \P_{N_2} v_2\|_{L^2_{t,x}}\|\P_{N_3}z_3\|_{L^{d+2}_{t,x}}
  \|v_4\|_{L^{\frac{2(d+2)}{d}}_{t,x}}\\
&\les
\sum_{N_1 \geq N_2}
\bigg(\frac{N_2}{N_1}\bigg)^{\frac{1}{2}-}\prod_{j = 1}^2 \|\P_{N_j}v_j\|_{X^\frac{d-2}{2}}
\sum_{N_3}  \|\P_{N_3}z_3\|_{L^{d+2}_{t, x}}
\|v_4\|_{Y^{0}}
\intertext{By Lemma \ref{LEM:Schur}
and summing over $N_3$ with a slight loss of derivative, }
&\les
\prod_{j = 1}^2 \|v_j\|_{X^\frac{d-2}{2}}
 \| \jb{\nb}^{0+}z_3\|_{L^{d+2}_{t, x}}
\|v_4\|_{Y^{0}}.
\end{align*}

\noi
Hence,
by Lemma \ref{PROP:Str1},
the contribution to \eqref{nl1} in this case is at most
$
\les T^{0+}R \prod_{j = 1}^2 \|v_j\|_{X^\frac{d-2}{2}} $
outside a set of probability
\begin{equation*}
\leq  C\exp\bigg(-c \frac {R^2}{  T^{\frac{2}{d+2}-} \|\phi\|_{H^{0+}}^{2}}\bigg)
\end{equation*}

\noi
as long as $s > 0$.

\medskip

\noi
{\bf $\bullet$  Subcase (3.b):} $N_3\sim N_4 \gg N_1 \geq N_2$.

\smallskip

\noi
$\circ$ \underline{Subsubcase (3.b.i):} $N_1\ges  N_3^\frac 1{d-1}$.

By H\"older's inequality followed by  \eqref{Ys3} and \eqref{Ys4} in Lemma \ref{LEM:Ys1},
we have
\begin{align*}
\bigg|\int_0^T \int_{\R^d}  v_1  v_2  & \jb{\nb}^\frac{d-2}{2} z_3  v_4 dx dt \bigg|
  \les \|v_1\|_{L^\frac{2(d+2)}{d}_{t, x}}
\|\jb{\nb}^\frac{d-2}{2} z_3 \|_{L^{d+2}_{t, x}}\|v_2 v_4 \|_{L^2_{t, x}}\\
& \les
N_1^{-\frac{d-2}{2}}N_2^{\frac 12-}N_3^{\frac{d-3}{2}-s+}
 \|v_1\|_{X^\frac{d-2}{2}} \|v_2\|_{X^{\frac {d-2}2}} \|\jb{\nb}^{s} z_3\|_{L^{d+2}_{t, x}}
 \|v_4\|_{Y^0}\\
 & \les
N_3^{\frac{d-3}{d-1}\frac{d-2}{2} - s + }
 \|v_1\|_{X^\frac{d-2}{2}} \|v_2\|_{X^{\frac {d-2}2}} \|\jb{\nb}^{s} z_3\|_{L^{d+2}_{t, x}}
 \|v_4\|_{Y^0}.
\end{align*}

\noi
Hence,
by  Lemma \ref{PROP:Str1},
the contribution to \eqref{nl1} in this case is at most
$
\les T^{0+} R \prod_{j = 1}^2 \|v_j\|_{X^{\frac {d-2}2}} $
outside a set of probability
\begin{equation*}
\leq  C\exp\bigg(-c \frac {R^2}{  T^{\frac{2}{d+2}-} \|\phi\|_{H^{s}}^{2}}\bigg)
\end{equation*}

\noi
as long as
\begin{align}
s  > \frac{d-3}{d-1}
\cdot\frac{d-2}{2}.
\label{nl4}
\end{align}

\noi
Note that the condition \eqref{nl4}
is less restrictive than \eqref{nl3}.

\smallskip

\noi
$\circ$ \underline{Subsubcase (3.b.ii):} $N_2 \leq N_1\ll  N_3^\frac 1{d-1}$.

For small $\al > 0$,
it follows from
H\"older's inequality and Lemma \ref{LEM:Ys1} that
\begin{align}
\|v_1  \jb{\nb}^\frac{d-2}{2}z_3\|_{L^2_{t, x}}
& \les N_3^\frac{d-2}{2} \|v_1\|_{L^{\frac{2(d+2)}{d}}_{t, x} }^\al \|z_3\|_{L^{d+2}_{t, x} }^\al
\|v_1  z_3\|_{L^2_{t, x}}^{1-\al}\notag\\
& \les
N_1^{\frac 12 -\frac {d-1}2 \al-}
N_3^{\frac{d-3}{2} - s+\frac{1}{2}\al +}
\|v_1\|_{X\frac{d-2}{2}}
\|\jb{\nb}^s z_3\|_{L^{d+2}_{t, x} }^\al
\|\P_{N_3} \phi^\o\|_{H^{s}}^{1-\al}.
\label{nl4a}
\end{align}

\noi
Then,
by \eqref{nl4a} and \eqref{Ys4} in Lemma \ref{LEM:Ys1},
 we have
\begin{align*}
\bigg|\int_0^T & \int_{\R^d}   v_1  v_2   \jb{\nb}^\frac{d-2}{2} z_3  v_4 dx dt \bigg|
  \les
  \|v_1  \jb{\nb}^\frac{d-2}{2}z_3\|_{L^2_{t, x}}
\|v_2 v_4 \|_{L^2_{t, x}}\\
& \les
N_1^{\frac 12 -\frac {d-1}2 \al-}
N_2^{\frac 12-}
N_3^{\frac{d-4}{2} - s+\frac{1}{2}\al +}
 \|v_1\|_{X^\frac{d-2}{2}} \|v_2\|_{X^{\frac {d-2}2}}
 \|\jb{\nb}^s z_3\|_{L^{d+2}_{t, x} }^\al
\|\P_{N_3} \phi^\o\|_{H^{s}}^{1-\al} \|v_4\|_{Y^0}.
\end{align*}

\noi
Hence,
by  Lemmata \ref{PROP:Str1} and \ref{LEM:Hs},
the contribution to \eqref{nl1} in this case is at most
$
\les T^{0+} R \prod_{j = 1}^2 \|v_j\|_{X^{\frac {d-2}2}} $
outside a set of probability
\begin{equation*}
\leq  C\exp\bigg(-c \frac {R^2}{  T^{\frac{2}{d+2}-} \|\phi\|_{H^{s}}^{2}}\bigg)
+  C\exp\bigg(-c \frac {R^2}{  \|\phi\|_{H^{s}}^{2}}\bigg)
\end{equation*}

\noi
as long as \eqref{nl4}
is satisfied
and $\al< \frac{1}{d-1}$.

\medskip

\noi
{\bf Case (4):} $v z z$ case.

Without loss of generality,  assume $N_3 \geq N_2$.

\smallskip

\noi
{\bf $\bullet$  Subcase (4.a):} $N_1 \ges N_3$.

By $L^\frac{2(d+2)}{d}_{t,x}L^{d+2}_{t,x}L^{d+2}_{t,x}L^\frac{2(d+2)}{d}_{t,x}$-H\"older's inequality
and \eqref{Ys3} in Lemma \ref{LEM:Ys1},
we have
\begin{align*}
\bigg|\int_0^T \int_{\R^d} \jb{\nb}^\frac{d-2}{2} v_1 z_2 z_3 v_4 dx dt \bigg|
& \les  \|v_1\|_{X^{\frac{d-2}{2}}}
\|z_2\|_{L^{d+2}_{t, x}} \|z_3\|_{L^{d+2}_{t, x}} \|v_4\|_{Y^0}.
\end{align*}

\noi
Hence,
by  Lemma \ref{PROP:Str1},
the contribution to \eqref{nl1} in this case is at most
$
\les T^{0+}R^2  \|v_1\|_{X^\frac{d-2}{2}} $
outside a set of probability
\begin{equation}
\leq  C\exp\bigg(-c \frac {R^2}{  T^{\frac{2}{d+2}-} \|\phi\|_{H^{0+}}^{2}}\bigg)
\label{T4b}
\end{equation}

\noi
as long as $s> 0$.
As before, we have
 $\|\phi\|_{H^{0+}}$  instead of $\|\phi\|_{L^2}$
 in \eqref{T4b},
 allowing us to sum over $N_2$ and $N_3$.
If $N_3 \ges \max(N_1, N_4)$,
then this also allows us to sum over $N_1$ and $N_4$.
Otherwise,  we have $N_1\sim N_4\gg N_3$.
In this case,
 we can use Cauchy-Schwarz inequality to sum over $N_1 \sim N_4$.

\medskip

\noi
{\bf $\bullet$  Subcase (4.b):} $N_3 \gg N_1$.

First, suppose that $N_2 \sim N_3$.
Note that we must have $N_3 \ges N_4$ in this case.
Then, by $L^{d+2}_{t, x}L^4_{t, x}L^4_{t, x}L^\frac{2(d+2)}{d}_{t, x}$-H\"older's inequality with
\eqref{Ys3} in
Lemma \ref{LEM:Ys1},
we have
\begin{align*}
\bigg|\int_0^T \int_{ \R^d}  v_1 z_2  \jb{\nb}^\frac{d-2}{2} z_3 v_4 dx dt \bigg|
&  \les\|v_1\|_{X^\frac{d-2}{2}} \|\jb{\nb}^\frac{d-2}{4}z_2\|_{L^4_{t, x}}
\|\jb{\nb}^\frac{d-2}{4} z_3 \|_{L^4_{t, x}}\|v_4 \|_{Y^0}\\
& \les
N_3^{\frac{d-2}{2} - 2s}
\|v_1\|_{X^\frac{d-2}{2}}
\|\jb{\nb}^s z_2\|_{L^4_{t, x}}
\|\jb{\nb}^s z_3 \|_{L^4_{t, x}} \|v_4 \|_{Y^0}.
\end{align*}

\noi
Hence,
by  Lemma \ref{PROP:Str1},
the contribution to \eqref{nl1} in this case is at most
$
\les T^{0+} R^2  \|v_1\|_{X^{\frac {d-2}2}} $
outside a set of probability
\begin{equation*}
\leq  C\exp\bigg(-c \frac {R^2}{  T^{\frac{1}{2}-} \|\phi\|_{H^{s}}^{2}}\bigg)
\end{equation*}

\noi
as long as $ s> \frac {d-2}4$.

Hence, it remains to consider the case $N_3 \sim N_4 \gg N_1, N_2$.

\smallskip

\noi
$\circ$ \underline{Subsubcase (4.b.i):} $N_1, N_2\ll N_3^\frac 1{d-1}$.

By \eqref{T1a} and  \eqref{Ys4} in Lemma \ref{LEM:Ys1},
we have
\begin{align*}
\bigg|\int_0^T \int_{ \R^d} &  v_1 z_2  \jb{\nb}^\frac{d-2}{2} z_3 v_4 dx dt \bigg|
 \les \|z_2\jb{\nb}^\frac{d-2}{2} z_3 \|_{L^2_{t, x}}\|v_1 v_4 \|_{L^2_{t, x}}\\
& \les
N_1^{\frac 12 -}N_2^{\frac{d-1}{2}-s-\frac{d-1}{2}\al-}
 N_3^{\frac{d-4}{2}-s+\frac{1}{2}\al+}
\|v_1\|_{X^\frac{d-2}{2}}
\prod_{j = 2}^3\big(\|\jb{\nb}^s z_j\|_{L^4_{t, x} }^\al
\|\P_{N_j} \phi^\o\|_{H^{s}}^{1-\al}\big)
\|v_4\|_{Y^0}.
\end{align*}

\noi
Hence,
by  Lemmata \ref{PROP:Str1} and  \ref{LEM:Hs},
the contribution to \eqref{nl1} in this case is at most
$
\les T^{0+} R^2  \|v_1\|_{X^\frac{d-2}{2}} $
outside a set of probability
\begin{equation*}
\leq
 C\exp\bigg(-c \frac {R^2}{  T^{\frac{1}{2}-} \|\phi\|_{H^{s}}^{2}}\bigg)
+ C\exp\bigg(-c \frac {R^2}{  \|\phi\|_{H^{s}}^{2}}\bigg)
\end{equation*}

\noi
as long as
\begin{align}
s > \frac{(d-2)^2}{2d} = \frac{d-2}{d}\cdot \frac{d-2}{2}
\label{nl5}
\end{align}

\noi
and $\al< 1 -  \frac{2}{d-1} s$.
Note that the condition \eqref{nl5}
is less restrictive than \eqref{nl3}
and thus does not add a further constraint.

\smallskip

\noi
$\circ$ \underline{Subsubcase (4.b.ii):} $N_1\ll N_3^{\frac 1{d-1}} \les N_2$.

By H\"older's inequality and
\eqref{Ys4} in Lemma \ref{LEM:Ys1},
we have
\begin{align*}
\bigg|\int_0^T \int_{ \R^d}  v_1 z_2 & \jb{\nb}^\frac{d-2}{2} z_3 v_4 dx dt \bigg|
 \les \|z_2\|_{L^4_{t, x}} \|\jb{\nb}^\frac{d-2}{2} z_3 \|_{L^4_{t, x}}\|v_1 v_4 \|_{L^2_{t, x}}\\
& \les
N_1^{\frac 12 - } N_2^{-s} N_3^{\frac{d-3}{2} - s+}
\|v_1\|_{X^\frac{d-2}{2}} \prod_{j = 2}^3 \|\jb{\nb}^s z_j\|_{L^4_{t, x}}\|v_4\|_{Y^0}.
\end{align*}

\noi
Hence,
by Lemma \ref{PROP:Str1},
the contribution to \eqref{nl1} in this case is at most
$
\les T^{0+}R^2  \|v_1\|_{X^{\frac {d-2}2}} $
outside a set of probability
\begin{equation*}
\leq C\exp\bigg(-c \frac {R^2}{ T^{\frac 12-} \|\phi\|_{H^{s}}^{2}}\bigg)
\end{equation*}

\noi
as long as
\eqref{nl5} is satisfied.

\smallskip

\noi
$\circ$ \underline{Subsubcase (4.b.iii):} $N_2 \ll N_3^{\frac 1{d-1}} \les N_1$.

By H\"older's inequality and  Lemma \ref{LEM:Ys1},
we have
\begin{align*}
\bigg|\int_0^T \int_{\R^d}  v_1 z_2 & \jb{\nb}^\frac{d-2}{2} z_3 v_4 dx dt \bigg|
 \les \|v_1\|_{L^\frac{2(d+2)}{d}_{t, x}} \|\jb{\nb}^\frac{d-2}{2} z_3 \|_{L^{d+2}_{t, x}}\|z_2 v_4 \|_{L^2_{t, x}}\\
& \les
N_1^{ - \frac {d-2}{2}} N_2^{\frac{d-1}{2} - s-}
N_3^{\frac{d-3}{2}-s+}
\|v_1\|_{X^{\frac {d-2}2}} \|\P_{N_2}\phi^\o \|_{H^s}
\|\jb{\nb}^s z_3\|_{L^{d+2}_{t, x}}\|v_4\|_{Y^0}.
\end{align*}

\noi
Hence,
by Lemmata \ref{LEM:Hs} and  \ref{PROP:Str1},
the contribution to \eqref{nl1} in this case is at most
$
\les T^{0+}  R^2  \|v_1\|_{X^{\frac {d-2}2}} $
outside a set of probability
\begin{equation*}
\leq
C\exp\bigg(-c \frac {R^2}{  \|\phi\|_{H^{s}}^{2}}\bigg)
+
C\exp\bigg(-c \frac {R^2}{ T^{\frac 2{d+2}-}  \|\phi\|_{H^{s}}^{2}}\bigg)
\end{equation*}

\noi
as long as \eqref{nl5} is satisfied.

\smallskip

\noi
$\circ$ \underline{Subsubcase (4.b.iv):} $N_1, N_2\ges N_3^{\frac 1{d-1}}$.

By $L^\frac {2(d+2)}{d}_{t,x}L^{d+2}_{t,x}L^{d+2}_{t,x}L^\frac{2(d+2)}{d}_{t,x}$-H\"older's inequality
and
\eqref{Ys3} in Lemma \ref{LEM:Ys1},
we have
\begin{align*}
\bigg|\int_0^T \int_{\R^d}  v_1 z_2 \jb{\nb}^\frac{d-2}{2} z_3 v_4 dx dt \bigg|
& \les \|v_1\|_{L^\frac{2(d+2)}{d}_{t, x}} \|z_2\|_{L^{d+2}_{t, x}}
\|\jb{\nb}^\frac{d-2}{2}  z_3 \|_{L^{d+2}_{t, x}}\| v_4 \|_{L^\frac{2(d+2)}{d}_{t, x}}\\
& \les
N_1^{- \frac {d-2}2 } N_2^{-s} N_3^{\frac {d-2}2 -s}
\|v_1\|_{X^\frac{d-2}{2}} \prod_{j = 2}^3 \|\jb{\nb}^s z_j\|_{L^{d+2}_{t, x}}\|v_4\|_{Y^0}.
\end{align*}

\noi
Hence,
by  Lemma \ref{PROP:Str1},
the contribution to \eqref{nl1} in this case is at most
$
\les T^{0+} R^2  \|v_1\|_{X^{\frac {d-2}2 }} $
outside a set of probability
\begin{equation*}
\leq
C\exp\bigg(-c \frac {R^2}{ T^{\frac{2}{d+2}-} \|\phi\|_{H^{s}}^{2}}\bigg)
\end{equation*}

\noi
as long as \eqref{nl5} is satisfied.

\smallskip

Putting together Cases (1) - (4)  above,
the conclusion of Part (i)  follows, provided that
\eqref{nl3} is satisfied.

\medskip

\noi
(ii)
First, define  $\wt \G_N$ by
\begin{equation*}
\wt \G_N v(t) =\mp i  \int_0^t S(t-t')\P_{\leq N} \N (v+\eps z)(t') dt',
\end{equation*}

\noi
for $N \geq 1$.
As before, we have
\begin{align}
\|\P_{\leq N} \N (v+\eps z)\|_{L^1_t (\R; H^{\frac{d-2}{2}}_x)}
& \les N^\frac{d-2}{2}
\| v\|_{L^3_t(\R;  L^6_x)}^3
+ \eps^3 N^\frac{d-2}{2}
\| z\|_{L^3_t(\R;  L^6_x)}^3.
\label{NNLS5}
\end{align}

\noi
By a computation similar to \eqref{NNLS3},
we see that  the first term is finite.
Noting that
$(3, \frac{6d}{3d-4})$ is Strichartz admissible
and  $6 \geq \frac{6d}{3d-4}$,
it follows from  Lemma \ref{PROP:Str2}
that the second term
on the right-hand side of \eqref{NNLS5}
 is finite almost surely.
Hence, we can apply Lemma \ref{LEM:Ys1} (i)  to $\wt \G_N v$
for each finite $N\geq 1$, almost surely.

The rest of the proof for this part follows in a similar manner to the proof of Part (i)
by changing the time interval from $[0, T)$ to $\R$
and replacing $z$ by $\eps z$.
By applying Lemma \ref{PROP:Str2}
instead of
Lemma \ref{PROP:Str1} in the above computation,
we see that
the contribution to \eqref{nl1}, where $[0, T)$ is replaced by  $\R$,
is given by
\[\text{Case (2):} \  R^3,\quad
\text{Case (3):} \ R\prod_{j = 1}^2 \|v_j\|_{X^\frac{d-2}{2}(\R)},
\quad
\text{Case (4):} \ R^2 \|v_1\|_{X^\frac{d-2}{2}(\R)}\]

\noi
outside a set of probability
\begin{equation*}
\leq
C\exp\bigg(-c \frac {R^2}{\eps^2   \|\phi\|_{H^{s}}^{2}}\bigg)
\end{equation*}

\noi
in all cases as long as $s > s_d$.
\end{proof}

\section{Proofs of Theorems \ref{THM:1} and \ref{THM:2}}
\label{SEC:THM12}

In this section, we establish the almost sure local well-posedness (Theorem \ref{THM:1})
and probabilistic small data global theory (Theorem \ref{THM:2}).
First, we present the proof of  Theorems \ref{THM:1}.
Given $C_1$ and $C_2$ as in \eqref{nl1a} and \eqref{nl1b},
let $\eta_1 > 0$ be sufficiently small such that
\begin{equation}
C_1 \eta_1^2 \leq \frac{1}{2} \qquad \text{and}\qquad
2C_2 \eta_1^2 \leq \frac{1}{4}.
\label{T0a}
\end{equation}

\noi
Also, given $R\gg1$,
choose $T = T(R)$
such that
\[ T^\theta = \min \bigg( \frac{\eta_1}{2C_1R^3}, \frac{1}{4C_2R^2}\bigg).\]

\noi
Then, it follows from Proposition \ref{PROP:NL1} that $\G$ is a contraction
on the ball $B_{\eta_1}$ defined by
\[
B_{\eta_1}
:= \{ u\in X^\frac{d-2}{2}([0, T))\cap C([0, T); H^\frac{d-2}{2}):\,
\|u\|_{X^\frac{d-2}{2}([0, T))} \leq \eta_1\}\]

\noi
outside a set of probability
\[\leq C \exp\bigg(-c \frac{R^2}{\|\phi\|_{H^s}^2}\bigg)
\sim C \exp\bigg(-c \frac{1}{T^\g \|\phi\|_{H^s}^2}\bigg)\]

\noi
for some $\g > 0$.
This proves Theorem \ref{THM:1}.

Next, we prove Theorem \ref{THM:2}.
Let $\eta_2>0$ be sufficiently small such that
\begin{equation}
2 C_3 \eta_2^2 \leq 1\qquad \text{and}\qquad
3C_4 \eta_2^2 \leq \frac{1}{2},
\label{small0}
\end{equation}

\noi
where $C_3$ and $C_4$ are as in \eqref{nl1c} and \eqref{nl1d}.
Then, by Proposition \ref{PROP:NL1} with $R = \eta_2$
and $\phi^\o$ replaced by $\eps \phi^\o$,
we have
\begin{align}
\|\wt \G v\|_{X^\frac{d-2}{2}(\R)}
& \leq 2 C_3\eta_2^3 \leq \eta_2,
\label{small1}\\
\|\wt \G v_1 - \wt \G v_2  \|_{X^\frac{d-2}{2}(\R)}
& \leq 3 C_4 \eta_2^2
\|v_1 -v_2 \|_{X^\frac{d-2}{2}(\R)}
\leq \frac{1}{2}
\|v_1 -v_2 \|_{X^\frac{d-2}{2}(\R)}
\label{small2}
\end{align}

\noi
outside a set of probability
$\leq C \exp\big(-c \frac{\eta_2^2}{\eps^2\|\phi\|_{H^s}^2}\big)$.
Noting that $\eta_2$ is an absolute constant,
we conclude that
there exists a set $\Omega_\eps \subset \O$
such that
(i) $\wt \G = \wt \G^\o$ is a contraction on the ball $B_{\eta_2}$ defined by
\[
B_{\eta_2}
:= \{ u\in X^\frac{d-2}{2}(\R)\cap C(\R; H^\frac{d-2}{2}):\,
\|u\|_{X^\frac{d-2}{2}(\R)} \leq \eta_2\}\]

\noi
for $\o \in \O_\eps$,
and
(ii) $P(\O_\eps^c)
\leq C \exp\big(- \frac{c}{\eps^2 \|\phi\|_{H^s}^2}\big)$.
This proves global existence for \eqref{NLS1}
with initial data $\eps \phi^\o$ if $\o \in \O_\eps$.

Fix $\o \in \O_\eps$
and let $v = v(\eps, \o)$ be the global-in-time solution with $v|_{t = 0} = \eps \phi^\o$
constructed above.
In order to prove scattering, we need to show that
there exists $v_+^\o \in H^\frac{d-2}{2}(\R^d)$ such that
\begin{equation}
 S(-t) v(t) = \mp i \int_0^t S(-t') \N(v + \eps z)(t') dt' \to v_+^\o
\label{small2a}
 \end{equation}

\noi
in $H^\frac{d-2}{2}(\R^d)$ as $ t\to \infty$.
With $w(t) = S(-t) v(t)$,
define $I(t_1, t_2)$ and $\wt I(t_1, t_2)$ by
\begin{align*}
I(t_1, t_2)
& :=  S(t_2) \big(w(t_2) - w(t_1) ), \\
\wt I(t_1, t_2)& :=\mp i  \int_{0}^{t_2} S(t_2-t') \chi_{[t_1, \infty)}(t')\N(v + \eps z)(t') dt'.
\end{align*}

\noi
Then, for $0< t_1\leq t_2< \infty$,  we have
\begin{align*}
I(t_1, t_2)
= \mp i
S(t_2) \int_{t_1}^{t_2} S(-t') \N(v + \eps z)(t') dt'
= \wt I(t_1, t_2).
\end{align*}

\noi
Also, note that $\wt I(t_1, t_2) = 0$ if $t_1 > t_2$.
In the following, we view
$\wt I(t_1, t_2)$ as a function of $t_2$
and estimate its $X^\frac{d-2}{2}([0, \infty))$-norm.
We now revisit the computation in the proof of Proposition \ref{PROP:NL1}
for $\wt I(t_1, t_2)$.
In Case (1), we proceed slightly differently.
By Lemma \ref{LEM:Ys1} (i), H\"older's inequality,  and \eqref{Ys3},
we have
\begin{align}
\|\wt I(t_1, \cdot)\|_{X^\frac{d-2}{2}(\R_+)}
& \les
\sup_{\substack{v_4 \in Y^{0}(\R_+)\\\|v_4\|_{Y^{0}} = 1}}
\bigg|\int_0^\infty\int_{ \R^d}
\chi_{[t_1, \infty)}(t)  \jb{\nb}^\frac{d-2}{2} v\cj{ v} v v_4 dx dt \bigg| \notag \\
& \leq
\| \jb{\nb}^\frac{d-2}{2} v\|_{L^\frac{2(d+2)}{d}_{t, x}([t_1, \infty))}
\|v\|_{L^{d+2}_{t, x}([t_1, \infty))}^2.
\label{small3}
\end{align}

\noi
By \eqref{Ys3} in Lemma \ref{LEM:Ys1},  we have
\[ \| \jb{\nb}^\frac{d-2}{2} v\|_{L^\frac{2(d+2)}{d}_{t, x}(\R)} +
\|v\|_{L^{d+2}_{t, x}(\R)}\les
 \|v\|_{X^\frac{d-2}{2}(\R)} \leq \eta_2.\]

\noi
Then, by the monotone convergence theorem,
 \eqref{small3} tends to 0 as $t_1 \to \infty$.

In  Cases (2), (3), and (4),
we had at least one factor of $z$.
We multiply
 the cutoff function $\chi_{[t_1, \infty)}$
only on the $(\eps z)$-factors but not on the $v$-factors.
Note that $\|v\|_{X^\frac{d-2}{2}(\R)} \leq \eta_2$.
As in the proof of Proposition \ref{PROP:NL1},
we estimate at least a small portion of these $z$-factors
in  $\|\jb{\nb}^s \eps z^\o\|_{L^q_{t, x}([t_1, \infty))}$,
$q = 4,  \frac{6(d+2)}{d+4}$, or $d+2$, in each case.
Recall that  we have
$\|\jb{\nb}^s \eps z^\o\|_{L^q_{t, x}(\R)} \leq \eta_2$
for $\o \in \O_\eps$. See Lemma \ref{PROP:Str2}.
Hence, again by the monotone convergence theorem,
we have
  $\|\jb{\nb}^s \eps z^\o\|_{L^q_{t, x}([t_1, \infty))} \to 0$
  as $t_1 \to \infty$
  and thus
the contribution from Cases (2), (3), and (4) tends to 0 as $t_1 \to \infty$.
Therefore, we have
\[ \lim_{t_1\to \infty } \| \wt I(t_1, t_2) \|_{X^\frac{d-2}{2}([0, \infty))}
= 0.\]

In conclusion, we obtain
\begin{align*}
\lim_{t_1\to \infty }&  \sup_{t_2 > t_1} \|  w(t_2) -  w(t_1)  \|_{H^\frac{d-2}{2}}
=  \lim_{t_1\to \infty } \sup_{t_2 > t_1} \|  I(t_1, t_2)  \|_{H^\frac{d-2}{2}}\notag\\
&  = \lim_{t_1\to \infty } \| \wt I(t_1, t_2) \|_{L^\infty_{t_2}([0, \infty);  H^\frac{d-2}{2})} 
 \les \lim_{t_1\to \infty } \| \wt I(t_1, t_2) \|_{X^\frac{d-2}{2}([0, \infty))}
= 0.
\end{align*}

\noi
This proves \eqref{small2a} and scattering of $u^\o(t) = \eps S(t) \phi^\o + v^\o(t)$, which completes the proof of Theorem \ref{THM:2}.

\section{Local well-posedness of  NLS with a deterministic perturbation}
\label{SEC:6}

 In this and the next sections, we  consider the following Cauchy problem
of the defocusing NLS with a perturbation:
\begin{equation}
\begin{cases}
	 i \dt v + \Dl v =  |v + f|^2(v+f)\\
v|_{t = t_0} = v_0
 \end{cases}
\label{ZNLS1}
\end{equation}

\noi
where $f$ is a given {\it deterministic} function.
Assuming some suitable conditions on $f$,
we prove local well-posedness of \eqref{ZNLS1} in this section (Proposition \ref{PROP:LWP2})
and long time existence
under further assumptions
in Section \ref{SEC:7} (Proposition \ref{PROP:perturb2}).
Then, we
show,  in Section \ref{SEC:8},  that
the conditions imposed on $f$ for long time existence are satisfied with a large probability
by setting  $f(t) = z(t)  = S(t) \phi^\o$.
This yields Theorem \ref{THM:3}.

Our main goal is to prove long time existence of solutions to the perturbed NLS \eqref{ZNLS1}
by iteratively applying a perturbation lemma (Lemma \ref{LEM:perturb}).
For this purpose, we first prove a ``variant'' local well-posedness of \eqref{ZNLS1}.
As in the usual critical regularity theory, we first introduce
an  auxiliary scaling-invariant norm which is  weaker than the $X^\frac{d-2}{2}$-norm.
Given an interval $I \subset \R$, we introduce the $Z$-norm by
\begin{equation}
\|u\|_{Z(I)} := \bigg( \sum_{\substack{N\geq 1 \\ \text{ dyadic}}}
N^{d-2} \|\P_{N} u\|_{L^4_{t, x}(I\times \R^d)}^4 \bigg)^\frac{1}{4}.
\label{Z0}
\end{equation}

\noi
By the Littlewood-Paley theory
and  \eqref{Ys3} in Lemma \ref{LEM:Ys1},
we have
\begin{equation*}
 \|u\|_{Z(I)}
\les \| \jb{\nb}^\frac{d-2}{4} u \|_{L^4_{t, x}(I\times \R^d)}
\les \|u\|_{X^\frac{d-2}{2}(I)}.
\end{equation*}

\noi
Given $\theta \in (0, 1)$,
we define the  $Z_\ta$-norm by
\[ \|u\|_{Z_\ta(I)} := \|u\|_{Z(I)}^\ta \|u\|_{X^\frac{d-2}{2}(I)}^{1-\ta}.\]

\noi
Note that the $Z_\ta$-norm is weaker than the $X^\frac{d-2}{2}$-norm:
\begin{equation}
 \|u\|_{Z_\ta(I)} \leq C_0  \|u\|_{X^\frac{d-2}{2}(I)}.
\label{Z0a}
\end{equation}

\noi
for some $C_0 > 0$ independent of $I$.

First, we present the bilinear Strichartz estimate
adapted to the $Z_\ta$-norm.


\begin{lemma}\label{LEM:Z1}
Let $N_1 \leq N_2$.
Then, we have
\begin{align}
\| \P_{N_1} u_1 \P_{N_2}u_2\|_{L^2_{t, x}(I\times \R^d)}
& \les  \bigg(\frac{N_1}{N_2}\bigg)^{\frac 12(1-\theta)-}
\|\P_{N_1} u_1\|_{Z_\ta(I)}\|\P_{N_2} u_2\|_{Y^0(I)}.
\label{Z1a}
\end{align}

\end{lemma}

\begin{proof}
Given a cube $R$  of side length $N_1$ centered at $\xi_0\in N_1\Z^d$,
let $\P_R = \psi\big(\frac{D-\xi_0}{N_1}\big)$ denote
a  smooth projection onto $R$ on the frequency side.
Here, $\psi$ denotes the smooth projection onto the unit cube defined in \eqref{mod1a}.
By \eqref{Ys5},
we have
\begin{align*}
\| \P_{N_1} u_1 \P_{R}\P_{N_2}u_2\|_{L^2_{t, x}(I)}
& \leq
\|\P_{N_1} u_1 \|_{L^4_{t, x}}
\|\P_{R}\P_{N_2}u_2\|_{L^4_{t, x}}\\
& \les N_1^\frac{d-2}{4} \|\P_{N_1} u_1 \|_{L^4_{t, x}(I)}
\|\P_{R}\P_{N_2}u_2\|_{Y^0(I)}.
\end{align*}

\noi
Then, by almost orthogonality, we have
\begin{align*}
\| \P_{N_1} u_1 & \P_{N_2}u_2\|_{L^2_{t, x}(I)}
\sim
\Big(\sum_{R}
\| \P_{N_1} u_1 \P_R \P_{N_2}u_2\|^2_{L^2_{t, x}(I)}\Big)^\frac{1}{2}\notag\\
& \les
\|\P_{N_1} u_1\|_{Z(I)} \Big(\sum_{R}
\|\P_{R}\P_{N_2} u_2\|^2_{Y^0(I)}\Big)^\frac{1}{2}
\les \|\P_{N_1} u_1\|_{Z(I)}
 \|\P_{N_2} u_2\|_{Y^0(I)}.
 \end{align*}

\noi
Then, \eqref{Z1a} follows from interpolating this with \eqref{Ys4}.
\end{proof}

Next, we state the key nonlinear estimate.
Given $I \subset \R$, we define the $W^s$-norm by
\begin{align}
\|f\|_{W^s(I)} := \max\Big(\|\jb{\nb}^s f \|_{L^4_{t, x}(I)} ,  \|\jb{\nb}^sf\|_{L^{d+2}_{t, x}(I)},
 \|\jb{\nb}^sf\|_{L^\frac{6(d+2)}{d+4}_{t, x}(I)}\Big).
\label{W1}
\end{align}

\noi
As in the proof of Proposition \ref{PROP:NL1},
different  space-time norms of  $f$
appear in the estimate
but they are all controlled by this $W^s$-norm.
The following lemma  is analogous to Proposition \ref{PROP:NL1}
but with one important difference.
All the terms on the right-hand side have
(i) two factors of the $Z_\ta$-norm of $v_j$, which is weaker than the $X^s$-norm,
or (ii)  the $W^s$-norm of $f$, which can be made small by shrinking the interval $I$.

\begin{lemma}\label{LEM:NL2}
Let $d \geq 3$ and $\ta \in (0, 1)$.
Suppose that
$s , \al \in \R$ satisfy
\begin{align}
s \in (s_d, s_\textup{crit}] \qquad
\text{and} \qquad 0 < \al < 1-  \frac{2}{d-1} s,
\label{NL2a}
\end{align}

\noi
where $s_d$ is as in \eqref{Sd1}.
Then,
 given any interval $I = [t_0, t_1]\subset \R$,
we have
\begin{align*}
\bigg\|  \prod_{j = 1}^3(v_j & + f)^* \bigg\|_{N^\frac{d-2}{2}(I)}
 \les
\sum_{j = 1}^3 \|v_j\|_{X^\frac{d-2}{2}(I)}
\prod_{\substack{k = 1\\k\ne j}}^3\|v_k\|_{Z_\ta(I)}\\
& \hphantom{X}
+ \prod_{\substack{j, k =1\\j\ne k}}^3
\|v_j\|_{X^\frac{d-2}{2}(I)}\|v_k\|_{X^\frac{d-2}{2}(I)}
 \Big(  \|f\|_{W^s(I)} +
\|f\|_{Y^s(I)}^{1-\al}
\|f\|_{W^s(I)}^\al
\Big) \\
& \hphantom{X}
 + \sum_{j = 1}^3 \|v_j\|_{X^{\frac{d-2}{2}}(I)}
\Big(
\|f\|_{Y^s(I)}
 \|f\|_{W^s(I)}
+ \|f\|_{Y^s(I)}^{2-2\al}
 \|f\|_{W^s(I)}^{2\al}
+
 \|f\|_{W^s(I)}^2\Big)\\
& \hphantom{X}
 +  \|f \|_{Y^s(I)}\|f\|_{W^s(I)}^2
+ \|f\|_{Y^s(I)}^{3-2\al}
 \|f\|_{W^s(I)}^{2\al}
+
\|f\|_{W^s(I)}^3,
\end{align*}
		
\noi
for all $f \in W^s(I) \cap Y^s(I)$
and $v_j\in X^{\frac{d-2}{2}}(I)$, $j=1,2,3$,
where $(v_j + f)^* = v_j + f$ or $\cj{v}_j + \cj{f}$.

\end{lemma}

We first state and prove the following local well-posedness
result for the perturbed NLS \eqref{ZNLS1}, assuming Lemma \ref{LEM:NL2}.
The proof of Lemma \ref{LEM:NL2} is presented at the end of this section.


\begin{proposition}[Local well-posedness of the perturbed NLS]\label{PROP:LWP2}
Given $d \geq 3$,
let $s \in ( s_d,  s_\textup{crit}]$, where
$s_d$ is defined in \eqref{Sd1}.
Let $\theta \in (\frac 12, 1)$
and $ \al \in \R$ satisfy \eqref{NL2a}.
Suppose that
\[\|v_0\|_{H^{\frac{d-2}{2}}}\leq R \quad \text{ and } \quad \|f\|_{Y^s(I)}\leq M,\]

\noi
for some $R, M\geq 1$.
Then, there exists  small $\eta_0 = \eta_0(R,M)>0$ such that
if
\[ \|S(t-t_0)v_0\|_{Z_\theta(I)}\leq \eta
\qquad \text{and}
\qquad
\|f\|_{W^s(I)} \leq \eta^{\frac{4-\al}{\al}}\]

\noi
for some $\eta \leq \eta_0$ and
some time interval  $I = [t_0, t_1] \subset \R$,
then there exists a unique solution $v \in X^\frac{d-2}{2}(I)\cap C(I; H^\frac{d-2}{2}(\R^d))$
to \eqref{ZNLS1}
with $v(t_0) = v_0$.
Moreover, we have
\begin{align}
\|v - S(t-t_0) v_0\|_{X^\frac{d-2}{2}(I)} \les \eta^{3-2\ta}.
\label{LWP2a}
\end{align}

\end{proposition}

\begin{proof}
For  $\ta \in (\frac 12 , 1)$, we show that the map $\G$ defined by
\begin{equation}
\G v(t) := S(t-t_0) v_0 -i  \int_{t_0}^t S(t -t') \N(v+f)(t') dt'
\label{LWP2b}
\end{equation}

\noi is a contraction on
\[ B_{R, M,  \eta} = \{ v \in X^\frac{d-2}{2}(I)\cap C(I; H^\frac{d-2}{2}):
 \, \|v\|_{X^\frac{d-2}{2}(I)}\leq 2\wt R,
\ \|v\|_{Z_\ta(I)} \leq 2\eta\}\]

\noi
where $\wt R :=\max(R, M)$.
Now, choose
\begin{align}
\eta_0  \ll \wt R^{-\frac{1}{2\ta -1}}.
\label{LWP2c}
\end{align}

\noi
In particular, we have $\eta_0 \ll \wt R^{-1} \leq 1$.
Fix $\eta \leq \eta_0$ in the following.
Noting that $\frac{4-\al}{\al} > 3$,
 Lemma \ref{LEM:NL2}
with Lemma \ref{LEM:Ys1} yields
\begin{align}
  \|\G v\|_{X^\frac{d-2}{2}(I)}
&  \leq
 \|S(t-t_0)v_0\|_{X^\frac{d-2}{2}(I)}
+  \|\Gamma v - S(t-t_0) v_0\|_{X^\frac{d-2}{2}(I)} \notag \\
& \leq \|v_0\|_{H^\frac{d-2}{2}} + C \eta^2 \wt R \leq 2\wt R,
\label{LWP2d}
\end{align}

\noi
and
\begin{align*}
  \| \G v_1 - \G v_2 \|_{X^\frac{d-2}{2}(I)}
  \leq \frac 12
\| v_1 - v_2 \|_{X^\frac{d-2}{2}(I)}
\end{align*}

\noi
for $v, v_1, v_2 \in B_{R, M, \eta}$.
Moreover,   we have
\begin{align*}
 \|\G v\|_{Z_\ta(I)}
 & \leq
\big(\|S(t-t_0) v_0\|_{Z(I)} +
C \eta^2 \wt R\big)^{\ta}
\big(\|S(t-t_0) v_0\|_{X^\frac{d-2}{2}(I)} +
C \eta^2\wt R \big)^{1-\ta}\notag \\
& \leq  \eta + C \eta^{2\ta}\wt R +C\eta^{2-\theta}\wt R^{1-\theta}+ C\eta^2 \wt R
\leq 2 \eta
\end{align*}

\noi
for $v  \in B_{R, M, \eta}$.
Hence,
$\G$ is a contraction on $B_{R, M, \eta}$.	
The estimate \eqref{LWP2a}
follows from \eqref{LWP2c} and  \eqref{LWP2d}.
\end{proof}

We conclude this section by  presenting the proof of Lemma \ref{LEM:NL2}.
Some cases follow directly from  the proof of Proposition \ref{PROP:NL1}.
However, due to the use of the $Z_\ta$-norm,
we need to make modifications
in several cases.

\begin{proof}[Proof of Lemma \ref{LEM:NL2}]

As in the proof of Proposition \ref{PROP:NL1},
we need to estimate the right-hand side of \eqref{nl1}
by performing a case-by-case analysis of expressions of the form:
\begin{align}
\bigg| \iint_{I\times \R^d}
\jb{\nb}^\frac{d-2}{2} ( w_1 w_2 w_3 )v_4 dx dt\bigg|
\label{Znl1}
\end{align}

\noi
where $\|v_4\|_{Y^{0}(I)} \leq 1$
and $w_j=  v$ or $f$,  $j = 1, 2, 3$.
Before proceeding further, let us simplify some of the notations.
In the following,
as before, we drop the complex conjugate sign
and denote
$X^s(I)$ and $Y^s(I)$ by $X^s$ and $Y^s$.
Lastly,
we dyadically decompose
 $w_j$, $j = 1, 2, 3$,
and $v_4$ such that their spatial frequency supports are $\{ |\xi_j|\sim N_j\}$
for some dyadic $N_j \geq 1$
but still denote them as $w_j = v_j$ or $f_j$, $j = 1, 2, 3$, and $v_4$
if there is no confusion.

\medskip
\noi
{\bf  Case (1):} $v v v$ case.

Without loss of generality, assume that $N_1\geq N_2, N_3$.

\smallskip

\noi
{\bf $\bullet$  Subcase (1.a):} $N_1 \sim N_4$.

By Lemma \ref{LEM:Z1}, we have
\begin{align*}
\bigg|\iint_{ I\times \R^d} & \jb{\nb}^\frac{d-2}{2} v_1 v_2  v_3 v_4 dx dt \bigg|
 \les
\sum_{N_1\sim  N_4\ges N_2, N_3}
N_1^\frac{d-2}{2}
\|\P_{N_1} v_1\P_{N_3} v_3\|_{L^2_{t, x}}
\|\P_{N_2}v_2\P_{N_4} v_4\|_{L^2_{t, x}}\\
& \les \sum_{N_1, \dots, N_4}
\bigg(\frac{N_2N_3}{N_1N_4}\bigg)^{\frac 12(1-\ta)-}
\|\P_{N_1} v_1\|_{Y^\frac{d-2}{2}}\|\P_{N_2}v_2\|_{Z_\ta}
\|\P_{N_3}v_3\|_{Z_\ta}
\|\P_{N_4}v_4\|_{Y^0}
\intertext{By first summing over $N_2, N_3 \leq N_1$
and then applying  Cauchy-Schwarz inequality
in summing over $N_1 \sim N_4$, we have}
& \les
\|v_1\|_{X^\frac{d-2}{2}(I)}\|v_2\|_{Z_\ta(I)} \|v_3\|_{Z_\ta(I)}.
\end{align*}

\noi
{\bf $\bullet$  Subcase (1.b):} $N_1 \sim N_2 \gg N_4$.

By Lemma \ref{LEM:Z1}  and \eqref{Ys3} in  Lemma \ref{LEM:Ys1}, we have
\begin{align*}
\bigg|  \iint_{ I\times \R^d}  & \jb{\nb}^\frac{d-2}{2} v_1 v_2  v_3 v_4 dx dt \bigg|
 \les
\sum_{N_1 \sim N_2 \geq N_3,  N_4}
N_1^\frac{d-2}{2}
\|\P_{N_1} v_1\P_{N_3} v_3\|_{L^2_{t, x}}
\|\P_{N_2}v_2\P_{N_4} v_4\|_{L^2_{t, x}}\\
& \les \sum_{N_1 \sim N_2 \geq N_3,  N_4}
\bigg(\frac{N_3}{N_1}\bigg)^{\frac 12(1-\ta)-}
\|\P_{N_1} v_1\|_{Y^\frac{d-2}{2}}
\|\P_{N_3}v_3\|_{Z_\ta}
\|\P_{N_2}v_2\|_{L^4_{t, x}}
\|\P_{N_4}v_4\|_{L^4_{t, x}}\\
& \les
\sum_{N_1  \sim N_2 \geq N_3,  N_4}
\bigg(\frac{N_3}{N_1}\bigg)^{\frac 12(1-\ta)-}
\bigg(\frac{N_4}{N_2}\bigg)^\frac{d-2}{4}
\|\P_{N_1} v_1\|_{Y^\frac{d-2}{2}}
\|\P_{N_3}v_3\|_{Z_\ta}\\
& \hphantom{XXXXXXXXXXXXXXXX}
\times
N_2^\frac{d-2}{4} \|\P_{N_2}v_2\|_{L^4_{t, x}}
\|\P_{N_4}v_4\|_{Y^0}
\intertext{Summing over $N_3$
and taking a supremum in $N_2$,}
& \les
\|v_2\|_{Z}\|v_3\|_{Z_\ta}
\sum_{N_1 \gg  N_4}
\bigg(\frac{N_4}{N_1}\bigg)^\frac{d-2}{4}
\|\P_{N_1} v_1\|_{Y^\frac{d-2}{2}}
\|\P_{N_4}v_4\|_{Y^0}
\intertext{By Lemma \ref{LEM:Schur}, we have}
& \les
\|v_1\|_{Y^\frac{d-2}{2}}
\|v_2\|_{Z_\ta} \|v_3\|_{Z_\ta} \|v_4\|_{Y^0}
\les
 \|v_1\|_{X^\frac{d-2}{2}(I)}
\|v_2\|_{Z_\ta(I)}\|v_3\|_{Z_\ta(I)}.
\end{align*}


In the following,
 the desired estimates follow from the corresponding cases in the proof of Proposition \ref{PROP:NL1}.
Hence,  we just state the results.

\smallskip

\noi
{\bf Case (2):} $fff$ case. \quad

Without loss of generality, assume $N_3 \geq N_2 \geq N_1$.

\smallskip

\noi
{\bf $\bullet$  Subcase (2.a):} $N_2 \sim N_3$.

The contribution to \eqref{Znl1} in this case is at most
\[ \les
\|f\|_{L^{d+2}_{t, x}}^2 \|\jb{\nb}^{\frac{d-2}{4}+}f\|_{L^{4}_{t, x}}^2
\leq
 \|f\|^3_{W^s(I)}\]

\noi
as long as $s > \frac{d-2}{4}$.

\noi
{\bf $\bullet$  Subcase (2.b):} $N_3 \sim N_4 \gg N_1, N_2$.

\smallskip

\noi
$\circ$ \underline{Subsubcase (2.b.i):} $N_1, N_2 \ll N_3^\frac{1}{d-1}$.

The contribution to \eqref{Znl1} in this case is at most
\[ \les
\|f\|_{Y^s}^{3-2\al} \|\jb{\nb}^s f\|_{L^4_{t, x} }^{2\al}
\leq
\|f\|_{Y^s(I)}^{3-2\al}
 \|f\|_{W^s(I)}^{2\al}\]

\noi
as long as
\eqref{nl3} is satisfied
and $\al < 1-  \frac{2}{d-1} s$.

\smallskip

\noi
$\circ$ \underline{Subsubcase (2.b.ii):} $N_2\ges N_3^\frac{1}{d-1} \gg N_1$.

The contribution to \eqref{Znl1} in this case is at most
\[ \les \|f \|_{Y^s(I)} \|\jb{\nb}^s f\|_{L^4_{t, x}}^2
\leq \|f \|_{Y^s(I)}\|f\|_{W^s(I)}^2\]

\noi
as long as
\eqref{nl3} is satisfied.

\smallskip

\noi
$\circ$ \underline{Subsubcase (2.b.iii):} $N_1, N_2\ges N_3^\frac{1}{d-1} $.

The contribution to \eqref{Znl1} in this case is at most
\[ \les
\| \jb{\nb}^s f\|_{L^\frac{6(d+2)}{d+4}_{t, x}}^3
\leq  \|f\|_{W^s(I)}^3\]

\noi
as long as
\eqref{nl3} is satisfied.

\medskip

\noi
{\bf Case (3):} $v v f$ case.

Without loss of generality,   assume $N_1 \geq N_2$.

\medskip

\noi
{\bf $\bullet$  Subcase (3.a):} $N_1 \ges N_3$.

The contribution to \eqref{Znl1} in this case is at most
\[ \les
\|v\|_{X^\frac{d-2}{2}}^2
 \| \jb{\nb}^s f\|_{L^{d+2}_{t, x}}
\leq  \|v\|_{X^\frac{d-2}{2}(I)}^2
\|f\|_{W^s(I)}\]

\noi
as long as $s > 0$.

\medskip

\noi
{\bf $\bullet$  Subcase (3.b):} $N_3\sim N_4 \gg N_1 \geq N_2$.

\smallskip

\noi
$\circ$ \underline{Subsubcase (3.b.i):} $N_1\ges  N_3^\frac 1{d-1}$.

The contribution to \eqref{Znl1} in this case is at most
\[ \les
\|v\|_{X^\frac{d-2}{2}}^2
 \| \jb{\nb}^s f\|_{L^{d+2}_{t, x}}
\leq  \|v\|_{X^\frac{d-2}{2}(I)}^2
\|f\|_{W^s(I)}\]

\noi
as long as
\eqref{nl4} is satisfied.

\smallskip

\noi
$\circ$ \underline{Subsubcase (3.b.ii):} $N_2 \leq N_1\ll  N_3^\frac 1{d-1}$.

The contribution to \eqref{Znl1} in this case is at most
\[ \les
\|v\|_{X^\frac{d-2}{2}}^2
 \| f\|_{Y^s}^{1-\al} \| \jb{\nb}^s f\|_{L^{d+2}_{t, x}}^\al
\leq  \|v\|_{X^\frac{d-2}{2}(I)}^2
\| f\|_{Y^s(I)}^{1-\al} \|f\|_{W^s(I)}^\al
\]

\noi
as long as
\eqref{nl4} is satisfied.

\medskip

\noi
{\bf Case (4):} $v f f $ case.

Without loss of generality, assume $N_3 \geq N_2$.

\medskip

\noi
{\bf $\bullet$  Subcase (4.a):} $N_1 \ges N_3$.

The contribution to \eqref{Znl1} in this case is at most
\[ \les
\|v\|_{X^{\frac{d-2}{2}}}
\|f\|_{L^{d+2}_{t, x}}^2
\leq
\|v\|_{X^{\frac{d-2}{2}}(I)}
 \|f\|_{W^s(I)}^2\]

\noi
as long as $s> 0$.

\medskip

\noi
{\bf $\bullet$  Subcase (4.b):} $N_3 \gg N_1$.

First, suppose that $N_2 \sim N_3$.
Then,
the contribution to \eqref{Znl1} in this case is at most
\[ \les
\|v\|_{X^{\frac{d-2}{2}}}
\|\jb{\nb}^s f\|_{L^4_{t, x}}^2
\leq
\|v\|_{X^{\frac{d-2}{2}}(I)}
 \|f\|_{W^s(I)}^2\]

\noi
as long as $ s> \frac {d-2}4$.

Hence, it remains to consider the case $N_3 \sim N_4 \gg N_1, N_2$.

\smallskip

\noi
$\circ$ \underline{Subsubcase (4.b.i):} $N_1, N_2\ll N_3^\frac 1{d-1}$.

The contribution to \eqref{Znl1} in this case is at most
\[ \les
\|v\|_{X^{\frac{d-2}{2}}}
\|f\|_{Y^s}^{2-2\al} \|\jb{\nb}^s f\|_{L^4_{t, x} }^{2\al}
\leq
\|v\|_{X^{\frac{d-2}{2}}(I)}
\|f\|_{Y^s(I)}^{2-2\al}
 \|f\|_{W^s(I)}^{2\al}\]

\noi
as long as \eqref{nl5} is satisfied and $\al < 1 - \frac{2}{d-1}s$.

\smallskip

\noi
$\circ$ \underline{Subsubcase (4.b.ii):} $N_1\ll N_3^{\frac 1{d-1}} \les N_2$.

The contribution to \eqref{Znl1} in this case is at most
\[ \les
\|v\|_{X^{\frac{d-2}{2}}}
\|\jb{\nb}^s f\|_{L^4_{t, x}}^2
\leq
\|v\|_{X^{\frac{d-2}{2}}(I)}
 \|f\|_{W^s(I)}^2\]

\noi
as long as
 \eqref{nl5} is satisfied.

\smallskip

\noi
$\circ$ \underline{Subsubcase (4.b.iii):} $N_2 \ll N_3^{\frac 1{d-1}} \les N_1$.

The contribution to \eqref{Znl1} in this case is at most
\[ \les
\|v\|_{X^{\frac{d-2}{2}}}
\|f\|_{Y^s}
\|\jb{\nb}^s f\|_{L^{d+2}_{t, x}}
\leq
\|v\|_{X^{\frac{d-2}{2}}(I)}
\|f\|_{Y^s(I)}
 \|f\|_{W^s(I)}\]

\noi
as long as \eqref{nl5} is satisfied.

\smallskip

\noi
$\circ$ \underline{Subsubcase (4.b.iv):} $N_1, N_2\ges N_3^{\frac 1{d-1}}$

The contribution to \eqref{Znl1} in this case is at most
\[ \les
\|v\|_{X^{\frac{d-2}{2}}}
\|\jb{\nb}^s f\|_{L^{d+2}_{t, x}}^2
\leq
\|v\|_{X^{\frac{d-2}{2}}(I)}
 \|f\|_{W^s(I)}^2\]

\noi
as long as \eqref{nl5} is satisfied.
\end{proof}

\section{Long time existence of solutions to the perturbed NLS}
\label{SEC:7}

The main goal of this section is to establish long time existence of solutions
to the perturbed NLS \eqref{ZNLS1} under some assumptions.
See Proposition \ref{PROP:perturb2}.
We achieve this goal by iteratively applying the perturbation lemma
(Lemma \ref{LEM:perturb}) for the energy-critical NLS.

We first state the perturbation lemma for the energy-critical cubic NLS
involving the $X^\frac{d-2}{2}$- and  the $Z$-norms.
See
\cite{CKSTT,  TV, TVZ, KVZurich} for perturbation and stability results
on usual Strichartz and Lebesgue spaces.
In the context of the cubic NLS on $\R\times \T^3$,
Ionescu-Pausader \cite{IP} proved a perturbation lemma
involving the critical $X^{s_\textup{crit}}$-norm.
Our proof essentially follows their argument
and is included for the sake of completeness.

\begin{lemma}[Perturbation lemma]\label{LEM:perturb}
Let $d \geq 3$ and $I$ be
a compact interval
with $|I|\leq 1$.
Suppose that $v \in C(I; H^\frac{d-2}{2}(\R^d))$ satisfies the following perturbed NLS:
\begin{align}\label{PNLS1}
i \dt v + \Dl v = |v|^2 v + e,
\end{align}

\noi
satisfying
\begin{align}
\| v\|_{Z(I)} + \|v\|_{L^\infty(I; H^\frac{d-2}{2}(\R^d))} \leq R
\label{PP1}
\end{align}

\noi
for some $R \geq 1$.
Then, there exists $\eps_0 = \eps_0(R) > 0$
such that
if
we have
\begin{align}
\|u_0 - v(t_0) \|_{H^\frac{d-2}{2}(\R^d)}
+
\|e\|_{N^\frac{d-2}{2}(I)} \leq \eps
\label{PP2}
\end{align}

\noi
for
some  $u_0 \in H^\frac{d-2}{2}(\R^d)$,
some $t_0 \in I$, and some $\eps < \eps_0$, then
there exists a solution
$u \in X^\frac{d-2}{2}(I)\cap
C(I; H^\frac{d-2}{2}(\R^d))$
to  the defocusing cubic NLS  \eqref{NLS1}
with $u(t_0) = u_0$
such that
\begin{align}
\|u\|_{X^\frac{d-2}{2}(I)}+\|v\|_{X^\frac{d-2}{2}(I)} & \leq C(R), \\
\|u - v\|_{X^\frac{d-2}{2}(I)} & \leq C(R)\eps,
\label{PP3}
\end{align}

\noi
where $C(R)$ is a non-decreasing function of $R$.
\end{lemma}

\begin{proof}
Without loss of generality,  we assume $t_0=\min I$.
Given small $\eps_1 = \eps_1(R)>0$ (to be chosen later),
we divide the interval $I$ into   subintervals $I_j = [t_j, t_{j+1}]$ such that  $I = \bigcup_{j = 0}^{L} I_j$.
By choosing
$L\sim \big(\frac{R}{\eps_1}\big)^4$, we can guarantee that
\begin{equation}
\|v\|_{Z(I_j)} \leq \eps_1
\label{PP3y}
\end{equation}

\noi
for $j = 0, \dots, L$.
By assumption, we also have
\begin{equation}
\|e\|_{N^\frac{d-2}{2}(I_j)} \leq \eps < \eps_0
\label{PP3x}
\end{equation}

\noi
\noi
for $j = 0, \dots, L$.

\smallskip

\noi
$\bullet$  {\bf Step 1:}
Let $\ta \in (\frac 12, 1)$.
We first claim that
there exist $\eta_0= \eta_0(R) >0$
and $\eps_0= \eps_0(R)>0$
such that if
\begin{equation}
 \|S(t - t_*) v(t_*)\|_{Z_\theta(J)} \leq \eta_0
\qquad\text{and}
\qquad
\|e\|_{N^\frac{d-2}{2}(J)} \leq \eps_0
\label{PP3b}
\end{equation}

\noi
for some $t_*$ in a subinterval  $J \subset I$, then
there exists a unique solution $v$ to \eqref{PNLS1} on $J$,
satisfying
\begin{equation}
 \|v - S(t - t_*) v(t_*) \|_{X^\frac{d-2}{2}(J)}
\leq C \|S(t - t_*) v(t_*)\|_{Z_\theta(J)}^{3-2\ta} + 2 \|e\|_{N^\frac{d-2}{2}(J)}.
\label{PP3a}
\end{equation}

We choose $\eta_0 = \eta_0(R)$ and $\eps_0 = \eps_0(R)$ such that
\begin{align}\label{PP4a}
\eta_0  \ll  R^{-\frac{1}{2\ta -1}}
\qquad \text{and}
\qquad
\eps_0 \ll  R^{-\frac{2}{2\ta -1}}.
\end{align}

\noi
In the following,
we set
\[\eta :=  \|S(t - t_*) v(t_*)\|_{Z_\theta(J)} \leq \eta_0
\qquad
\text{and}
\quad
\eps: = \|e\|_{N^\frac{d-2}{2}(J)} \leq \eps_0.\]

\noi
Then, proceeding as in the proof of Proposition \ref{PROP:LWP2},
we show that the map $\G$ defined by
\begin{equation}
\G v(t) := S(t-t_*) v (t_*) -i  \int_{t_*}^t S(t -t')\N(v)(t') dt' -i  \int_{t_*}^t S(t -t')e(t') dt'.
\label{PP4}
\end{equation}

\noi is a contraction on
\[ B_{R,  \eta, \eps} = \big\{ v \in X^\frac{d-2}{2}(J)\cap C(J; H^\frac{d-2}{2}):
 \, \|v\|_{X^\frac{d-2}{2}(J)}\leq 2 R,
\ \|v\|_{Z_\ta(J)} \leq 2 (\eta+ \eps^{\frac{\ta}{2}+\frac 14})\big\}. \]

Indeed, by
Lemma \ref{LEM:NL2} (with $f = 0$),
we have
\begin{align}
  \|\G v\|_{X^\frac{d-2}{2}(J)}
& \leq \|v(t_\ast)\|_{H^\frac{d-2}{2}} + C
(\eta+\eps^{\frac{\ta}{2}+\frac 14})^2  R
+ \eps \notag \\
&
\leq \|v(t_\ast)\|_{H^\frac{d-2}{2}} + C
\eta^2 R+ 2\eps
\leq 2  R,
\label{PP5}
\end{align}

\noi
and
\begin{align*}
  \| \G v_1 - \G v_2 \|_{X^\frac{d-2}{2}(J)}
  & \leq C(\eta + \eps^{\frac{\ta}{2}+\frac 14})R
\| v_1 - v_2 \|_{X^\frac{d-2}{2}(J)}
  \leq \frac 12
\| v_1 - v_2 \|_{X^\frac{d-2}{2}(J)}, \notag
\end{align*}

\noi
for $v, v_1, v_2 \in B_{R,  \eta, \eps}$.
Moreover, we have
\begin{align*}
 \|\G v\|_{Z_\ta(J)}
 & \leq
\big(\|S(t-t_\ast) v(t_\ast)\|_{Z(J)} +
C \eta^2  R+ \eps\big)^{\ta}\\
& \hphantom{XXXXXXX}
\times
\big(\|S(t-t_\ast) v(t_\ast)\|_{X^\frac{d-2}{2}(J)} +
C \eta^2 R + \eps \big)^{1-\ta}\notag \\
& \leq  \eta+  C \eta^{2\ta} R
+ C \eta^{2-\ta} R^{1-\ta}
+ C\eta^\theta\eps^{1-\theta}
+ C\eps^\ta R^{1-\ta}
\leq 2   (\eta+ \eps^{\frac{\ta}{2}+\frac 14}),
\end{align*}

\noi
for $v \in B_{R,  \eta, \eps}$.
Hence,
$\G$ is a contraction on $B_{R, \eta_1}$.	
The estimate \eqref{PP3a}
follows from \eqref{PP4a}
and \eqref{PP5}.

\smallskip

\noi
$\bullet$  {\bf Step 2:}
Next, we claim that, given $\eps_2 > 0$,
we can choose $\eps_j = \eps_j(R, \eps_2)$, $j = 0, 1$,   in \eqref{PP3x} and \eqref{PP3y}
sufficiently small such that
we have
\begin{align}
\|S(t - t_j) v(t_j) \|_{Z_\ta(I_j)}
& \leq \eps_2
\qquad \text{and}
\qquad
\|v \|_{Z_\ta(I_j)}
 \leq \eps_2.
\label{PP7}
\end{align}

Without loss of generality, assume
$\eps_2 \leq \frac{\eta_0}{2}$,
where $\eta_0 = \eta_0(R)$ is as in Step 1.
Let $h(\tau) = \|S(t - t_j) v(t_j) \|_{Z_\ta([t_j, t_j + \tau])} $.
Note that $h$ is continuous and $h(0) = 0$.
Thus, we have $h(\tau) \leq 2\eps_2 \leq \eta_0$
for small $\tau > 0$.
Then, from the Duhamel formula \eqref{PP4}
with \eqref{PP3y}, \eqref{PP3x}, and \eqref{PP3a}, we have
\begin{align}
h(\tau)
& \leq
 \|S(t - t_j) v(t_j) \|_{X^\frac{d-2}{2}([t_j, t_j + \tau])}^{1-\ta}
 \|S(t - t_j) v(t_j) \|_{Z([t_j, t_j + \tau])}^{\ta} \notag \\
& \leq
C R^{1-\ta}\big(\eps_1 + \eps_2^{3-2\ta} + \|e\|_{N^\frac{d-2}{2}(I_j)}\big)^\ta \notag \\
& \leq
C R^{1-\ta} \eps_2^{\ta(3-2\ta)}
+ CR^{1-\ta}\big(\eps_1 +  \|e\|_{N^\frac{d-2}{2}(I_j)}\big)^\ta
\label{PP8}
\end{align}

\noi
From \eqref{PP4a} with $\eps_2 \leq  \frac{\eta_0}{2}$, we have
\begin{equation}
CR^{1-\theta}\eps_2^{\theta(3-2\theta)}\leq C\big
(R\eta_0^{(2\theta-1)}\big)^{1-\theta}\eps_2\ll\eps_2.
\label{PP8a}
\end{equation}

\noi
Hence, it follows from
\eqref{PP8} and \eqref{PP8a} that
\begin{align}
h(\tau)
 & \leq  \tfrac 12 \eps_2 + C R^{1-\theta}(\eps_1 + \eps_0 )^{\theta}
 \leq \eps_2
\label{PP9}
\end{align}

\noi
by choosing $\eps_j = \eps_j(R,\eps_2)>0$ sufficiently small, $j = 0, 1$.
Then, by the continuity argument, we see that
\eqref{PP9} holds for all $\tau \leq t_{j+1} - t_j$.
From Step 1 and \eqref{PP3y}, we have
\begin{equation} \|v\|_{Z_\ta(I_j)}  = \|v\|_{Z(I_j)}^\ta \|v\|_{X^\frac{d-2}{2}(I_j)}^{1-\ta}
\leq C \eps_1^\ta R^{1-\ta}.
\label{PP9a}
\end{equation}

\noi
Therefore,  \eqref{PP7} follows from \eqref{PP9} and \eqref{PP9a},
by choosing $\eps_1 = \eps_1(R, \eps_2)$ smaller if necessary.

\smallskip

\noi
$\bullet$  {\bf Step 3:}
Given $\eps_2 = \eps_2(R) >0$ (to be chosen later),
it follows from Step 2 that \eqref{PP7} holds
as long as
$\eps_ j= \eps_j(R)>0$ $j = 0, 1$, are sufficiently small.
From Step 1 with \eqref{PP3x}, \eqref{PP7},  and \eqref{PP5}, we have
\begin{align}
\|v\|_{X^\frac{d-2}{2}(I_j)} \leq 2R
\label{PP10}
\end{align}

\noi
as long as $\eps_j  = \eps_j(R)>0$, $j = 0, 1, 2$, are sufficiently small.

Let $u$ be a solution to the defocusing cubic NLS \eqref{NLS1}
with initial data $u(t_j)$ given at $t = t_j$
such that
\begin{align}
\| u(t_j) - v(t_j) \|_{H^\frac{d-2}{2}} \leq \eps < \eps_0.
\label{PP10a}
\end{align}

\noi
Let $J_j = [t_j, t_j + \tau] \subset I_j$ be the maximal time interval such that
\begin{equation}
\| u - v\|_{Z_\ta(J_j)} \leq 6C_0\eps,
\label{PP11}
\end{equation}

\noi
where $C_0$ is as in \eqref{Z0a}.
Such an interval exists and is non-empty,
since $\tau \mapsto \| u - v\|_{Z_\ta(t_j, t_j +  \tau )}$
is finite and continuous  (see Lemma \ref{LEM:U6}), at least  on the interval of local existence of $u$, 
and vanishes for $\tau = 0$.

Let $w : = u-v$.
By Lemma \ref{LEM:NL2} (with $f = 0$)
with \eqref{PP3x}, \eqref{PP7}, \eqref{PP10}, \eqref{PP10a}, and \eqref{PP11}, we have
\begin{align}
\|w\|_{X^\frac{d-2}{2}(J_j)}
& \leq \| u(t_j) - v(t_j)\|_{H^\frac{d-2}{2}}
+ C_1\Big\{
\|v\|_{X^\frac{d-2}{2}(J_j)}\|v\|_{Z_\ta(J_j)}\|w\|_{X^\frac{d-2}{2}(J_j)} \notag \\
& \hphantom{XXXXX}
 + \|v\|_{X^\frac{d-2}{2}(J_j)}\|w\|_{X^\frac{d-2}{2}(J_j)}\|w\|_{Z_\ta(J_j)} \notag \\
& \hphantom{XXXXX}
 + \|w\|_{X^\frac{d-2}{2}(J_j)}\|w\|_{Z_\ta(J_j)}^2\Big\}
+ \|e\|_{N^\frac{d-2}{2}(J_j)}\notag\\
& \leq 2 \eps + C_2 (\eps_0 + \eps_2) R \|w\|_{X^\frac{d-2}{2}(J_j)}. \notag
\end{align}

\noi
Taking  $\eps_j = \eps_j(R)>0$ sufficiently small, $j = 0, 2$,  such that
 $(\eps_0+\eps_2) R\ll 1$,
we obtain
\begin{align}
\|w\|_{X^\frac{d-2}{2}(J_j)}
 \leq 4 \eps.
\label{PP12}
\end{align}

\noi
Hence, from \eqref{Z0a}, we have
\begin{align}
\|w\|_{Z_\ta(J_j)} \leq C_0 \|w\|_{X^\frac{d-2}{2}(J_j)}
 \leq 4 C_0 \eps. \label{PP13}
\end{align}

From \eqref{PP10} and \eqref{PP12},
we have
$\|u\|_{X^\frac{d-2}{2}(J_j)} \leq 3R<\infty$.
Then,  from \eqref{Ys2},
we have
$\|u\|_{\dot{S}^{\frac{d-2}{2}}(J_j)} < \infty$.
In particular,
this implies that
$u$ can be extended to some larger interval $ J' \supset \cj {J_j}$.
Therefore, in view of \eqref{PP11} and \eqref{PP13},
we can apply the continuity argument and conclude that $J_j = I_j$.

\smallskip

\noi
$\bullet$  {\bf Step 4:}
By \eqref{PP2},
we have  $\|u(t_0)-v(t_0)\|_{H^{\frac{d-2}{2}}}\leq \eps$
for some $\eps < \eps_0$.
Then, by Step 3, we have $\|w\|_{X^{\frac{d-2}{2}}(I_0)}\leq 4\eps$
on $I_0 = [t_0, t_1]$. In particular, this yields
\[\|u(t_1)-v(t_1)\|_{H^{\frac{d-2}{2}}}\leq 4C\eps.\]

\noi
Then, we can apply  Step 3 on the interval $I_1$
by choosing $\eps_0$ (and hence $\eps)$ even smaller.
We argue recursively for each  interval $I_j$, $j=2,\dots, L$.
Note that,   at each step,
we make $\eps_0$ smaller by a factor of $(4C)^{-1}$.
Since
 $L\sim\big(\frac{R}{\eps_1}\big)^4$ and $\eps_1 = \eps_1 (R)$,
there are  a finite number of iterative steps depending only on $R$.
This allows us to choose new  $\eps_0 = \eps_0(R)> 0$
such that,   by  Lemma \ref{LEM:U2}, we have
\begin{align*}
\|u\|_{X^{\frac{d-2}{2}}(I)}+\|v\|_{X^{\frac{d-2}{2}}(I)}
&\lesssim L R \lesssim C(R), \\
\|u - v\|_{X^\frac{d-2}{2}(I)} &  \les L\eps \lesssim  C(R)\eps,
\end{align*}

\noi
 This completes the proof of Lemma \ref{LEM:perturb}.
\end{proof}

In the remaining part of this section,
we consider  long time existence of
solutions to the perturbed NLS \eqref{ZNLS1}
under several assumptions.
Given $T>0$,
we assume that
there exist $\beta, C, M > 0$  such that
\begin{equation}
 \|f\|_{W^s(I)} \leq C |I|^{\be}
 \qquad \text{and}\qquad
 \|f\|_{Y^s([0, T])} \leq M
\label{P0}
 \end{equation}

\noi
for any interval $I \subset [0, T]$.
Then,  Proposition \ref{PROP:LWP2}
guarantees
existence of a solution to the perturbed NLS \eqref{ZNLS1},
at least for a short time.

\begin{proposition}\label{PROP:perturb2}
Let $d \geq 3$.
Let $s \in (s_d,  s_\textup{crit}]$, where
$s_d$ is defined in \eqref{Sd1}.
Given $T>0$, assume the following conditions \textup{(i)} - \textup{(iii)}:
\begin{itemize}

\item[\textup{(i)}]
Hypothesis \textup{(B)} holds
if $d \ne 4$,

\item[\textup{(ii)}]
$f \in Y^s([0, T]) \cap W^s([0, T])$
satisfies \eqref{P0},

\item[\textup{(iii)}]
Given
a solution
$v$  to  \eqref{ZNLS1},
 the following a priori bound holds:
\begin{align}
\| v\|_{L^\infty([0, T]; H^\frac{d-2}{2}(\R^d))} \leq R
\label{P0a}
\end{align}

\noi
for some $R>0$.
\end{itemize}

\noi
Then,
there exists
$\tau = \tau(R,M, T, s, \be)>0$
such that,  given any $t_0 \in [0, T)$,
the solution $v$ to \eqref{ZNLS1}
exists on $[t_0, t_0 + \tau]\cap [0, T]$.
In particular, the condition \textup{(iii)} guarantees existence of $v$ on
the entire interval $[0, T]$.

\end{proposition}

\begin{remark}\label{REM:perturb3}\rm
We point out  that the first condition in \eqref{P0} can be weakened as follows.
Let $\tau = \tau(R, M, T, s, \be) > 0$ be as in  Proposition \ref{PROP:perturb2}.
Then,
it follows from the proof of Proposition \ref{PROP:perturb2}
(see \eqref{P4a} and \eqref{P4b} below)
that
if we assume that
\begin{equation*}
 \|f\|_{W^s([t_0, t_0+\tau_*))} \leq C |\tau_*|^{\be}
 \end{equation*}

\noi
for some $\tau_* \leq \tau$ instead of the first condition in \eqref{P0},
then the conclusion of Proposition \ref{PROP:perturb2} still holds
on $[t_0, t_0 + \tau_*]\cap [0, T]$.
Indeed, we use this version of Proposition \ref{PROP:perturb2}
in Section \ref{SEC:8}.

\end{remark}

\begin{proof}

By setting $e = |v+f|^2 (v+f) - |v|^2 v$, \eqref{ZNLS1}
reduces to \eqref{PNLS1}.
In the following, we iteratively apply Lemma \ref{LEM:perturb}
on short intervals and
 show that there exists $\tau = \tau(R, M, T, s, \be) > 0$ such that
\eqref{PNLS1} is well-posed on $[t_0, t_0 + \tau] \cap [0, T] $ for any $t_0 \in [0, T)$.

Let $w$ be the global solution to the defocusing cubic NLS \eqref{NLS1}
with $w(t_0) = v(t_0) = v_0$.
By \eqref{P0a}, we  have $\|w(t_0) \|_{H^\frac{d-2}{2}} \leq R$.
Then,
by Hypothesis (B), we have
\[\| w\|_{L^{d+2}_{t, x}([0, T])} \leq  C(R, T) < \infty.\]

\noi
By the standard argument, this  implies that
$\big\||\nb|^\frac{d-2}{2}  w\big\|_{L^q_tL^r_x([0, T])} \leq C'(R, T) <  \infty$
for all Schr\"odinger admissible pairs $(q, r)$.
In particular, we have
$\|w\|_{Z([0, T])} \leq C''(R, T) <  \infty$
and
\begin{align}
\|w\|_{X^\frac{d-2}{2}([0, T])}
& \leq \|v_0\|_{H^\frac{d-2}{2}}
+ \big\||\nb|^\frac{d-2}{2} w\big\|_{L^\frac{2(d+2)}{d}_{t, x}([0, T])}
\|w\|_{L^{d+2}_{t, x}([0, T])}^2 \notag \\
& \leq C'''(R, T)
< \infty.
\label{P1}
\end{align}

Let $\ta \in (\frac 12, 1)$.
Given small $\eta > 0$ (to be chosen later),
we  divide the interval $ [t_0, T]$
into $J = J(R, T, \eta)$  many subintervals $I_j = [t_j, t_{j+1}]$
such that
\[ \|w\|_{Z_\ta(I_j)} \sim \eta.\]
	
\noi
In the following,
we fix the value of $\ta$ and
suppress dependence of various constants such as $\tau$ and $\eta$ on $\ta$.

Fix $\tau > 0$ (to be chosen later in terms of $R$, $M$, $T$, $s$, and $\be$)
and write $[t_0, t_0+\tau] = \bigcup_{j = 0}^{J'} \big([t_0, t_0+\tau]\cap I_j\big)$
for some $J'  \leq J - 1$, where $[t_0, t_0+\tau]\cap I_j \ne \emptyset$ for $0 \leq j \leq J'$
and $[t_0, t_0+\tau]\cap I_j=\emptyset$ for $j \geq J'$.

Since the nonlinear evolution $w$ is small on each $I_j$,
it follows that the linear evolution $S(t-t_j) w(t_j)$ is also small on each $I_j$.
Indeed, from the Duhamel formula, we have
\[S(t-t_j) w(t_j) = w(t) +i \int_{t_j}^t S(t - t') |w|^2w(t') dt'.\]

\noi
Then,  from Case (1) in the proof of Lemma \ref{LEM:NL2} with \eqref{P1},
we have
\begin{align*}
\|S(t-t_j) w(t_j) \|_{Z_\ta(I_j)}
&\leq \|w\|_{Z_\ta(I_j)} + C \|w\|_{X^\frac{d-2}{2}(I_j)} \|w\|_{Z_\ta(I_j)}^2
\leq \eta + C(R, T)  \eta^2.
\end{align*}
	
\noi
By taking $\eta = \eta(R, T)>0$ sufficiently small, we have
\begin{align}
\|S(t-t_j) w(t_j) \|_{Z_\ta(I_j)}
\leq 2 \eta
\label{P2}
\end{align}

\noi
for all $j = 0, \dots, J-1$.

Now, we estimate $v$ on the first interval $I_0$.
Let $\eta_0 = \eta_0(R, M)$ be as in
Proposition \ref{PROP:LWP2}.
Then,
by Lemma \ref{LEM:Ys1} (i),
\eqref{P0a}, and
Proposition \ref{PROP:LWP2},
we have
\begin{align*}
\|v\|_{X^\frac{d-2}{2} (I_0)}
& \leq \|S(t- t_0) v(t_0)\|_{X^\frac{d-2}{2}(I_0)}
+  \|v - S(t- t_0) v(t_0)\|_{X^\frac{d-2}{2}(I_0)} \notag \\
& \leq R + C\eta^{3-2\ta} \leq 2R,
\end{align*}

\noi
as long as
$2\eta < \eta_0$
and
 $\tau = \tau(\eta, \al, \be) = \tau(R,M, T,  \al, \be)> 0$ is sufficiently small
so that
\begin{align}
\|f\|_{W^s([t_0, t_0+ \tau))} \leq C\tau^\be \leq \eta^\frac{4-\al}{\al},
\label{P4a}
\end{align}

\noi
where $\al = \al(s)$  satisfy  \eqref{NL2a}.

Next, we estimate the error term.
By Lemma \ref{LEM:NL2} with \eqref{P0}, we have
\begin{align}
\|e\|_{N^\frac{d-2}{2}(I_0)}
\leq C(R, M) \tau^{\al\be}.
\label{P4b}
\end{align}

\noi
Given $\eps > 0$, we can choose
 $\tau = \tau(R, M, T,  \eps, \al, \be)>0$ sufficiently small
 so that
\begin{align*}
\|e\|_{N^\frac{d-2}{2}(I_0)}
\leq \eps.
\end{align*}

\noi
In particular, for $\eps < \eps_0$ with $\eps_0= \eps_0(R) > 0$
dictated by Lemma \ref{LEM:perturb},
the condition \eqref{PP2}
is satisfied on $I_0$.

Therefore,  all the conditions of Lemma \ref{LEM:perturb} are satisfied
on the first interval $I_0$,
provided that $\tau = \tau (R, M,  T, \eps, \al, \be)>0$ is chosen sufficiently small.
Hence, we obtain
\begin{align}
\| w - v\|_{X^\frac{d-2}{2}(I_0)} \leq C_0(R) \eps.
\label{P5}
\end{align}
	
\noi
In particular, we have
\begin{align}
\|w(t_1) - v(t_1) \|_{H^\frac{d-2}{2}}
\leq C_1(R) \eps.
\label{P6}
\end{align}

\noi
Then, from  \eqref{P2} and  Lemma \ref{LEM:Ys1} (i) with \eqref{P6},
we have
\begin{align*}
\|S(t - t_1) v(t_1)\|_{Z_\ta(I_1)}
&\leq \|S(t-t_1)w(t_1)\|_{Z_\ta(I_1)}+\|S(t-t_1)(w(t_1)-v(t_1))\|_{Z_\ta(I_1)}\\
&\leq 2 \eta + C'_1(R) \eps
\leq 3\eta
\end{align*}

\noi
by choosing $\eps = \eps(R, \eta) > 0$ sufficiently small.

Proceeding as before,
it follows from
 Proposition \ref{PROP:LWP2}
with \eqref{P0a} and  \eqref{P2} that
\begin{align*}
\|v\|_{X^\frac{d-2}{2} (I_1)}
& \leq R + C\eta^{3-2\ta} \leq 2R,
\end{align*}

\noi
as long as $3\eta \leq \eta_0$
and
 $\tau > 0$ is sufficiently small
so that \eqref{P4a} is satisfied.
Similarly,
it follows from Lemma \ref{LEM:NL2} with \eqref{P0} that
\begin{align}
\|e\|_{N^\frac{d-2}{2}(I_1)}
\leq C(R, M) \tau^{\al \be}\leq \eps
\label{P7}
\end{align}
	
\noi
by choosing
 $\tau = \tau(R, M,  T, \eps, \al, \be)>0$ sufficiently small.
Therefore,  all the conditions of Lemma \ref{LEM:perturb} are satisfied
on the second interval $I_1$,
provided that $\tau = \tau (R, M, T, \eps,\alpha, \beta)$ is chosen sufficiently small
and that
$ (C_1(R) + 1)\eps <\eps_0$.
Hence,
by Lemma \ref{LEM:perturb},
we obtain
\begin{align*}
\| w - v\|_{X^\frac{d-2}{2}(I_1)} \leq C_0(R) (C_1(R) + 1)\eps.
\end{align*}
	
\noi
In particular, we have
\begin{align*}
\|w(t_2) - v(t_2) \|_{H^\frac{d-2}{2}}
\leq C_2(R) \eps.
\end{align*}

By choosing $\eta = \eta(R, M, T) > 0$
and $\tau = \tau(R, M, T, \eps, \alpha,\beta)>0$
sufficiently small,
we can argue inductively
and obtain
\begin{align}
 \|w(t_j) - v(t_j) \|_{H^\frac{d-2}{2}}
\leq C_j(R) \eps.
\end{align}

\noi
for all $0 \leq j \leq J'$,
as long as
(i) $(J'+2)\eta \leq \eta_0$
and (ii)
 $\eps  = \eps(R, \eta, J)$ is sufficiently small
such that
$(C_j(R)+1) \eps < \eps_0$,
$j = 1, \dots, J'$.
Recalling that $J'+1 \leq J = J(R, T, \eta)$,
we see that this can be achieved
by choosing  $\eta = \eta(R, M, T)>0$,
$\eps = \eps(R, M, T) > 0$,
and $\tau = \tau(R,M, T, \al, \be)
=  \tau(R,M, T, s, \be) >0$ sufficiently small.
This guarantees existence of the solution $v$ to \eqref{PNLS1} on $[t_0, t_0+\tau]$.

Under the conditions (i) - (iii),
we can apply the above local argument on time intervals of length
$\tau = \tau(R, M, T, s, \be)>0$, thus extending the solution $v$ to \eqref{ZNLS1}
on the entire interval $[0, T]$.
\end{proof}

\section{Proof of Theorem \ref{THM:3}}	
\label{SEC:8}

In this section, we prove the following ``almost'' almost sure global existence result.

\begin{proposition}\label{PROP:asGWP}
Let $d\geq 3$ and $s \in (s_d,  s_\textup{crit}]$.
Assume Hypothesis  \textup{(A)}.
Furthermore, assume
Hypothesis \textup{(B)} if $d \ne 4$.
Given $\phi \in H^s(\R^d)$, let $\phi^\o$ be its Wiener randomization defined in \eqref{R1},
satisfying \eqref{R2}.
Then, given any $T, \eps > 0$, there exists a set $\wt \O_{T, \eps}\subset \O$
such that
\begin{itemize}
\item[\textup{(i)}]
$P(\wt \O_{T, \eps}^c) < \eps$,

\item[\textup{(ii)}]
For each $\o \in \wt \O_{T, \eps}$, there exists a (unique) solution $u$
to \eqref{NLS1}  on $[0, T]$
with $u|_{t = 0} = \phi^\o$.

\end{itemize}

\end{proposition}
	
\noi
It is easy to see that
``almost'' almost sure global existence implies almost sure global existence.
See \cite{CO}.
For completeness, we first show how
Theorem \ref{THM:3} follows as an immediate consequence of
Proposition \ref{PROP:asGWP}.

Given small $\eps > 0$, let $T_j = 2^j$ and $\eps_j = 2^{-j} \eps$, $j \in \mathbb{N}$.
For each $j$, we apply Proposition \ref{PROP:asGWP} and construct $ \wt \Omega_{T_j, \eps_j}$.
Then, let $\Omega_\eps = \bigcap_{j = 1}^\infty  \wt \Omega_{T_j, \eps_j}$.
Note that (i) $P (\Omega_\eps^c) < \eps$,
and (ii)
for each $\o \in \O_\eps$,
we have a global solution $u$ to \eqref{NLS1} with $u|_{t = 0} = \phi^\o$.
 Now, let $\Sigma = \bigcup_{\eps > 0} \Omega_\eps$.
 Then,
 we have $P (\Sigma^c) = 0$.
 Moreover,
 for each $\o \in \Sigma$,
we have a global solution $u$ to \eqref{NLS1} with $u|_{t = 0} = \phi^\o$.
This proves Theorem \ref{THM:3}.

The rest of this section is devoted to the proof of Proposition \ref{PROP:asGWP}.

\begin{proof}[Proof of Proposition \ref{PROP:asGWP}]
Given $T, \eps > 0$,
set
\begin{equation}	
M = M(T, \eps) \sim \|\phi\|_{H^s}\Big(\log\frac{1}{\eps}\Big)^\frac{1}{2}.
\label{asGWP0}
\end{equation}

\noi
Defining $ \O_1 = \O_1(T, \eps)$
by
\[  \O_1 := \big\{ \o \in \O:\, \|S(t) \phi^\o \|_{Y^s([0, T])} \leq M\big\},\]

\noi
it follows from
 Lemma \ref{LEM:Ys1} (i) and Lemma \ref{LEM:Hs} that
\begin{equation}
 P( \O_1^c) < \frac{\eps}{3}.
\label{asGWP1}
\end{equation}

Given $T, \eps > 0$,
let $R = R(T, \frac \eps 3)$ and $M$ be as in \eqref{HypA} and \eqref{asGWP0},
respectively.
With $\tau = \tau(R, M, T)$ as in Proposition \ref{PROP:perturb2}.
write
\[ [0, T] = \bigcup_{j = 0}^{[\frac T{ \tau_*}]} [j \tau_*, (j + 1) \tau_*] \cap [0, T], \]

\noi
for some $\tau_* \leq \tau$ (to be chosen later).
Now, define $\O_2$ by
\[ \O_2 = \Big\{ \o \in \O:
\| S(t) \phi^\o \|_{W^s([j \tau_*, (j+1)\tau_*])}
\leq C  \tau_*^\frac{1}{2(d+2)}, \, j = 0, \dots, \big[\tfrac T{\tau_*}\big] \Big\},
\]

\noi
\noi
where $C$ is as in \eqref{P0}.
Then,
by Lemma \ref{PROP:Str1}, we have
\begin{align*}
P(\O_2^c)
& \leq  \sum_{j = 0}^{[\frac{T}{\tau_*}] }
P\Big( \| S(t) \phi^\o  \|_{W^s([j\tau_*, (j+1)\tau_*])}\geq C\tau_*^\frac{1}{2(d+2)}\Big)
  \les \frac{T}{\tau_*}  \exp\bigg(-\frac{c}{\tau_*^\frac{1}{d+2} \|\phi\|_{H^s}^2}\bigg)\notag \\
\intertext{By making $\tau_* = \tau_*(\|\phi\|_{H^s})$ smaller, if necessary, we have}
&  \les \frac{T}{\tau_\ast}\cdot \tau_* \exp\bigg(-\frac{c}{2\tau_*^\frac{1}{d+2} \|\phi\|_{H^s}^2}\bigg)
= T \exp\bigg(-\frac{1}{2\tau_*^\frac{1}{d+2} \|\phi\|_{H^s}^2}\bigg).
\end{align*}

\noi
Hence, by choosing $\tau_* = \tau_*(T, \eps, \|\phi\|_{H^s})$ sufficiently small,
we have
\begin{align}
P(\O_2^c) < \frac{\eps}{3}.
\label{asGWP4}
\end{align}

Finally, set $\wt \O_{T, \eps} = \Omega_{T, \frac \eps 3}\cap \O_1 \cap\O_2$,
where $\Omega_{T, \frac \eps 3}$ is as in Hypothesis (A) with $\eps$
replaced by $\frac \eps 3$.
Then, from \eqref{asGWP1} and \eqref{asGWP4}, we have
\[ P(\wt \O_{T, \eps}^c) < \eps.\]

\noi
Moreover, for $\o \in \wt \O_{T, \eps}$,
 we can iteratively apply  Proposition \ref{PROP:perturb2} and Remark \ref{REM:perturb3}
and construct the solution $v = v^\o$ on each $[j \tau_*, (j+1)\tau_*]$, $j = 0, \dots, [\frac T\tau_*]-1$,
and $\big[ [\frac T \tau_*]\tau, T\big]$.
This completes the proof of Proposition \ref{PROP:asGWP}.
\end{proof}

\begin{remark}\label{REM:asGWP} \rm
It is worthwhile to mention that
the proof of Proposition \ref{PROP:asGWP}
strongly depends on the quasi-invariance property of
the distribution of the linear solution $S(t) \phi^\o$.
More precisely,
in the proof above, we exploited the fact that
the distribution of
$\| S(t) \phi^\o \|_{W^s([t_0, t_0 + \tau_*])}$ depends
basically only on the length $\tau_*$ of the interval, but
is independent of  $t_0$.
\end{remark}

\section{Probabilistic global existence via randomization on dilated cubes}
\label{SEC:9}

In this section, we present the proof of Theorem \ref{THM:4}.
The main idea is to exploit the dilation symmetry of the cubic NLS \eqref{NLS1}.
For a function $\phi=\phi(x)$, we define its scaling $\phi_\mu$ by
\[\phi_\mu (x) := \mu^{-1} \phi(\mu^{-1}x),\]

\noi
while for a function $f=f(t,x)$,  we define its scaling $f_\mu$ by
\[ f_\mu(t,x):=\mu^{-1}f(\mu^{-2}t,\mu^{-1}x).\]

\noi
Then, given $\phi \in H^s(\R^d)$,
we have
\begin{align}
\|\phi_\mu\|_{\dot H^s(\R^d)}
= \mu^{\frac{d-2}{2} - s}\|\phi\|_{\dot H^s(\R^d)} .
\label{D1}
\end{align}

\noi
If $s < s_\textup{crit} = \frac{d-2}{2}$, that is, if $\phi$ is supercritical with respect to the scaling symmetry,  then
we can make
the $H^s$-norm of the scaled function $\phi_\mu$ small
by taking $\mu \ll 1$.
The issue is
 that the Strichartz estimates we employ in proving probabilistic well-posedness
 are (sub-)critical and do not become small even if we take $\mu \ll 1$.
It is for this reason that
we consider
the  randomization  $\phi^{\omega, \mu}$ on {\it dilated} cubes.

Fix  $\phi \in H^s(\R^d)$ with $s \in( s_d, s_\textup{crit})$, where $s_\textup{crit} = \frac{d-2}{2}$
and $s_d$ is as in \eqref{Sd1}.
Let $\phi^{\omega, \mu}$
be its randomization on dilated cubes of scale $\mu$ as in \eqref{R3}.
Instead of considering \eqref{NLS1} with $u_0 = \phi^{\omega, \mu}$,
we consider the scaled Cauchy problem:
\begin{equation}
\begin{cases}\label{NLS3}
i \partial_t u_\mu + \Delta u_\mu = \pm |u_\mu|^{2}u_\mu \\
u_\mu\big|_{t = 0} = u_{0, \mu} =  (\phi^{\omega, \mu})_\mu,
\end{cases}
\end{equation}

\noi
where $u_\mu$ is as in \eqref{scaling}
and
$(\phi^{\omega, \mu})_\mu(x) := \mu^{-1} \phi^{\omega, \mu}(\mu^{-1}x)$
is the scaled randomization.
For notational simplicity,
we denote
$(\phi^{\omega, \mu})_\mu$ by $\phi^{\omega, \mu}_\mu$ in the following.
Denoting the linear and nonlinear part of $u_\mu$ by
$z_\mu (t) = z^\o_\mu(t) : = S(t) \phi^{\o, \mu}_\mu$
and $v_\mu(t) := u_\mu(t) - S(t) \phi^{\o, \mu}_\mu$ as before,
we  reduce \eqref{NLS3} to
\begin{equation}
\begin{cases}
	 i \dt v_\mu + \Dl v_\mu = \pm |v_\mu + z_\mu|^2(v_\mu+z_\mu)\\
v_\mu|_{t = 0} = 0.
 \end{cases}
\label{NLS4}
\end{equation}

\noi
Note that if $u$ satisfies \eqref{NLS1} with initial data
$u(0)=\phi^{\o,\mu}$
then $u_\mu$, $z_\mu$, and $v_\mu$
are indeed the scalings of $u$, $z:=S(t)\phi^{\o, \mu}$,
and $v : = u -z$,  respectively.
For $u_\mu$ this simply follows from the scaling symmetry of  \eqref{NLS1}.
For $z_\mu$ and $v_\mu$,  this follows from the following observation:
\begin{align}\label{S_mu}
\mathcal F_x\Big[\big(S(t)\phi^{\o,\mu}\big)_\mu\Big](\xi)
= \mu^{d-1}e^{-i\frac{t}{\mu^2}|\mu\xi|^2} \ft{\phi^{\o,\mu}}(\mu \xi)
= e^{-it|\xi|^2}\ft{\phi^{\o, \mu}_\mu}( \xi)
=\widehat{z_\mu}(t, \xi).
\end{align}

Define $\G_\mu$ by
\begin{equation}
\G_\mu  v_\mu(t) =\mp i  \int_0^t S(t-t') \N (v_\mu+z_\mu )(t') dt'.
\label{NLS7}
\end{equation}

\noi
In the following,
we show that there exists $\mu_0 = \mu_0(\eps, \|\phi\|_{H^s} ) > 0$
such that, for $\mu \in (0, \mu_0)$,
 the estimates \eqref{nl1c} and \eqref{nl1d}
in Proposition \ref{PROP:NL1}
(with $\wt \G$ replaced by $\G_\mu$)
hold with $R = \eta_2$ outside a set of probability $< \eps$,
where  $\eta_2$ be as in \eqref{small0}.
In view of  \eqref{psi}, it is easy to see that
\begin{equation*}
\psi(D-n) \phi_\mu =\big( \psi^\mu(D-\mu n) \phi)_\mu.
\end{equation*}

\noi
Hence, we have
\begin{equation}
\phi^{\o, \mu}_\mu = (\phi^{\o, \mu})_\mu= \sum_{n \in \Z^d} g_n(\o) \psi(D-n) \phi_\mu.
\label{D2}
\end{equation}

\noi
Given $\eta_2$  as in \eqref{small0} and $\mu>0$,
define $\Omega_{1, \mu}$ by
\begin{equation*}
\Omega_{1, \mu} = \Big\{ \o \in \O:\,
\|S(t) \phi^{\omega, \mu}_\mu  \|_{L^q_t W^{s, q}_x ( \R \times \R^d)} \leq \eta_2, \,
q = 4, \tfrac{6(d+2)}{d+4},  d+2\Big\} .
\end{equation*}

\noi
We also define $\Omega_{2, \mu}$ by
\begin{equation*}
\Omega_{2, \mu} = \big\{ \o \in \O:\,
\| \phi^{\omega, \mu}_\mu  \|_{H^s(\R^d)} \leq \eta_2\big\}.
\end{equation*}

\noi
Now, let $\O_\mu = \Omega_{1, \mu} \cap \Omega_{2, \mu} $.
Noting that  $ 4,  \frac{6(d+2)}{d+4}$, and $d+2$
are larger than the diagonal Strichartz admissible index $\frac{2(d+2)}{d}$,
it follows from
Lemma \ref{PROP:Str2} and
Lemma \ref{LEM:Hs} with \eqref{D2}  and \eqref{D1} that
\begin{align*}
P( \O_{ \mu}^c)
 \leq C\exp\bigg( -c \frac{\eta_2^2}{ \|\phi_\mu\|_{H^s}^2}\bigg)
 \leq C\exp\bigg( -c \frac{\eta_2^2}{ \mu^{d-2 -2 s}\|\phi\|_{H^s}^2}\bigg)
\end{align*}

\noi
for $\mu \leq 1$.
Then, by
setting
\begin{equation}
 \mu_0 \sim
\Bigg(\frac{\eta_2}{\|\phi\|_{H^s} \big(\log \frac{1}{\eps}\big)^\frac{1}{2}}\Bigg)^\frac{1}{\frac{d-2}{2}-s},
\label{D2a}
\end{equation}

\noi
we have
\begin{align}
P( \O_\mu^c)
<\eps
\label{D3}
\end{align}

\noi
for $\mu \in (0,  \mu_0)$.
Note that $\mu_0 \to 0$ as $\eps \to 0$.
Recall that
$q = 4,  \frac{6(d+2)}{d+4}$, and $d+2$ are the only relevant values
of the space-time Lebesgue indices controlling the random forcing term
in the proof of Proposition \ref{PROP:NL1}.
Hence,  the estimates \eqref{nl1c} and \eqref{nl1d}
in Proposition \ref{PROP:NL1}
(with $\wt \G$ replaced by $\G_\mu$)
hold with $R = \eta_2$ for each $\o \in \O_\mu$.
Then, by repeating the proof of Theorem \ref{THM:2} in Section \ref{SEC:THM12},
we see that, for each $\o \in \O_\mu$,
there exists a  global solution
$u_\mu$  to \eqref{NLS3} with $u_\mu|_{t = 0} = \phi^{\o, \mu}_\mu$
which scatters both forward and backward in time.
By undoing the scaling, we obtain a global solution
$u$ to \eqref{NLS1} with $u|_{t = 0} = \phi^{\o, \mu}$
for each $\o \in \O_\mu$. Moreover, scattering for $u_\mu$
implies scattering for $u$. Indeed,
as in Theorem \ref{THM:2}, there exists $v_{+,\mu}\in H^{\frac{d-2}{2}}(\R^d)$
such that
\[\lim_{t\to\infty}\big\|u_\mu(t)-S(t)\big(\phi^{\omega,\mu}_\mu+v_{+,\mu}\big)\big\|_{H^{\frac{d-2}{2}}}=0.\]
Then, a computation analogous to \eqref{S_mu}
yields
\[S(t)(\phi^{\omega,\mu}_\mu+v_{+,\mu})
=\big(S(t)(\phi^{\omega,\mu}+v_+)\big)_{\mu},\]

\noi
where $v_+:=(v_{+,\mu})_{\mu^{-1}}\in H^{\frac{d-2}{2}}(\R^d)$.
Then, by \eqref{D1}, it follows that
\[\lim_{t\to \infty}\big\|u-S(t)\big(\phi^{\omega,\mu}+v_+\big)\big\|_{H^{\frac{d-2}{2}}}=0.\]

\noi
This proves that $u$ scatters forward in time.
Scattering of $u$ as $ t\to -\infty$ can be proved analogously.
This completes the proof of Theorem \ref{THM:4}.

\appendix

\section{On the properties of the $U^p$- and $X^s$-spaces}
\label{SEC:A}

In this appendix, we prove some additional properties
of the $U^p$- and $X^s$-spaces.
In the following,
all intervals are half open intervals of the form $[a, b)$
and $p$ denotes a number such that $1\leq p < \infty$.

\begin{lemma}\label{LEM:U0}
Let $u = \sum_{j = 1}^\infty \ld_j a_j$, where $\{\ld_j\}_{j = 1}^\infty \in \l^1 (\mathbb N; \mathbb C)$
and $a_j$'s are $U^p$-atoms.
Given an interval $I \subset \R$,
we can write  $u \cdot \chi_I $ as
$ u \cdot \chi_I = \sum_{j = 1}^\infty \wt \ld_j \wt a_j$
for some $\{\wt \ld_j\}_{j = 1}^\infty \in \l^1$
and some sequence $\{\wt a_j\}_{j = 1}^\infty$ of $U^p$-atoms
such that
\begin{equation}\label{U0}
\sum_{j = 1}^{\infty} |\wt \ld_j|
  \leq\sum_{j = 1}^{\infty} | \ld_j|.
\end{equation}
As a consequence, we have
\begin{equation}\label{bound_chi_I}
\|u\cdot \chi_I\|_{U^p(\R)}\leq \|u\|_{U^p(\R)}
\end{equation}
for any $u\in U^p(\R)$ and any $I\subset \R$.

\end{lemma}

\begin{proof}

With $a_j = \sum_{k = 1}^{K_j}  \phi_{k-1}^j\chi_{[t^j_{k-1}, t^j_k)}$,
we have
\begin{align*}
u \cdot \chi_I= \sum_{j = 1}^\infty \ld_j a_j \chi_I
= \sum_{j = 1}^\infty \ld_j   \sum_{k = 1}^{K_j}
\phi_{k-1}^j \chi_{[t^j_{k-1}, t^j_k)\cap I} .
\end{align*}

\noi
Then, setting $\wt \ld_j$ and $\wt a_j$ by
\begin{align}
\wt \ld_j & = \Big( \sum_{k \in A_j(I) }
\|\phi_k^j\|_{H}^p\Big)^\frac{1}{p} \ld_j ,\label{U0a}\\
 \wt a_j & =  \frac{1}{\big( \sum_{k \in A_j(I) }
\|\phi_k^j\|_{H}^p\big)^\frac{1}{p}}
\sum_{k \in A_j(I) }
\phi_{k-1}^j  \chi_{[t^j_{k-1}, t^j_k)\cap I} ,  \notag
 \end{align}

\noi
where $A_j(I)$ is defined by
\[A_j(I) = \big\{ k \in \{1, \dots, K_j\}: [t^j_{k-1}, t^j_k)\cap I \ne \emptyset\big\}, \]

\noi
we have
$u\cdot \chi_I =  \sum_{j = 1}^\infty \wt \ld_j \wt a_j$.
Moreover, noting that
\[  \Big( \sum_{k \in A_j(I) }
\|\phi_k^j\|_{H}^p\Big)^\frac{1}{p}
\leq
\Big( \sum_{k = 1}^{K_j}
\|\phi_k^j\|_{H}^p \Big)^\frac{1}{p} =   1,\]

\noi
we obtain \eqref{U0} from \eqref{U0a}.
Finally, by \eqref{U0}, we have
\[\|u\cdot\chi_I\|_{U^p(\R)}\leq \sum_{j=1}^\infty |\lambda_j|,\]

\noi
for any representation $u=\sum_{j=1}^\infty \ld_j a_j$
with $\{\ld_j\}_{j = 1}^\infty \in \l^1 (\mathbb N; \mathbb C)$
and  $U^p$-atoms $a_j$'s.
Hence,
by taking an infimum over all
such representations of $u$,
we obtain \eqref{bound_chi_I}.
\end{proof}

Given an interval $I \subset \R$, we define
the local-in-time $U^p$-norm  in the usual manner
as a restriction norm:
\[ \|u\|_{U^p(I)} = \inf_v\big\{ \|v\|_{U^p(\R)}:\, v|_I = u\big\}.\]

\begin{remark}\label{REM:U0}\rm
The infimum is achieved by $v = u \cdot \chi_I$  in view of Lemma \ref{LEM:U0}.
In the following, however,
we may use  other extensions,  depending on our purpose.
\end{remark}

The next lemma states the  subadditivity of the local-in-time
$U^p$-norm over intervals.

\begin{lemma}\label{LEM:U1}
Given an interval  $I \subset \R$,
let $I = \bigcup_{j = 1}^\infty I_j$ be a partition of $I$.
Then, we have
\begin{equation}
  \|u\|_{U^p(I) }\leq \sum_{j = 1}^\infty \|u\|_{U^p(I_j) }.
\label{U1}
  \end{equation}

\end{lemma}

\begin{proof}
Given $\eps > 0$,
it follows from the definition of the local-in-time $U^p$-norm that
there exists $v_j \in U^p(\R) $ such that
$v_j|_{I_j} = u$ and
\begin{equation}
	 \|v_j\|_{U^p(\R) } \leq  \|u\|_{U^p(I_j) } +\frac{\eps}{2^j}.
	 \label{U1a}
\end{equation}

\noi
for each $j \in \mathbb N$.
Then, by
\eqref{bound_chi_I} and  \eqref{U1a},
 we have
\begin{align}
 \|u\|_{U^p(I)}
& \leq \Big\|\sum_{j= 1}^\infty v_j \cdot \chi_{I_j} \Big\|_{U^p(\R)}
\leq \sum_{j= 1}^\infty \| v_j \cdot \chi_{I_j} \|_{U^p(\R)}
 \leq \sum_{j = 1}^\infty \| v_j \|_{U^p(\R)} \notag \\
&   \leq \sum_{j = 1}^\infty  \|u\|_{U^p(I_j)}  + \eps.
\label{U1d}
\end{align}

\noi
Since $\eps  >0$ is arbitrary,
\eqref{U1} follows from \eqref{U1d}.
\end{proof}

As a corollary, we immediately obtain the following subadditivity
property
of the local-in-time $X^s$-norm
over intervals.

\begin{lemma}\label{LEM:U2}
Let $s \in \R$.
Given an interval  $I \subset \R$,
let $I = \bigcup_{j = 1}^\infty I_j$ be a  partition of $I$.
Then, we have
\begin{equation*}
  \|u\|_{X^s(I) }\leq \sum_{j = 1}^\infty \|u\|_{X^s(I_j) }.
  \end{equation*}

\end{lemma}

We say that $u$ on
$[a, b)$
is a regulated function if both left and right limits exist at every point
(including one-sided limits at the endpoints\footnote{We allow $a = -\infty$ and/or $b = \infty$.}).
Given a regulated function $u$ on $[a, b)$
and a partition $\mathcal P = \{ \tau_1, \dots, \tau_n\}$ of $[a, b)$:
$a < \tau_1 < \cdots < \tau_n < b$,
we defined a step function $u_{\mathcal P}$ by
\begin{align*}
u_\mathcal{P}(t) =
\begin{cases}
u(t), & \text{if } t = \tau_j,\\
u(\tau_j+), & \text{if } \tau_{j} < t < \tau_{j+1},\\
\end{cases}
\end{align*}

\noi
where we set $\tau_0 = a$ and $\tau_{n+1} = b$.
In particular, if $u$ is right-continuous, we have
$u_\mathcal{P}(t) =
u(\tau_j)$ for $ \tau_{j} \leq t < \tau_{j+1}$.
Note that the mapping $\mathcal P: u \mapsto u_{\mathcal P}$ is linear.

\begin{lemma}\label{LEM:U3}
Let $u \in U^p(\R)$.

\noi
\textup{(i)}
\noi
For any partition $\mathcal P$ of $\R$,
 we have
\begin{equation}
 \| u_{\mathcal P}\|_{U^p(\R)} \leq \|u\|_{U^p(\R)}.
\label{U3a}
 \end{equation}

\smallskip

\noi
\textup{(ii)}
Given $\eps > 0$, there exists a partition $\mathcal P$ of $\R$ such that
\begin{equation}
 \| u - u_{\mathcal P}\|_{U^p(\R)} < \eps.
\label{U3aa}
 \end{equation}

\end{lemma}
	
\begin{proof}
(i)
We first claim that, given
 a $U^p$-atom $a$,
we have
 $\|a_{\mathcal P}\|_{U^p(\R)} \leq 1$
for any partition $\mathcal P$.
Given a $U^p$-atom
$ a = \sum_{k = 1}^K\phi_{k-1} \chi_{[t_{k-1}, t_k)}$
and a partition $\mathcal P = \{ \tau_1, \dots, \tau_n\}$ of $\R$, we have
\begin{align}
a_\mathcal{P}  = \sum_{j = 1}^n a(\tau_j) \chi_{[\tau_j, \tau_{j+1})},
\label{U3b}
\end{align}
	
\noi
where $\tau_{n+1} = \infty$.
Note that we have
\[ a(\tau_j) = \begin{cases}
\phi_{k-1}, &
\text{if }\tau_j \in [t_{k-1}, t_k) \text{ for some }k, \\
0,  & \text{otherwise.}
\end{cases}
\]

\noi
We can simplify the expression in \eqref{U3b}
by concatenating
neighboring intervals
$[\tau_j, \tau_{j+1})$ and $[\tau_{j+1}, \tau_{j+2})$
if $a(\tau_j) = a(\tau_{j+1})$
and obtain
\begin{align*}
a_\mathcal{P}  = \sum_{\l = 1}^L a(\tau_{j_\l} ) \chi_{[\tau_{j_\l}, \tau_{j_{\l+1}})}
\end{align*}

\noi
for some subpartition $\{ \tau_{j_\l}\}_{\l = 1}^L$ of $\mathcal{P}$,
where $a(\tau_{j_\l}) = \phi_{k-1}$ for some $k$ or $a(\tau_{j_\l}) = 0$.
Note that, given $k\in \{1, \dots, K\}$, there exists at most one $\l \in \{1, \dots, L\}$
such that $a(\tau_{j_\l}) = \phi_{k-1}$ (unless $\phi_{k-1} = \phi_{k'-1}$ for some $k' \ne k$).
In particular, we have
\[\ld:= \Big(\sum_{\l = 1}^L \|a(\tau_{j_\l})\|_H^p\Big)^\frac{1}{p}
\leq \Big(\sum_{k = 1}^K \|\phi_{k-1}\|_H^p\Big)^\frac{1}{p}
= 1.\]

\noi
If $\ld = 0$, then $a_{\mathcal P} = 0$.
Otherwise, we have $a_{\mathcal P} = \ld b$,
where $b$ is a $U^p$-atom given by
\[b = \sum_{\l = 1}^L \frac{a(\tau_{j_\l} )}{\ld} \chi_{[\tau_{j_\l}, \tau_{j_{\l+1}})}.\]

\noi
Hence, $\|a_{\mathcal P}\|_{U^p(\R)} \leq 1$.

Given $u \in U^p(\R)$, write
$u = \sum_{j = 1}^\infty\ld_j a_j$ for some $\{\ld_j\}_{j = 1}^\infty \in \l^1$
and some sequence of $\{a_j\}_{j = 1}^\infty$ of $U^p$-atoms.
Then, we have
\begin{align}
\| u_{\mathcal P}\|_{U^p(\R)}
= \bigg\|\sum_{j = 1}^\infty \ld_j \cdot (a_j)_{\mathcal P}\bigg\|_{U^p(\R)}
\leq \sum_{j = 1}^\infty |\ld_j| \|(a_j)_{\mathcal P}\|_{U^p(\R)}
\leq \sum_{j = 1}^\infty |\ld_j| .
\label{U3c}
\end{align}

\noi
Therefore, we obtain \eqref{U3a},
since \eqref{U3c}  holds for any
$\{\ld_j\}_{j = 1}^\infty \in \l^1$
and any sequence of $\{a_j\}_{j = 1}^\infty$ of $U^p$-atoms
such that
$u = \sum_{j = 1}^\infty \ld_j a_j$.

\smallskip

\noi
(ii)
Fix a representation
 $u = \sum_{j=1}^\infty \ld_j a_{j}$
for some $\{\ld_j\}_{j = 1}^\infty\in\ell^1$
and some  sequence $\{a_j\}_{j = 1}^\infty$ of $U^p$-atoms.
Then, by setting $u_J = \sum_{j=1}^J \ld_j a_{j}$ for sufficiently large $J$, we have
\begin{equation}
 \|u - u_J\|_{U^p(\R)} \leq \sum_{j = J+1}^\infty |\ld_j| < \frac{\eps}{2}.
\label{U3d}
 \end{equation}

\noi
Note that $u_J$ is a step function with finitely many jump discontinuities.
Now, we define a partition $\mathcal P$
by setting $\mathcal P = \{ t \in \R: \, u_J \text{ is discontinuous at } t\}$.
Then, by right-continuity of $u_J$, we have
$u_J - (u_J)_\mathcal{P} = 0$.
Hence, from \eqref{U3d} and part (i), we obtain
\begin{align*}
\| u - u_{\mathcal{P}}\|_{U^p(\R)}
& \leq
\| u - u_J\|_{U^p(\R)}
+\| u_J  -( u_J)_{\mathcal{P}}\|_{U^p(\R)}
+
\|( u - u_J)_{\mathcal{P}}\|_{U^p(\R)}\\
& <  \frac{\eps}{2} + 0 + \frac{\eps}{2} = \eps.
\end{align*}
	
\noi
Note that any refinement $\mathcal P'$ of the partition $\mathcal P$
also yields \eqref{U3aa}.
\end{proof}

\begin{lemma}\label{LEM:U4}
Let $I = [a, b)\subset \R$ be an interval.
Given $u \in U^p(I) \cap C(I; H)$,
the mapping $t \in I \mapsto \|u\|_{U^p([a, t))}$ is continuous.
\end{lemma}
	
\begin{remark}\rm
It follows from the proof that we need the (left-)continuity of $u$
only in proving left-continuity of the mapping
$t \in I \mapsto \|u\|_{U^p([a, t))}$.
\end{remark}

\begin{proof}

\noi
$\bullet$  {\bf Part 1:} Left-continuity.

Suppose that the mapping $t \mapsto \|u\|_{U^p([a, t))}$ is not left-continuous at $t =t_* \in (a, b)$.
Then,  there exist $\eps > 0$
and a sequence $\delta_n\in (0,t_\ast-a)$, $\delta_n\to 0$ as $n\to\infty$,
such that
\begin{equation}
 \|u \|_{U^p([a, t_* - \dl_n))} <  \|u \|_{U^p([a, t_* ))} - \eps.
\label{U4a0}
 \end{equation}

By definition, for any $\delta\in [0, t_\ast-a)$, there exists $v_\dl \in U^p (\R)$ such that
$v_\dl |_{[a, t_* - \dl)} = u$ and
\begin{equation}
 \|v_\dl\|_{U^p(\R) }\leq \|u \|_{U^p([a, t_* - \dl))} + \frac{\eps}{4}.
\label{U4a1}
 \end{equation}

\noi
Moreover, in view of Lemma \ref{LEM:U0},
we can assume that $v_\dl = v_\dl \cdot\chi_{[a, t_*-\dl)}$.
In particular, we have
\begin{align}
v_\dl  = \sum_{j = 1}^\infty \ld_j^\dl  a_j^\dl
= \sum_{j = 1}^\infty \ld_j^\dl    \sum_{k = 1}^{K_j^\dl }
\phi_{k-1}^{\dl, j} \chi_{[t^{\dl, j}_{k-1}, t^{\dl, j}_k)} ,
\label{U4a}
\end{align}

\noi
where $t_{K_j^\dl}^{\dl, j} \leq t_*-\dl$.
Now, we define an extension $\wt v_\dl$ of
$v_\dl $ onto $\R$
by setting  $t_{K_j^\dl}^{\dl, j} = \infty$ in \eqref{U4a} if $t_{K_j^\dl}^{\dl, j} = t_* - \dl$.
By continuity of $u$ and $v_\dl  |_{[a, t_* - \dl)} = u$,
we have $\wt v_\dl (t) = u(t_*-\dl)$ for $t  \in [ t_*-\dl, \infty)$.
By construction, we have
\begin{equation}
\|v_\dl\|_{U^p(\R)} = \|\wt v_\dl\|_{U^p(\R)}.
\label{U4a2}
\end{equation}

Let $\wt v$ be the extension of $u \cdot \chi_{[a, t_*)}$
constructed as above with $\dl = 0$.
Then, by definition of the $U^p$-norm and Lemma \ref{LEM:U3} (ii),
there exists a partition $\mathcal P$ of $\R$ such that
\begin{align}
\| u \|_{U^p([a, t_*))} \leq \|\wt v\|_{U^p(\R)}
\leq \|\wt v_{\mathcal P}\|_{U^p(\R)} + \frac{\eps}{8}.
\label{U4b}
\end{align}

\noi
Since \eqref{U4b} holds for any refinement
$\mathcal P'$ of $\mathcal P$, we can assume that $t_* \in \mathcal P$.

By uniform continuity of $u$, there exists $\dl_0 > 0$ such that 	
\begin{align}
\| u(t_1) - u(t_2) \|_H < \frac{\eps}{8 \cdot (\#\mathcal P+1)}
\label{U4c}
\end{align}

\noi
for any $t_1, t_2 \in (t_* - \dl_0, t_*]$.	
Since $\wt v_\dl = \wt v$ on $(-\infty, t_*-\dl]$,
$\wt v(t) = u(t_*)$ for $t \geq t_*$,
and $\wt v_\delta(t)=u(t_\ast-\delta)$ for $t\geq t_\ast-\delta$,
we have
\begin{align*}
 \wt v_{\mathcal P} - (\wt v_\dl)_{\mathcal P}
& = \sum_{\substack{\tau_j \in \mathcal P\\
\tau_j > t_* - \dl}} \big( u (\tau_j) - u (t_* - \dl)\big)\chi_{[\tau_j, \tau_{j+1})}\\
& = \sum_{\substack{\tau_j \in \mathcal P\\
\tau_j \in ( t_* - \dl, t_*)}} \big( u (\tau_j) - u (t_* - \dl)\big)\chi_{[\tau_j, \tau_{j+1})}
+ \big(u(t_*) - u(t_* - \dl) \big)\chi_{[t_*, \infty)}.
\end{align*}

\noi
Then, from \eqref{U4c}, we have
\begin{align}
\| \wt v_{\mathcal P} - (\wt v_\dl)_{\mathcal P}\|_{U^p(\R)}
& \leq  \sum_{\substack{\tau_j \in \mathcal P\\
\tau_j \in ( t_* - \dl, t_*)}} \| u (\tau_j) - u (t_* - \dl)\|_{H}
+ \| u(t_*) - u(t_* - \dl)\|_H \notag \\
& < \frac{\eps}{8} \label{U4d}
\end{align}

\noi
 for any $\dl \in (0, \dl_0)$.

Finally, from \eqref{U4b}, \eqref{U4d}, Lemma \ref{LEM:U3} (i),
\eqref{U4a2}, \eqref{U4a1}, and \eqref{U4a0},
we have
\begin{align*}
\| u \|_{U^p([a, t_*))}
& \leq \|(\wt v_{\dl_n})_{\mathcal P}\|_{U^p(\R)} + \frac{\eps}{4}
\leq \|\wt v_{\dl_n}\|_{U^p(\R)} + \frac{\eps}{4}
 \leq \|u \|_{U^p([a, t_* - \dl_n))} + \frac{\eps}{2}\\
& \leq \|u \|_{U^p([a, t_*))} - \frac{\eps}{2}
\end{align*}

\noi
 for  sufficiently large $n$ such that $\delta_n<\delta_0$.
This is a contradiction.
Therefore,
the mapping $t \mapsto \|u\|_{U^p([a, t))}$ is  left-continuous at $t =t_*$.

\smallskip

\noi
$\bullet$ {\bf Part 2:} Right-continuity.

Fix $t_* \in I$ and small $\eps > 0$.
As in Part 1,
let $\wt v_0$ be the extension
of $v_0 = v_0 \cdot \chi_{[a, t_*)}$ satisfying \eqref{U4a1}.
In particular, from \eqref{U4a1} and \eqref{U4a2}, we have
\begin{align}
\|\wt v_0\|_{U^p(\R)} \leq \|u \|_{U^p([a, t_*))} + \frac{\eps}{4}.
\label{U5b}
\end{align}

\noi
Note that $\wt v_0 = 0$ on $(-\infty, a)$.

Let $w = \wt u - \wt v_0$, where $\wt u \in U^p(\R) $ is an extension of $u$ from $I$ onto $\R$
such that $\wt u = 0$ on $(-\infty, a)$.
Since $w \in U^p(\R)$, we can write $w = \sum_{j = 1}^\infty \ld_j a_j$
for some $\{\ld_j\}_{j = 1}^\infty \in \l^1$
and some sequence  $\{a_j\}_{j = 1}^\infty$ of $U^p$-atoms.
Since $w = \wt u - \wt v_0 = 0$ on $(-\infty, t_*]$,
we can assume
 that $\supp (a_j) \subset (t_*,  \infty)$ for all $j$.
Then, we can choose large $J = J(\eps) \in \mathbb{N}$ such that
\begin{align}
\sum_{j = J+1}^\infty |\ld_j| < \frac{\eps}{4}.
\label{U5c}
\end{align}
	
\noi	
Noting that
$w_J := \sum_{j = 1}^J \ld_j a_j$ is a finite linear combination of
characteristic functions, there exists $\dl_0>0$ such that
$w_J$ is constant on $[t_*, t_*+\dl_0) \subset I$.
Define $\ld_0$, $\phi$, and $a_0$ by
\[\ld_0 := \|w_J(t_*)\|_H, \quad
\phi := \ld_0^{-1} w_J(t_*), \quad
\text{and} \quad a_0  :=  \chi_{[t_*, \infty)} \phi .\]

\noi
Then, define $\wt w$ by
$\wt w(t)  :=w_{J}(t_\ast)\chi_{[t_\ast,\infty)}+ \sum_{j = J+1}^\infty \ld_j a_j= \ld_0 a_0 + \sum_{j = J+1}^\infty \ld_j a_j$.
Note that $\wt w = 0$ on $(-\infty, t_*]$.
it follows from \eqref{U5c} that
\begin{align}
\| \wt w\|_{U^p(\R)} \leq
|\ld_0| + \sum_{j = J+1}^\infty |\ld_j|
< \frac{ \eps}{2},
\label{U5d}
\end{align}

\noi
since we have
\[ \ld_0 = \bigg\| w(t_*) - \sum_{j = J+1}^\infty \ld_j a_j(t_*)\bigg\|_H
=\bigg\|\sum_{j = J+1}^\infty \ld_j a_j(t_*)\bigg\|_H
\leq \sum_{j = J+1}^\infty |\ld_j| < \frac{\eps}{4}.
\]

\noi
Here, we used the fact that   $\|a_j\|_{H}\leq 1$ for a $U^p$-atom $a$.
By construction, we have $\supp (\wt w) \subset [t_*, \infty)$.
Then, noting that
$u - \wt v_0 = \wt w$
on $(-\infty, t_*+\dl_0) \subset I$,
it follows from \eqref{U5b} and  \eqref{U5d} that
\begin{align*}
\| u \|_{U^p([a, t_*+\dl))}
& \leq
 \|  \wt v_0\|_{U^p([a, t_*+\dl))}
+ \|\wt w \|_{U^p([a, t_*+\dl))}
  \leq \|  \wt v_0\|_{U^p(\R)} + \| \wt w \|_{U^p(\R)}
  \notag \\
& \leq \| u \|_{U^p([a, t_*))}
+\frac 34 \eps
\end{align*}
	
\noi
for any $\dl \in (0, \dl_0]$.
Therefore,
the mapping $t \mapsto \|u\|_{U^p([a, t))}$ is  right-continuous at $t =t_*$.
\end{proof}

\begin{lemma}\label{LEM:U6}
Let $s \in \R$ and  $I = [a, b) \subset \R$.
Given $u \in X^s(I) \cap C(I; H^s(\R^d))$,
the mapping $t \in I \mapsto \|u\|_{X^s([a, t))}$ is continuous.
\end{lemma}

\begin{proof}
First, we claim that the infimum in the definition of
the local-in-time $X^s$-norm on an interval $[a, t)$ is achieved by $u \cdot \chi_{[a, t)}$
for any $ t \leq b$. Namely, we have
\begin{align}
\|u \|_{X^s([a, t))} = \|u\cdot \chi_{[a, t)}\|_{X^s(\R)}.
\label{UP0}
\end{align}

\noi	
On the one hand, given any extension $v$ on $\R$ of $u$ restricted to $[a, t)$,
i.e.~ $v|_{[a, t)} = u$, we have
\begin{align}
\|u \|_{X^s([a, t))}
\leq \| v\|_{X^s(\R)}.
\label{UP1}
\end{align}

\noi
On the other hand, by Lemma \ref{LEM:U0}, we have
\begin{align}
\|u\cdot \chi_{[a, t)} \|_{X^s(\R)}
= \|v\cdot \chi_{[a, t)} \|_{X^s(\R)}
\leq \| v\|_{X^s(\R)}.
\label{UP2}
\end{align}

\noi
Hence, \eqref{UP0} follows, since  \eqref{UP1} and  \eqref{UP2} hold
for any extension $v$.
Moreover, we have
\begin{align}
\|u \|_{X^s([a, t))}
& = \|u\cdot \chi_{[a, t)}\|_{X^s(\R)}
= \bigg(\sum_{\substack{N\geq 1\\ \text{dyadic}}}
N^{2s} \|\P_N u \cdot \chi_{[a, t)} \|^2_{U^2_\Dl(\R; L^2)}\bigg)^\frac{1}{2} \notag \\
& = \bigg(\sum_{\substack{N\geq 1\\ \text{dyadic}}}
N^{2s} \|\P_N u \|^2_{U^2_\Dl([a, t); L^2)}\bigg)^\frac{1}{2},
\label{UP3}
\end{align}

\noi
where the last equality follows from Remark \ref{REM:U0}.

Let $v$ be an extension of $u$ onto $\R$
such that $ \| v\|_{X^s(\R)} < \infty.$
Given $\eps > 0$, we can choose $J \in \mathbb{N}$ such that
\begin{align*}
 \bigg(\sum_{j = J}^\infty
2^{2js} \|\P_{2^j} v \|^2_{U^2_\Dl(\R; L^2)}\bigg)^\frac{1}{2}<\frac{\eps}{4}.
\end{align*}

\noi
Then, we have
\begin{align}
 \bigg(\sum_{j = J}^\infty
2^{2js} \|\P_{2^j} u \|^2_{U^2_\Dl([a, t); L^2)}\bigg)^\frac{1}{2}<\frac{\eps}{4}
\label{UP4}
\end{align}
	
\noi
for any $t \in I$.
Fix $t_* \in I$.
By Lemma \ref{LEM:U4},
for each $j = 0, 1, \dots, J-1$,
there exists $\dl_j >0$
such that
\begin{align}
2^{js}\Big|
\|\P_{2^j} u \|_{U^2_\Dl([a, t_*); L^2)}
- \|\P_{2^j} u \|_{U^2_\Dl([a, t_*+\dl); L^2)}
 \Big|
< \frac \eps {2 J^\frac{1}{2}}
\label{UP5}
\end{align}
	
\noi
for  $|\dl| < \dl_j$.
Then, by Minkowski's inequality with \eqref{UP3},  \eqref{UP4},  and \eqref{UP5}, we have
\begin{align*}
\Big|\|  u \|_{X^s([a, t_*))} & - \|   u  \|_{X^s([a, t_*+\dl ))} \Big|
\leq  \bigg(\sum_{j = 0}^\infty
2^{2js} \big(\|\P_{2^j} u \|_{U^2_\Dl([a, t_*); L^2)}
- \|\P_{2^j} u \|_{U^2_\Dl([a, t_*+\dl); L^2)}\big)^2 \bigg)^\frac{1}{2}\\
& \leq
J^\frac{1}{2} \max_{j = 0, \dots, J-1}
2^{js}\Big|
\|\P_{2^j} u \|_{U^2_\Dl([a, t_*); L^2)}
- \|\P_{2^j} u \|_{U^2_\Dl([a, t_*+\dl); L^2)}
 \Big| +\frac{\eps}{2}\\
& < \eps
\end{align*}

\noi
for  $0<|\dl| < \min (\dl_0, \dots, \dl_{J-1})$.
This proves the lemma.
\end{proof}

\begin{acknowledgment}

\rm 
This work was partially supported by a grant from the Simons Foundation (No.~246024 to \'Arp\'ad B\'enyi).
T.O.~was supported by the European Research Council (grant no.~637995 ``ProbDynDispEq'').
O.P.~was supported by the NSF grant under agreement No.~DMS-1128155.
Any opinions, findings, and conclusions or recommendations expressed in this material
are those of the authors and do not necessarily reflect the views of the NSF

\end{acknowledgment}

\end{document}